\documentclass[12pt,reqno,a4paper, english]{amsart}

\usepackage{changepage}

\usepackage[top=3.5cm,bottom=4cm,left=3cm,right=3cm]{geometry}

\usepackage[mathscr]{euscript}

\usepackage[utf8]{inputenc}
\usepackage[main=british, french, german, latin]{babel}

\usepackage{xcolor}
\usepackage{tikz} 
\usepackage{tikz-3dplot}
\usetikzlibrary{shapes.geometric, arrows, positioning} 
\tikzstyle{startstop} = [rectangle, rounded corners, minimum width=3cm, minimum height=1cm, text centered, draw=black, fill=red!30]
\tikzstyle{io} = [trapezium, trapezium left angle=70, trapezium right angle=110, minimum width=3cm, minimum height=1cm, text centered, draw=black, fill=blue!30]
\tikzstyle{process} = [rectangle, minimum width=3cm, minimum height=1cm, text centered, draw=black, fill=orange!30]
\tikzstyle{arrow} = [thick,->,>=stealth]

\usepackage{amsmath}
\usepackage{cases}
\usepackage{microtype}
\usepackage{amssymb}
\usepackage{hyperref}
\hypersetup{colorlinks={true},linkcolor={blue},citecolor=blue}

\usepackage{mathrsfs}

\newcommand{\noi}{\noindent}

\newcommand{\hexagonal}[1]{\mathrel{\textcircled{\scriptsize$#1$}}}

\theoremstyle{plain}

\newtheorem{theorem}{Theorem}
\newtheorem{proposition}[theorem]{Proposition}
\newtheorem{lemme}[theorem]{Lemma}

\newtheorem{corollaire}[theorem]{Corollary}
\newtheorem{definition}[theorem]{Definition}
\newtheorem{remarque}[theorem]{Remark}

\numberwithin{equation}{section}
\numberwithin{theorem}{section}

\def\Im{\textrm{Im}}
\def\Re{\textrm{Re}}
\def\ov{\overline} 
\def\11{{\rm 1~\hspace{-1.4ex}l} }
\def\R{\mathbb R}
\def\C{\mathbb C}

\def\N{\mathbb N}
\def\E{\mathbb E}
\def\T{\mathbb T}

\def\S{\mathbb S}

\def\lg{\langle}
\def\rg{\rangle}

\def\wh{\widehat}

\def\e{\mathrm{e}}
\def\d{\mathrm{d}  }

\makeatletter
\@namedef{subjclassname@2020}{\textup{2020} Mathematics Subject Classification}
\makeatother

\begin{document}

\title[Probabilistic well-posedeness for NLS on the $2d$ sphere I]{Probabilistic well-posedeness for the nonlinear Schr{\"o}dinger equation on the $2d$ sphere I: positive regularities}

\author{Nicolas Burq}
\address{Universit\'e Paris-Saclay, Laboratoire de Math\'ematique d’Orsay, UMR CNRS 8628, Orsay,
France, and Institut universitaire de France }
\email{nicolas.burq@universite-paris-saclay.fr}

\author{Nicolas Camps}
\address{Univ Rennes, IRMAR - UMR CNRS 6625, F-35000 Rennes, France}
\email{nicolas.camps@univ-rennes.fr}

\author{Chenmin Sun}
\address{CNRS, Universit\'e Paris-Est Cr\'eteil,  Laboratoire d'Analyse et de Math\'ematiques Appliqu\'ees, UMR CNRS 8050, Cr\'eteil, France}
\email{chenmin.sun@cnrs.fr}

\author{Nikolay Tzvetkov}
\address{Ecole Normale Sup\'erieure de Lyon, Unit\'e de Math\'ematiques Pures et Appliqu\'es,  UMR CNRS 5669, Lyon, France}
\email{nikolay.tzvetkov@ens-lyon.fr}



\date{\today}

\begin{abstract} 
We establish the probabilistic well-posedness of the nonlinear Schr\"odinger equation on the $2d$ sphere $\mathbb{S}^{2}$. The initial data are distributed according to  Gaussian measures with typical regularity $H^{s}(\mathbb{S}^{2})$, for $s>0$. This level of regularity goes significantly beyond existing deterministic results, in a regime where the flow map cannot be extended uniformly continuously. 
\end{abstract}

\ \vskip -1cm  \hrule \vskip 1cm \vspace{-8pt}
 \maketitle 
{ \textwidth=4cm \hrule}

\maketitle
\setcounter{tocdepth}{1}
\tableofcontents

\section{Introduction}

\subsection{Context}
The study of the nonlinear Schr\"odinger equation (NLS) on general Riemannian manifolds was initiated in \cite{BGT}, following a series of works by Bourgain~\cite{B1} on the flat torus.
It was shown in \cite{BGT} that in the cases of Riemannian surfaces the Cauchy problem for the cubic NLS is locally well-posed in the Sobolev space $H^s$ with $s>\frac{1}{2}$. This result is of interest because the classical methods, based solely on Sobolev embeddings, yield well-posedness under the considerably stronger restriction $s>1$ (one half of the dimension). In the case of the sphere ${\mathbb{S}^2}$ the restriction $s>\frac{1}{2}$ was relaxed to $s>\frac{1}{4}$ in \cite{BGT2}. 
As proved in \cite{BGT_MRL}, the restriction $s>\frac{1}{4}$ is in a sense an optimal limit, given that it corresponds to the threshold of semi-linear well-posedness methods. 
The aim of this paper is to show that, in the case of the $2d$ sphere $\mathbb{S}^{2}$, it is possible to go beyond the $s>\frac{1}{4}$ threshold by randomizing the initial data in a Sobolev space of low regularity. This is in the spirit of the program initiated in~\cite{BT1/2}, which aims to study dispersive partial differential equations with random data (not necessarily related to an invariant measure) beyond the deterministic thresholds, where some instabilities are known to occur.
\subsection{Setup and main results}
Let $\mathbb{S}^2$ be the unit sphere in $\R^3$, endowed with the canonical metric, where we normalize the measure on $\mathbb{S}^2$ to adopt the convention $\int_{\mathbb{S}^2}1=1.$ The cubic NLS, posed on  $\mathbb{S}^2$, we consider in this article is\footnote{In this work, the Wick-ordered nonlinearity is only favorable when estimating the regularity for purely random objects. } 
\begin{equation}\label{NLScubic} 
i\partial_t u-(-\Delta+1) u=(|u|^2-2\|u\|_{L^2}^2)u,
\end{equation}
where $u\colon \R_t\times\mathbb{S}^2\rightarrow \C$ and $\Delta$ is the Laplace-Beltrami operator on $\mathbb{S}^2$. 
It is well-known that $-\Delta+1$ is a self adjoint operator on  $L^2(\mathbb{S}^2)$ with domain the Sobolev space $H^2(\S^2)$ and that its spectrum is discrete with eigenvalues $\lambda_n^2$ given by
	\[
\lambda_n^2=n^2+n+1,\quad n=0,1,2,\cdots
	\]
Moreover, $\lambda_n^2$ has multiplicity $2n+1$ and the corresponding eigenspace consists of the spherical harmonics of degree $n$ (i.e. the restriction of the harmonic 
polynomials of degree $n$ on $\R^3$ to $\S^2$). We denote by $E_{n}$ the eigenspace of $-\Delta+1$ associated with the eigenvalue $\lambda_n^2$.
\medskip

If $u$ is a solution of \eqref{NLScubic} then $v=\e^{it-2it\|u\|_{L^2}^2}u$ is a solution of
\begin{equation}\label{NLScubic_pak} 
i\partial_t v+\Delta v=|v|^2v
\end{equation}
which is the traditionally studied NLS. For a sake of convenience, in this article we will restrict to \eqref{NLScubic} but all results we obtain can be easily rephrased in the context of \eqref{NLScubic_pak}  thanks to the straightforward link between the two equations. 
\medskip

We denote by  $\lg f|g\rg:=\int_{\mathbb{S}^2}f\overline{g}$ the scalar product in $L^2(\S^2)$. For $n\geq 0$, we denote by $\pi_n$ the orthogonal projection of $L^2(\S^2)$ on $E_{n}$.  More precisely, if $(\mathbf{b}_{n,k})_{|k|\leq n}$ is an orthonormal basis of $E_{n}$ then
	\[ \pi_n(f)=\sum_{|k|\leq n}\lg f\;|
\mathbf{b}_{n,k}
\rg \mathbf{b}_{n,k},\quad \forall\, f\in L^2(\S^2)\,.
	\]
The $L^p$ mapping properties of $\pi_n$ and the precise concentration properties of its kernel will play an important role in the analysis below.  Since $\pi_n$ is invariant under the action of the rotation group $SO(3)$, the kernel of $\pi_n$ is invariant under the rotation,  hence  $\sum_{|k|\leq n}|\mathbf{b}_{n,k}(x)|^2$ is independent of $x\in\mathbb{S}^2$ (see e.g. \cite{BL}). Therefore 
\begin{align}\label{Weyl}
\sum_{|k|\leq n}|\mathbf{b}_{n,k}(x)|^2=\sum_{|k|\leq n}\|\mathbf{b}_{n,k}\|_{L^2(\S^2)}^2= 2n+1.
\end{align} 
If $f$ is decomposed as
$
f=\sum_{n\geq 0} \pi_n(f)
$
then the Sobolev norm $H^s(\S^2)$  of $f$ is equivalent to
\begin{equation}\label{Sob}
\sum_{n\geq 0}  \lambda_n^{2s}\|\pi_n f\|_{L^2}^2\,. 
\end{equation}
We next define the random initial data we will consider.  Fix a probability space $(\Omega,\mathcal{F},\mathbb{P})$ and a sequence of i.i.d. complex standard Gaussian random variables $(g_{n,k}(\omega))_{n\in\N,|k|\leq n}$ on it, i.e.
	\[ 
g_{n,k}(\omega)=\frac{\Re(g_{n,k})(\omega)+i\,\Im(g_{n,k})(\omega) }{\sqrt{2}}\,,
	\]
where $\Re(g_{n,k}),\Im(g_{n,k})$ are independent real-valued normal random variables on $(\Omega,\mathcal{F},\mathbb{P})$.
\medskip

Let again $(\mathbf{b}_{n,k})_{|k|\leq n}$ be an orthonormal basis of $E_{n}$. For $\alpha\in\R$, we set
\begin{equation}\label{phi_alpha}
 \phi_\alpha(x,\omega)=\sum_{n\geq 0}\frac{1}{\lambda_n^\alpha}\,\sum_{|k|\leq n}g_{n,k}(\omega)\mathbf{b}_{n,k}(x).
\end{equation}
Thanks to \eqref{Weyl} and \eqref{Sob}, we can easily compute the typical Sobolev regularity of $ \phi_\alpha(x,\omega)$. Namely
	\[
\E
\|
\phi_\alpha(\cdot,\omega)
\|_{H^s}^2
=
\sum_{n\geq 0}\lambda_n^{2s-2\alpha}(2n+1)\,,
	\]
which is finite if and only if $s<\alpha-1$. Therefore $\phi_\alpha(\cdot,\omega)$ belongs almost surely to the Sobolev space $H^s(\S^2)$ for $s<\alpha-1$. It can also be shown (see e.g. \cite[Lemma B.1]{BT1/2})  that   $\phi_\alpha(\cdot,\omega)$  almost surely does not belong  to $H^{\alpha-1}(\S^2)$.  We also observe that thanks to the  invariance of the Gaussians under rotations, we have that the law of  $\omega\mapsto  \phi_\alpha(\cdot,\omega)$ (as a random variable on $H^s(\S^2)$ for some $s<\alpha-1$) is independent of the choice of the basis $(\mathbf{b}_{n,k})_{|k|\leq n}$, and is denoted by $\mu_{\alpha}$.
\medskip

As detailed in \cite[Section 3.1]{BCLST} we can re-organize the sum \eqref{phi_alpha}:
\[
\phi_{\alpha}^{\omega}(x) = \sum_{n\geq0} \frac{1}{\widetilde{\lambda}_{n}^{(\alpha-\frac{1}{2})}}e_{n}^{\omega}(x)\,,
\]
where 
\[
e_{n}^{\omega}(x) = \frac{1}{\sqrt{2n+1}}\sum_{|k|\leq n}g_{n,k}(\omega)\mathbf{b}_{n,k}(x)\,,\quad \widetilde{\lambda}_{n} = \lambda_{n}\big(\frac{2n+1}{\lambda_{n}}\big)^{\frac{1}{2}}\,.
\]
Up to a factor 2, $\lambda_{n}$ and $\widetilde{\lambda}_{n}$ have the same asymptotic so, abusing notations, we keep writing 
\begin{equation}
\label{eq:en}
\phi_{\alpha}^{\omega}(x) = \sum_{n\geq0}\frac{1}{\lambda_{n}^{\alpha-\frac{1}{2}}}e_{n}^{\omega}(x)\,.
\end{equation}
Note that $e_{n}^{\omega}$ is a Gaussian vector on $E_{n}$ with $\mathbb{E}[e_{n}^{\omega}]=0$ and $\mathrm{Cov}(e_{n}^{\omega})=\frac{1}{2n+1}\mathrm{Id}_{E_{n}}$. 
\medskip

For $N\geq 0$, we define the projectors $\Pi_N$ and $\Pi_N^{\perp}$ by
	\[
\Pi_N=\sum_{n\leq N}\pi_n,\quad \Pi_N^{\perp}:=\mathrm{Id}-\Pi_N\,.
	\]
We denote by $v_N$ the solution of \eqref{NLScubic} with truncated initial data $v_N|_{t=0}=\Pi_N\phi_{\alpha}$:
\begin{equation}\label{trun_trun}
i\partial_t v_{N}-(-\Delta+1) v_{N}=|v_{N}|^2v_{N},\quad v_{N}\vert_{t=0}=\Pi_N \phi_\alpha(\cdot,\omega).
\end{equation}
The initial data in \eqref{trun_trun} is smooth and therefore, thanks to \cite{BGT},  $v_{N}$ is a (unique) globally defined smooth function on $\R\times\S^2$. We can now formulate our main result concerning the probabilistic well-posedness of \eqref{NLScubic} with data \eqref{phi_alpha}. 
\begin{theorem}\label{main:local} 
Let $\alpha>1$. There exists a set $\Sigma$ of full probability such that for every~$\omega~\in~\Sigma$ there exists $T_\omega>0$ such that  the sequence of smooth solutions $(v_{N})_{N\geq 1}$ in~{\eqref{trun_trun}} with the initial data $\Pi_N\phi_{\alpha}$, converges in $L^{\infty}([-T_\omega,T_\omega];H^s(\S^2))$ for $s<\alpha-1$ to a limit that satisfies the cubic NLS in a distributional sense on  $(-T_\omega,T_\omega)\times \S^2$.
\end{theorem}
We give a more precise version in Theorem \ref{main:rigorous}.
For $\alpha>\frac{5}{4}$, the statement of Theorem \ref{main:local} is implied by the local well-posedness of  \eqref{NLScubic} in $H^s(\S^2)$, $s>\frac{1}{4}$, obtained in \cite{BGT2}.
For  $\alpha\in (1,\frac{5}{4}]$, Theorem \ref{main:local}  requires a proof that goes far beyond the analysis developed in \cite{BGT2}. In particular, it requires a proof of the result  of  \cite{BGT2} which is less dependent on the bilinear Strichartz estimates for the free evolution and which is sensitive to the smoothing effect coming from the modulations in the Bourgain spaces associated with the Schr\"odinger equation on the sphere.  These smoothing effects are not exploited in  \cite{BGT2} while they play a key role in the probabilistic well-posedness of  \eqref{NLScubic} when estimating the regularity of the stochastic objects involved in the construction of the solutions. 
\begin{remarque}
By considering Fourier-Lebesgue spaces,~\cite{GH08} proved (deterministic) local well-posedness for the derivative-NLS equation on $\R$ beyond the optimal regularity threshold. It seems that a similar semi-linear approach is unsuccessful on the sphere. One obstruction is that that the operator norm on $L^{q}(\S^{2})$ of the spectral projector $(\pi_{n})_{n}$ grows with $n$ when $q\neq2$.
 \end{remarque}

The proof of Theorem \ref{main:local} is based on a probabilistic resolution scheme originally developed by Bourgain in \cite{B3}. This scheme was further refined in recent years in \cite{Bringmann,CGI_JDE,GKO,DNY1,DNY2,DNY3,LW-preprint,Sun-Tz}. In the present work we will greatly benefit from these developments, especially \cite{Bringmann,DNY1}.
\medskip

The problem is more challenging on the sphere than on the torus due to the larger regularity gap between the Gibbs measure support and the optimal deterministic results.

Moreover, as demonstrated in~\cite{BCLST}, the analysis on $\S^{2}$ presents a significant challenge due to the divergence of the first nonlinear Picard iteration with data $\phi_{\alpha}$
in $H^{\alpha-1}(\S^{2})$. In the case of the flat tori, however, the first nonlinear Picard iteration nearly gains~$\frac{1}{2}$ derivative, as demonstrated in the pioneering work \cite{B3}. All subsequent probabilistic well-posedness results  on flat tori, or more generally in Euclidean spaces, rely on this smoothing property. \medskip

We point out that in our analysis of NLS with data  \eqref{phi_alpha} we are strongly inspired by \cite{BL}.  This allows us to use Weyl asymptotic ideas in the context of Khinchin and Wiener chaos type estimates. In contrast with the case of flat tori, a significant part of the analysis on $\S^2$ is performed in the physical space, requiring in particular the introduction of vector valued Fourier-Lebesgue spaces. We also point out that even in the local in time analysis we use propagation of measure global arguments on the high dimensional space $E_n$, $n\gg 1$. Arguments of this spirit are not presented in the analysis on flat tori.  
\medskip
 
The globalization of the solutions obtained in Theorem \ref{main:local}  is a challenging open problem.  It would also be interesting to decide whether Theorem \ref{main:local}  holds for more general randomized functions, in the spirit of \cite{BT1/2,BT_JEMS}. 
\medskip

In the forthcoming work \cite{BCST2}, we will consider the case $\alpha=1$. In the case $\alpha=1$ the statement of  Theorem \ref{main:local}  should be properly modified because the sequence $(v_N)_{N\geq 1}$ is not expected to converge. However, after introducing some new ideas, we will show that it does converge after a suitable renormalization.
Moreover, thanks to invariant measure considerations, we will show that the limit solutions are globally defined.
\subsection{Structure of the nonlinearity }
	
 Denote the wick cubic power 
 	\[
	\mathcal{N}(u):=|u|^2u-2\|u\|_{L^2}^2u\,.
	\]
 By abusing the notation, we denote $\mathcal{N}(\cdot,\cdot,\cdot)$ the canonical trilinear form such that $\mathcal{N}(u,u,u)=\mathcal{N}(u)$.
	\medskip

	We isolate the \emph{singular trilinear form}:
	\begin{align} \mathcal{N}_{(0,1)}(f,g,h):=&\sum_{n,n_2,n_3}\pi_{n}(\pi_{n}f\cdot (\pi_{n_2}g\diamond \pi_{n_3}h ) ), 
	\label{Ns}
	\end{align}
	where
	\begin{align}\label{defwicksquare}
		(g\diamond h)(x):=\ov{g}(x)h(x)-\lg h|g\rg
	\end{align}
	is the wick-product of functions $g$ and $h$. In particular,  $|g|^{\diamond 2}:=g\diamond g = |g|^2-\|g\|_{L^2}^2$. 
\medskip

	It was shown in \cite{BCLST}
	that, on $\mathbb{S}^2$, contrarily to the case of $\mathbb{T}^2$, Bourgain's re-centering ansatz
	\[
	u(t)=\mathrm{e}^{it(\Delta-1)}\phi_{\alpha}^{\omega}(x)+r(t)
	\]
	fails to solve NLS on $\mathbb{S}^2$, as $r(t)$ cannot be viewed as a remainder living in more regular spaces. Roughly speaking, the Duhamel integration of the high-low-low type resonant interaction
	\begin{equation}
	\label{eq:sing}
	 \mathcal{N}_{(0,1)}(u_{\mathrm{high}},u_{\mathrm{low}},u_{\mathrm{low}})
	\end{equation}
	does not gain any regularity. This compromise the standard semi-linear iteration scheme based on Picard's iterations. Note that in the flat case, the zero momentum condition (a property inherited from the structure of the plane waves) and the Wick ordering remove these interactions from the nonlinearity.
	\medskip
	
Inspired by \cite{DNY1}, our resolution is based on the random averaging operator ansatz.
	\subsection{Random averaging operator (RAO) ansatz} We now prepare and state in Theorem \ref{main:rigorous} the more precise statement of Theorem \ref{main:local}, in which we describe the precise structure of the limit solutions. For a dyadic integer $N\in 2^{\N}$, denote by
	\begin{align}\label{LittlewoodPaley}
	 \mathbf{P}_N=\Pi_N-\Pi_{\frac{N}{2}}
	\end{align}
the sharp dyadic projection.
	Let $u_N$ be the (smooth) solution of the Wick-ordered cubic Schrödinger equation:
	\begin{align}\label{NLS-N}
		i\partial_tu_N+(\Delta-1) u_N=\mathcal{N}(u_N),\quad  u_N|_{t=0}=\Pi_N\phi_{\alpha}^{\omega}\,.
	\end{align}
	Let $\psi_{N}$ be the solution of the linear equation
	\begin{align}\label{eq:psiN}
		(i\partial_t+\Delta-1)\psi_{N}=2\Pi_N\mathcal{N}_{(0,1)}(\psi_{ N},u_{\frac{N}{2}},u_{\frac{N}{2}}),\quad \psi_{ N}|_{t=0}=\phi_N^{\omega}:=\mathbf{P}_N\phi_{\alpha}^{\omega}.
	\end{align} 	
	For each $N$, with the knowledge of $u_{\frac{N}{2}}$, the properties of $\psi_{N}$ are encoded in the random averaging operator (RAO) denoted $\mathcal{H}_n^{N}(t)$ on $E_{n}$ and implicitly defined by
	\begin{equation}\label{Hndef}
		f_n\mapsto \mathcal{H}_n^{N}(t)(f_n)=\e^{-it\lambda_n^2}f_n -2i\int_0^t\e^{-i(t-t')\lambda_n^2}\pi_n(\mathcal{H}_n^{N}(t')(f_n)\cdot|u_{\frac{N}{2}}(t')|^{\diamond 2})dt'.
	\end{equation}
	Given an orthonormal basis $(\mathbf{b}_{n,k})_{|k|\leq n}$ of $E_{n}$, we can also define the random operator $\mathcal{H}_n^{N}(t)$ by its matrix elements 
	\begin{align}\label{matrixHnkp}
		H_{n;\ell,k}^{N}(t):=\lg\mathcal{H}_n^{N}(t)(\mathbf{b}_{n,k}) | \mathbf{b}_{n,\ell}\rg=\lg \mathbf{b}^{N}_{n,k}(t) |\mathbf{b}_{n,\ell} \rg,
	\end{align}
	where $\mathbf{b}_{n,k}^{N}(t,x)$ is the solution of 
	\begin{align}\label{ODE:rotation}
		(i\partial_t-\lambda_n^2)\mathbf{b}_{n,k}^{N}(t):=2\pi_n(\mathbf{b}_{n,k}^{N}(t)\cdot |u_{\frac{N}{2}}(t)|^{\diamond 2} ),\quad \mathbf{b}_{n,k}^{N}|_{t=0}=\mathbf{b}_{n,k}.
	\end{align}
	In other words,
	\[
	\mathcal{H}_n^{N}(t)(\mathbf{b}_{n,k})=\mathbf{b}_{n,k}^{N}(t)=\sum_{|\ell|\leq n}H_{n;\ell,k}^{N}(t)\mathbf{b}_{n,\ell}\,.
	\]
	It turns out that for all~$t~\in~\R$ the operator $\mathcal{H}_n^{N}(t)$ is an $L^2$-isometry of $E_{n}$.
	\medskip
	
	We can now express the \emph{colored} Gaussian variables $\psi_{ N}^{\omega}$, referred to as {\it terms of type $\mathrm{(C)}$}, as follows:
\begin{align}\label{psiNbyHnN} \psi_{ N}^{\omega}(t,x)
=\sum_{\frac{N}{2}< n\leq N}\pi_n\psi_{N}^{\omega}(t,x)
=\sum_{\frac{N}{2}< n\leq N}\frac{1}{\lambda_n^{\alpha-\frac{1}{2} } }e_n^{N}(t,x),
\end{align}
where given $\mathcal{H}_n^{N}$ the random averaging operator (RAO) defined in \eqref{Hndef}, we let
\begin{equation}
\label{def:eNn}
e_n^{N}(t,x):=\mathcal{H}_n^{N}(t)(e_n^{\omega}).
\end{equation}
	With these notations, for $x\in\S^{2}$,
	\[
	e_{n}^{N}(t,x)  = \frac{1}{\sqrt{2n+1}} \sum_{|k|\leq n} g_{n,k}(\omega)\mathbf{b}_{n,k}^{N}(t,x)\,, \quad e^{N}_n(0,x)=e_n^\omega(x)\,.
	\]
	When the context is clear we drop the dependence on $\omega$ from the notations. 
	\medskip
	
	Heuristically, the colored object $\psi_{N}$ captures the singular interaction and it can be viewed as a correction of the Bourgain's re-centering ansatz:
	\begin{equation}
	\label{eq:stru}
	\psi_{ N}(t)= \e^{it(\Delta-1)}\phi_N^{\omega}+\e^{it(\Delta-1)}\Theta_{N}(t)(\phi_N^{\omega})+\mathrm{remainder},
	\end{equation}
	where $\Theta_{N}(t)$ is the linear operator (with the knowledge of $u_{\frac{N}{2}}$) defined by
	\[ \phi\mapsto 
	-i\Pi_N\int_0^t\e^{i(t-t')(\Delta-1)}\mathcal{N}_{(0,1)}(\mathbf{P}_N\phi,u_{\frac{N}{2}},u_{\frac{N}{2}})(t')dt'\,.
	\]
To describe the smooth solution $u_{N}$ we decompose
	\[
	u_N=\sum_{M\leq N}v_M,\; v_M:=u_M-u_{\frac{M}{2}}\,,
	\] 
	and we impose
	\begin{equation}
	\label{eq:ans}
	 v_M:=\psi_M+w_M\,.
	\end{equation}
	The remainders $w_M$, referred to as {\it terms of type $\mathrm{(D)}$}, will be estimated in a functional space that can be embedded into $C_{\mathrm{loc}}(\R_t,H^{\frac{1}{2}+}(\mathbb{S}^2))$.
	Since 
	\[
	\Pi_N\mathcal{N}_{(0,1)}(\psi_N,u_{\frac{N}{2}},u_{\frac{N}{2}})=\mathcal{N}_{(0,1)}(\psi_N,u_{\frac{N}{2}},u_{\frac{N}{2}})\,,
	\]
	the equation for $w_N$ reads:
	\begin{align}\label{eq:wN}
		(i\partial_t+\Delta-1)w_N=&\mathcal{N}(u_{\frac{N}{2}}+v_N)-\mathcal{N}(u_{\frac{N}{2}})-2\Pi_N\mathcal{N}_{(0,1)}(\psi_{ N},u_{\frac{N}{2}},u_{\frac{N}{2}}) \\
		=&2\big(\mathcal{N}(v_N,u_{\frac{N}{2}},u_{\frac{N}{2}})-\mathcal{N}_{(0,1)}(\psi_{ N},u_{\frac{N}{2}},u_{\frac{N}{2}})\big)
		+\mathcal{N}(u_{\frac{N}{2}},v_N,u_{\frac{N}{2}})
		\notag \\
		\notag
		+&2\mathcal{N}(v_N,v_N,u_{\frac{N}{2}})+\mathcal{N}(v_N,u_{\frac{N}{2}},v_N)+\mathcal{N}(v_N).
	\end{align}
	We see from the equation above that the ansatz \eqref{eq:ans} removes the most singular high-low-low type interaction from the equation solved by the remainder $w_{N}$.
	
	\medskip
	
	
	After solving \eqref{eq:wN}, the full ansatz can be written as
	\begin{equation}
	\label{fullansatz}
	u_N =\psi_{\leq N} + w_{\leq N}\,,\quad 
	\psi_{\leq N}:=\sum_{L\leq N}\psi_L,\quad w_{\leq N}=\sum_{L\leq N}w_L.
	\end{equation}
	Formally taking $N\rightarrow\infty$, the local structure of \eqref{NLScubic} is given by
	\[
	u=\psi_{\leq \infty}+w_{\leq\infty}\,.
	\]
	We are now ready to formulate the more precise statement of the local well-posedness Theorem \ref{main:local}.
	\begin{theorem}\label{main:rigorous} 
		Let $\alpha>1$. There exist absolute constants $C_1>c_1>0$ and $\delta_0>0$,  such that the following statements hold. For $R\geq 1$, let $\tau_{R}:=R^{-C_{1}}$. There exists a measurable set $\Sigma_{R}\subset H^{\alpha-1-}(\mathbb{S}^2)$ with 
	\[
	\mu_{\alpha}(\Sigma_R^c)<C_1\e^{-c_1R^{\delta_0}}\,,
	\] 
	such that for all $\phi\in\Sigma_{R}$, the sequence of smooth solutions $u_N$ of \eqref{NLS-N} with initial data $\Pi_{N}\phi$ converges to $u$ in $C([-\tau_R,\tau_R];H^{\alpha-1-}(\mathbb{S}^2))$. 
	\medskip
	
	Moreover, the limit $u$ is solution to the Wick-ordered cubic Schrödinger equation and can be written
			\[
			u(t)=\psi_{\leq\infty}(t)+w_{\leq\infty}(t)\,,
			\]
			where, for some $s_0>\frac{1}{2}$,
			\[
			\psi_{\leq\infty}\in C([-\tau_R,\tau_R];\mathcal{C}^{(\alpha-1)-}(\mathbb{S}^2))\quad, w_{\leq\infty}\in C([-\tau_R,\tau_R];H^{s_0}(\S^2))\,.
			\]
	\end{theorem}
In the above statement we denoted $\mathcal{C}^{(\alpha-1)-}(\S^{2})=\bigcap_{\beta<\alpha-1}\mathcal{C}^{\beta}(\S^{2})$, where $\mathcal{C}^{\beta}(\S^{2})$ is the space of $\beta$-Hölder continuous functions.
\medskip

\subsection{Main ideas of the proof} With the objects introduced earlier we can now explain in more details the ingredients of our proofs. 
\medskip

Having identified the singular interactions in \cite{BCLST} we now have to construct an adapted refined ansatz in order to show the strong convergence of the regularized solutions $(u_{N})_{N}$, and to prove the decomposition claimed in Theorem \ref{main:rigorous}. As mentioned earlier the overall strategy is inspired from both \cite{Bringmann} and \cite{DNY1}.

\medskip

\noindent {$-$ \it Non-perturbative resolution scheme.} 
For the derivative nonlinear wave equation considered in \cite{Bringmann}, Picard's iterations are not regularizing and this compromises the standard semi-linear approach of~\cite{B3,BT1/2}. Instead, \cite{Bringmann} developed a probabilistic quasi-linear (or non-perturbative) resolution scheme.  It has been known since \cite{BGT_MRL} that solving deterministically NLS on $\S^{2}$ when $s\leq\frac{1}{4}$
would require a quasi-linear ansatz too. We proved in~\cite{BCLST} that it is also the case when the initial data are randomly distributed, thus confirming that a resolution scheme like in \cite{Bringmann} was needed to go beyond the Cauchy theory achieved in \cite{BGT2} twenty years ago. 
\medskip

\noindent{$-$ \it Ansatz.} To address the Gibbs measure problem with general power type nonlinearities for NLS on $\T^{2}$, \cite{DNY1} identified the structure of the probabilistic objects in the resolution scheme of \cite{Bringmann}. This formulation, particularly the random averaging operators (RAO) as expressed in \eqref{fullansatz} in our context, provides the foundations for our ansatz. Leveraging  our understanding of the singular interactions \cite{BCLST}, we arrive at the precise description \eqref{fullansatz} of the limit solutions. A notable aspect of our work is that the RAO are not regularizing: in the limit $N\to\infty$,
\[
\psi_{N}(t) - \e^{it(\Delta-1)}\phi_{N}^{\omega}
\] 
does not exhibit smoother properties than the initial data. This is in contrast with the previous works \cite{DNY1,DNY2,DNY3} and \cite{LW-preprint,Sun-Tz}. In particular, the random tensor theory introduced in \cite{DNY2} to capture the smoothing effect in the Picard iterations does not seem to improve the (RAO) method for NLS on $\S^{2}$.  
\medskip

\noindent {$-$ \it Random objects of type $\mathrm{(C)}$.} These terms, which are not smoother but display a probabilistic structure, encode the singular interactions that contribute to the main part of the effective nonlinear dynamics. In our context, they are a superposition of time-dependent eigenfunctions $e_{n}^{N}(t)$ obtained by applying the operator $\mathcal{H}_{n}^{N}(t)$ (unitary on $E_{n}$) to the initial Gaussian spherical harmonics $e_{n}^{\omega}$, as expressed in \eqref{psiNbyHnN}.

In the present work, the rigorous construction of the random averaging operators necessitates the functional framework established in Section \ref{sec:4} and the multilinear estimates from Section \ref{sec:6}.

Another new aspect of our analysis is the use of the law invariance stated in Lemma \ref{lemme:law}. This property comes from the unitary structure of the RAO \cite{DNY3} and the probabilistic independence between the RAO and the high-frequency component of the initial data, first observed in \cite{Bringmann}.

We use the law-invariance to prove pointwise-in-time estimates for resonant interactions as well as $L^{\infty}(\mathbb{S}^2)$-bounds for time-dependent eigenfunctions $e_n^N(t)$. However, when we use  Fourier-restriction norms to exploit the time-oscillations, we have to consider averaged in time quantities for which we cannot use the law-invariance. Instead, we prove large deviation bounds in Section \ref{sec:large-dev} that allow to capture the~$L^{\infty}(\S^{2})$-bounds in the Fourier-Lebesgue norm of $e_{n}^{N}(t)$.

 \medskip
 
 \noindent{$-$ \it Remainder terms of type $\mathrm{(D)}$. } These smoother terms denoted $w_{N}$ solve the equation \eqref{eq:wN}.  We prove semi-classical energy estimates stated in \eqref{eq:bo-r}, which imply the convergence of $(w_{N})_{N}$ at regularity $H^{s_{0}}(\S^{2})$ for some $s_{0}>\frac{1}{2}$. In particular, we recover the gain of almost $\frac{1}{2}$ derivatives obtained by Bourgain on $\T^{2}$.  
 \medskip
 
Since some parts of the trilinear estimates could not be performed using general  linear or bilinear Strichartz estimates, 
we must carefully analyze the nonlinear interactions in Fourier-restriction spaces in order to exploit the time oscillations. This required the development of a more flexible analysis going much beyond the methods used in \cite{BGT2}.
\medskip
  
In Section \ref{sec:CCC} we derive the trilinear probabilistic estimates for the  (C)(C)(C) terms, and the deterministic estimates for the interactions with at least one term of type~(D) in Section \ref{sec:det}, respectively. For the (C)(C)(C) interactions we also cover the case $\alpha=1$ (corresponding to the Gaussian free field on $\S^{2}$), which explains the length of Section~\ref{sec:CCC}. 
 \medskip

Finally, we note that our proof of Theorem \ref{main:rigorous} can be strengthened to get a stability result in the spirit of \cite[Proposition 5.5]{DNY1}, which is useful in order to iterate in time the local well-posedness result.

\subsection{Remark on a completely resonant system} For the following completely resonant system
\begin{align}\label{eq:res}
i\partial_tu+\Delta u=\sum_{n,m}\pi_n(\pi_nu|\pi_mu|^2),\quad (t,x)\in\mathbb{R}\times\mathbb{S}^2,
\end{align}
one easily verifies that for every $s\in\mathbb{R}$, the $H^s(\S^{2})$-norm of the smooth solution $u(t)$ to~\eqref{eq:res} is conserved. Such conservation laws lead to invariant measures of the formal form 
\[
\frac{1}{Z}\exp\Big(-\frac{1}{2}\|u\|_{H^\alpha(\S^{2})}^2\Big)du\,.
\]
As a by-product of our analysis for proving Theorem \ref{main:rigorous} (as well as the stability statement), together with Bourgain's invariant measure argument, we have the following almost sure global well-posedness result for the completely resonant system of the sphere, at positive regularity:
\begin{theorem}\label{thm:resonant}
Let $\alpha>1$. There is a set $\Sigma$ of full probability such that for every $T>0$, $\omega\in \Sigma$, the sequence of smooth solution $(u_N)_{N\geq 1}$ of \eqref{eq:res} with initial data $\Pi_N\phi_{\alpha}^{\omega}$ converges in $L^{\infty}([-T,T];H^{\beta}(\mathbb{S}^2)), \beta<\alpha-1$ to a limit which satisfies \eqref{eq:res} in a distributional sense on $[-T,T]\times\mathbb{S}^2$.
\end{theorem}
We stress out that for $\alpha\leq\frac{5}{4}$, we do not know how to prove Theorem \ref{thm:resonant} without using the machinery elaborated in the proof of Theorem \ref{main:rigorous}.
	
	\subsection*{Acknowledgments}
	This research was supported by the European research Council (ERC) under the European Union’s Horizon 2020 research and innovation programme (Grant agreement 101097172 - GEOEDP). C.S and N.T. were partially supported by the ANR project Smooth ANR-22-CE40-0017.
	N.C. benefited from the support of the Centre Henri Lebesgue ANR-11-LABX-0020-0, the region Pays de la Loire through the project MasCan and the ANR project KEN ANR-22-CE40-0016.

	\section{Dictionary of global notations}
	
	We summarize generic notations and some notations for important objects.

	\begin{itemize}
		
		\item {\it Estimates:} We will adopt the generic PDE notations $f\lesssim g, f\gtrsim g$, and $ f\sim g$ if $f\lesssim g$ and $f\gtrsim g$. If we want to emphasize the dependence of a parameter $A$ in the implicit constants, we will write as $f\lesssim_A g$. These notations will be often used in the proof, but we will also use $C, C_A$.
		\medskip

		\item {\it Probability}: Given a fixed probability space $(\Omega,\mathcal{F},\mathbb{P})$, $\mathcal{B}_{\leq N}$ denotes the $\sigma$-algebra generated by Gaussians $(g_{n,k})_{|k|\leq n, n\leq N}$, and $\mathcal{B}_N$ denotes the $\sigma$-algebra generated by Gaussians $(g_{n,k})_{|k|\leq n, \frac{N}{2}<n\leq N}$. Given a sub-$\sigma$ algebra $\mathcal{B}$ of $\mathcal{F}$ and $p\geq 1,$  denote by
		\[
		\|X(\omega)\|_{L_{\omega}^p|\mathcal{B}}:= \big(\mathbb{E}[|X(\omega)|^p|\mathcal{B}]\big)^{\frac{1}{p}}
		\]
		the conditional $L^p$-moment of a random variable $X(\omega)$ with respect to $\mathcal{B}$.
		\medskip
		
		\item {\it Operators:} For Banach spaces $\mathcal{X},\mathcal{Y}$, we denote by $\mathcal{L}(\mathcal{X},\mathcal{Y})$ the space of bounded linear operators from $\mathcal{X}$ to $\mathcal{Y}$. If $\mathcal{X}=\mathcal{Y}$, we denote  it simply  by $\mathcal{L}(\mathcal{X})$.
		
		\medskip
			
		\item {\it Dyadic frequencies:}  Capital numbers $N,M,L,N_j,\cdots$
		always refer to dyadic integers, when they appear as subscript/superscript of functions/operators as well as the summation. Given a set of dyadic numbers $N_1,N_2,N_3,\cdots$, we denote by $N_{(1)},N_{(2)}, N_{(3)},\cdots$ the non-increasing ordering among them. In particular, $N_{(1)}\geq N_{(2)}\geq N_{(3)}\geq\cdots$. 
		\medskip
		
		\item {\it Projections:} For $n\in\N$, $\pi_n: L^2(\S^2)\rightarrow E_n$ is the orthonogonal projection on the eigenspace $E_n$ of the eigenvalue $\lambda_n^2=n^2+n+1$. The space $E_n$ is identified with the $L^2$ topology. Given $f\in E_n$ and $1\leq r\le\infty$ we denote by $\|f\|_{L^r(E_n)}:=\|f\|_{L^r(\S^2)}$.
		For dyadic integers $N$, $\Pi_N,\mathbf{P}_N$ defined in \eqref{LittlewoodPaley} are the (sharp) Littlewood-Paley projectors.
				
		\medskip

		\item{\it Important objects:} $u_N$ is the solution of the truncated NLS equation \eqref{NLS-N} with initial data $\Pi_N\phi^{\omega}$, and $\psi_{N}$ is the solution of \eqref{eq:psiN} with initial data $\mathbf{P}_N\phi^{\omega}$.  $\mathcal{H}_n^N(t)$ is the random averaging operator (RAO). 
		
		\medskip

		\item {\it Time cut-off functions}: $\chi,\chi_1,\varphi, \eta\cdots$. For $T>0$, we set 		\[
		\chi_T(\cdot):=\chi(T^{-1}\cdot)\,.
		\]		
		
		\medskip

		\item For Lebesgue spaces $L^q$, we will use the conjugate exponent $q'=\frac{q}{q-1}$ (with the convention that $q'=1$ if $q=\infty$ and $q'=\infty$ if $q=1$).  
		
		\medskip

		\item {\it Generic small parameters:} We reserve a very small parameter $\sigma\in(0,2^{-100})$ throughout this article and define several related large/small parameters $q_{\sigma}, \gamma_{\sigma},\delta_{\sigma},\theta_{\sigma},b_{\sigma}$ around some specific numbers such that: 
		\begin{align}
		\label{smallparameters} 
		\frac{1}{q_{\sigma}}&=\sigma\,,\quad  
		1-\gamma_{\sigma}=\sigma-\sigma^{10}\,,\quad
		1-\gamma_{1,\sigma}=\sigma-\sigma^{15}\,,\\ 
		\label{smallparameters1} 
		\delta_{\sigma}&=\sigma^{20}\,, 
		\ 
		s_{\sigma}=\alpha-\frac{1}{2}-100\sigma, \quad \theta_{\sigma}=\sigma^5\,.
		\intertext{The parameter $b_{\sigma}$ is restricted to the interval}
		\label{smallparameters2} 
		\frac{1}{2}&<b_{\sigma}<\frac{1}{2}+\theta_{\sigma}\,.
		\intertext{When $\alpha>1$ we suppose that $\sigma$ is small enough (depending on $\alpha-1$) so that}
		\label{smallparameters3} 
		s_{\sigma} &= \alpha-\frac{1}{2} -\frac{2}{q_{\sigma}} -\delta_{\sigma} - 100\sigma > \frac{1}{2} + 100\sigma\,,
		\quad
		\end{align}
		Note that $1<\gamma_{1,\sigma}q_{\sigma}'<\gamma_{\sigma}q_{\sigma}'$. The hierarchy of the smallness is needed to estimate deterministic objects into the Fourier-restriction type spaces, for technical reason. In summary, we have the following hierarchy of parameters:
		\[
		0<\delta\ll \gamma_{1,\sigma}-\frac{1}{q_{\sigma}'}\ll \gamma-\frac{1}{q_{\sigma}'}\ll \theta\ll \frac{1}{q_{\sigma}}\ll s_{\sigma}-\frac{1}{2}\ll 1\,.
		\]
			
		\item {\it Special symbols:}
		\medskip
		
		\begin{itemize}
			\item[(a)]
			For simplicity, later we will write 
			\begin{align}\label{defapproxi}
				n\approx L\quad \iff\quad \frac{L}{2}<n\leq L\,.
			\end{align}
			
			\item[(b)] The wick square $|f|^{\diamond 2}$ is defined in \eqref{defwicksquare} and the exotic products $f\hexagonal{=} g$ and $f\hexagonal{\neq }g$ are defined in \eqref{def:hexagonal} 
			
		\end{itemize}
		
	\end{itemize}
Notation for the pairing and non-pairing products:  we decompose
\begin{align}\label{WickNonlinearity} 
	\mathcal{N}(u,u,u)=\mathcal{N}^{(1)}(u,u,u)+\mathcal{N}^{(2)}(u,u,u)+\mathcal{N}^{(3)}(u,u,u),
\end{align}
where
\begin{align}
	&\mathcal{N}^{(1)}(f_1,f_2,f_3)=\sum_{\substack{n_1,n_2,n_3\\
			n_2\neq n_1,n_2\neq n_3 }}\pi_{n_1}f_1\ov{\pi_{n_2}f_2}\pi_{n_3}f_3, \notag  \\
	&\mathcal{N}^{(2)}(f_1,f_2,f_3)=\sum_{\substack{n_1,n_2\\ n_1\neq n_2
	}}\big[\pi_{n_2}f_3\ov{\pi_{n_2}f_2}-\lg\pi_{n_2}f_3|\pi_{n_2}f_2\rg\big]\pi_{n_1}f_1 \notag \\&\hspace{2.5cm}+
	\sum_{\substack{n_2,n_3\\ n_2\neq n_3
	}}\big[\pi_{n_2}f_1\ov{\pi_{n_2}f_2}-\lg\pi_{n_2}f_1|\pi_{n_2}f_2\rg\big]\pi_{n_3}f_3,\notag \\
	&\mathcal{N}^{(3)}(f_1,f_2,f_3)=\sum_{n}\big[\pi_{n}f_1\ov{\pi_{n}f_2}\pi_{n}f_3-\lg \pi_{n}f_1|\pi_{n}f_2\rg\pi_nf_3-\lg\pi_nf_3|\pi_nf_2\rg\pi_nf_1 \big]. \notag
\end{align}
As the wick product is not always needed, we will also use the notation
\[
\mathcal{N}^{(0)}(f_1,f_2,f_3):=f_{1}\overline{f_{2}}f_{3} = \sum_{n_0,n_1,n_2,n_3}\pi_{n_0}(\pi_{n_1}f_1\pi_{n_2}\ov{f}_2\pi_{n_3}f_3) \,.
\]
We will add $(\,,\,)$ to present paired, and $[\,,\,]$ to present non-paired  constraints that appear as subscripts of $\mathcal{N}^{(j)}(f_1,f_2,f_3),\ j=0,1,2,3$. For example,
for $j,k~\in~\{0,1,2,3\}$,
\[
\mathcal{N}^{(0)}_{(j,k)}(f_1,f_2,f_3)
	=
	\sum_{n_0,n_1,n_2,n_3}\mathbf{1}_{n_j=n_k}\pi_{n_0}(\pi_{n_1}f_1\pi_{n_2}\ov{f}_2\pi_{n_3}f_3)\,.
\]
If several pairing conditions are satisfied simultaneously, we write them as accumulated subscripts. For example,
\[ \mathcal{N}^{(0)}_{(0,1)(2,3)}(f_1,f_2,f_3)=\sum_{n_0,n_1,n_2,n_3}\mathbf{1}_{n_0=n_1}\mathbf{1}_{n_2=n_3}\pi_{n_0}(\pi_{n_1}f_1\pi_{n_2}\ov{f}_2\pi_{n_3}f_3).
\]
For non-pairing constraints, we express the condition $n_j\neq n_k$ as subscript $[j,k]$. For example, 
\begin{align} \mathcal{N}^{(0)}_{[j,k]}(f_1,f_2,f_3)=\sum_{n_0,n_1,n_2,n_3}\mathbf{1}_{n_j\neq n_k}\pi_{n_0}(\pi_{n_1}f_1\pi_{n_2}\ov{f}_2\pi_{n_3}f_3).
\end{align}
The completely non-paired product can be represented by the subscript $[0,1,2,3]$:
\[ \mathcal{N}^{(0)}_{[0,1,2,3]}(f_1,f_2,f_3):=\sum_{\substack{n_0,n_1,n_2,n_3\\\mathrm{distinct}} } \pi_{n_0}(\pi_{n_1}f_1\pi_{n_2}\ov{f}_2\pi_{n_3}f_3).
\]
Similarly, if a pairing condition $(j,k)$ and a non-pairing condition $[l,i]$ simultaneously appears, we will use the accumulate subscript $(j,k)[l,i]$:
\[ \mathcal{N}^{(0)}_{(j,k)[l,i]}(f_1,f_2,f_3)=
\sum_{\substack{n_0,n_1,n_2,n_3\\
		n_j= n_k,n_l\neq n_i
} } \pi_{n_0}(\pi_{n_1}f_1\pi_{n_2}\ov{f}_2\pi_{n_3}f_3).
\]

	\section{Eigenfunction estimates and probabilistic Sobolev embedding}
	\label{sec:3}
	
	\subsection{Eigenfunction estimates on \texorpdfstring{$\S^2$}{S2}}
	
	The next two propositions are true on any Riemannian compact surface. 
	
	\begin{proposition}[Eigenfunction estimate, \cite{sogge0}]\label{Sogge} There exists $C>0$ such that for all $n\geq1$ and $f\in L^{2}(\S^{2})$,  
		\[ \|\pi_nf\|_{L^p(\mathbb{S}^2)}
		\leq C 
			\begin{cases}
			& n^{\frac{1}{2}(\frac{1}{2}-\frac{1}{p})}\|\pi_{n}f\|_{L^{2}(\S^{2})}
			\,,\quad 2\leq p\leq 6\,,   \\
			& n^{\frac{1}{2}-\frac{2}{p}}
			\|\pi_{n}f\|_{L^{2}(\S^{2})}
			\,,\quad 6\leq p\leq\infty\,.
		\end{cases}
		\]
	\end{proposition}
	
	\begin{proposition}[Bilinear eigenfunction estimate, \cite{BGT2}]\label{bilinearBGT} 
		There exists $C>0$ such that for all $n_1\geq  n_2\geq1$, every $f,g\in L^2(\mathbb{S}^2)$,
		\[ \|\pi_{n_1}f\cdot \pi_{n_2}g\|_{L^2(\mathbb{S}^2)}\leq  Cn_2^{\frac{1}{4}}
		\|\pi_{n_{1}}f\|_{L^2(\mathbb{S}^2)}
		\|\pi_{n_{2}}g\|_{L^2(\S^2)}\,.
		\]
	\end{proposition}
We will need a probabilistic estimate for the random spherical harmonics for which the proof is standard:
\begin{lemme}\label{concentration}
There exist $C>0$ such that for all $2\leq p<\infty$, $n\in\N$ and $x\in\S^{2}$, 
\[
\|e_{n}^{\omega}(x)\|_{L_{\omega}^{p}} \leq C\sqrt{p}. 
\] 
Moreover, there exists $C_0>c_0>0$, such that for every $R\geq 1$,
\[ \mathbb{P}\big[\|e_n^{\omega}\|_{L^p(\mathbb{S}^2)}>R \big]\leq C_0\e^{-c_0R^2}.
\]
\end{lemme}
Lemma \ref{concentration} is sufficient for our needs. We refer to \cite{BL2} for stronger concentration inequalities.  

\begin{proof}
The first part of Lemma \ref{concentration} is a consequence of Khinchin's inequality and the exact Weyl's law \eqref{Weyl}, from which the large deviation bound easily follows. 
\end{proof}


	\subsection{Elementary properties of the random averaging operators}
	We first establish the simple but crucial property that $\mathcal{H}_n^N(t)$ is unitary on $E_n$:
	\begin{lemme}\label{unitary}For fixed $t\in[0,T]$, the operator $\mathcal{H}_n^N(t)$ is unitary on $E_n$. In particular, the matrix element $(H_{n;k,k'}^N(t))$ expressed in a given orthonormal basis $(\mathbf{b}_{n,k})_{|k|\leq n}$ verifies
		\begin{equation}\label{id:unitary}
			\sum_{|k|\leq n }
			H_{n;k,\ell}^N(t)\ov{H}_{n;k,\ell'}^N(t)=\delta_{\ell\ell'},\quad
			\sum_{|\ell|\leq n}H^N_{n;k,\ell}(t)\ov{H}_{n;k',\ell}^N(t)=\delta_{kk'}\,.
		\end{equation}
	\end{lemme}
	This is the consequence of Lemma \ref{lem:unitary} and the fact that the potential operator
	\[ f_n\mapsto \pi_n(f_n\cdot |u_{\frac{N}{2}}(t)|^{\diamond 2})
	\]
	is self-adjoint. This type of unitary property for random averaging operators was first observed in \cite{DNY3}. As in \cite{DNY3}, we will use it to achieve some cancellation for Wick's cubic power of colored Gaussians. In addition, we use this to prove that the law of the colored terms is invariant (see Proposition \ref{cor:law}). 
	
	\medskip
	
	A key ingredient is the following cancellation from Wick-ordering. This is a global aspect of the RAO theory, first appeared in \cite{DNY3}.
	\begin{lemme}\label{Wickcancellation} For all $N$ and $n\approx N$, we have that point-wise in $\omega$, $x$, t 
		\begin{align}
		\nonumber
			|e_n^{N}(t,x)|^2-\|e_n^{N}(t,\cdot)\|_{L_x^2}^2
			=&\sum_{\substack{|k|,|k'|,|\ell|,|\ell'|\leq n \\
					\ell\ne \ell'
				}
			}H_{n;k,\ell}^{N}(t)\ov{H}_{n;k',\ell'}^{N}(t)
			\frac{g_{n,\ell}^{\omega}(\omega)\ov{g}_{n,\ell'}(\omega)}{2n+1 }\mathbf{b}_{n,k}(x)\ov{\mathbf{b}}_{n,k'}(x) 
			\\
			\nonumber
			&+\sum_{|k|,|k'|,|\ell|\leq n	
			}H_{n;k,\ell}^{N}(t)\ov{H}_{n;k',\ell}^{N}(t)\frac{|g_{n,\ell}(\omega)|^2-1 }{2n+1}\mathbf{b}_{n,k}(x)\ov{\mathbf{b}}_{n,k'}(x)
			\\
			&-\sum_{|\ell|\leq n}\frac{|g_{n,\ell}(\omega)|^2-1 }{2n+1}.
		\end{align}
	\end{lemme}
	\begin{proof} For simplicity we drop the dependence on $\omega$ and $t$ from the notation. Expanding $|e_n^{N}(t,x)|^2-\|e_n^{N}(t,\cdot)\|_{L_x^2}^2$ gives
		\begin{align*}
			&\sum_{k,k',\ell,\ell'}H_{n;k,\ell}^N\ov{H}_{n;k',\ell'}^N\frac{g_{n,\ell}\ov{g}_{n,\ell'}}{2n+1}\mathbf{b}_{n,k}(x)\ov{\mathbf{b}}_{n,k'}(x)-\sum_{k,\ell,\ell' }
			H_{n;k,\ell}^N\ov{H}_{n;k,\ell'}^N\frac{g_{n,\ell}\ov{g}_{n,\ell'}}{2n+1}
			\\
			=&\sum_{\substack{k,k',\ell,\ell'\\
					\ell \neq \ell'} }
			H_{n;k,\ell}^N\ov{H}_{n;k',\ell'}^N\frac{g_{n,\ell}\ov{g}_{n,\ell'}}{2n+1}\mathbf{b}_{n,k}(x)\ov{\mathbf{b}}_{n,k'}(x)\\
			+&\sum_{k,k',\ell} H_{n;k,\ell}^N\ov{H}_{n;k',\ell}^N\frac{|g_{n,\ell}|^2 }{2n+1}\mathbf{b}_{n,k}(x)\ov{\mathbf{b}}_{n,k'}(x)
			-\sum_{\ell}\frac{|g_{n,\ell}|^2}{2n+1},
		\end{align*}
		where we used the unitary property (Lemma \ref{unitary}). Using again Lemma \ref{unitary}, the second line on the right hand side can be computed as
		\[ 
		\sum_{k,k',\ell}H_{n;k,\ell}^N\ov{H}_{n;k',\ell}^N\frac{|g_{n,\ell}|^2-1}{2n+1}\mathbf{b}_{n,k}(x)\ov{\mathbf{b}}_{n,k'}(x)
		+\sum_{k}\frac{|\mathbf{b}_{n,k}(x)|^2}{2n+1}
		-\sum_{\ell}\frac{|g_{n,\ell}|^2}{2n+1}.
		\]
		By Weyl, we have 
		\[
		\sum_{|k|\leq n}\frac{|\mathbf{b}_{n,k}(x)|^2}{2n+1}
		= 1
		=\sum_{|\ell|\leq n}\frac{1}{2n+1}
		\,.
		\]
		Regrouping terms, the proof is complete.
	\end{proof}

	We conclude this section with a new property  for colored-Gaussians on $E_n$ defined
		\[
		e_n^{N}(t)~=~\mathcal{H}_n^{N}(t)~(e_n^{\omega})\,,
		\] 
		which is helpful in the estimate for the resonant interactions for which we cannot use time-modulation. 
		\begin{proposition}\label{cor:law}
		For fixed $n>N/2$ and $t\in\R$, the random field $e_n^{N}(t)$ has the same law as $e_n^\omega$:
		\[
		\mathscr{L}(e^{N}_n(t)) = \mathscr{L}(e_n^\omega)=\mathscr{N}_{E_n}(0 ; \operatorname{Id})\,.
		\]
	\end{proposition}

	The proof follows from the elementary probability Lemma \ref{lemme:law} and the probabilistic structure of the ansatz. 
	
	\begin{proof}
		For fixed $n>N/2$, the vector $e_n^{N}(t)=\mathcal{H}_n^{N}(t)(e_n^{\omega})$ is the image of the normal Gaussian variable $e_n^{\omega}$ on $E_n$ under the random rotation $\mathcal{H}_n^{N}(t)$. Note that $\mathcal{H}_n^{N}(t)$ depends on the $\sigma$-algebra that is $\mathcal{F}_{N/2}$-measurable, hence independent of $e_n^{\omega}$ since $n\leq N/2$. 
		It follows from Lemma~\ref{lemme:law} that for every $t\in[0,T]$,
		\begin{equation}
			\label{same:law}
			\mathscr{L}(e_n^{N}(t)) = \mathscr{L}(e_n^{\omega})\,. 
		\end{equation}
		This completes the proof of Proposition \ref{cor:law}.
		\end{proof}

	\section{Functional spaces for functions and operators}
	\label{sec:4}
	\subsection{Fourier-Lebesgue restriction spaces for functions}

	Recall that~$E_n
	$ is a closed subspace of $L^2(\mathbb{S}^2)$ of dimension $2n+1$, endowed with the $L^2(\mathbb{S}^2)$ norm. For a space-time Schwartz distribution  $F\in\mathcal{S}'(\mathbb{R}\times\mathbb{S}^2)$, $\pi_nF$ can be viewed as a Schwartz distribution with values in $E_n$, namely $\pi_nF\in \mathcal{S}'(\R;E_n)$.
	We denote by $\widehat{F}(\tau,\cdot)$ its time Fourier transform. By decomposing into eigenspaces $E_n$, we have 
	\[  \widehat{F}(\tau,x)=\sum_{n=0}^{\infty}\widehat{\pi_nF}(\tau,x).
	\]
	For $p,q,r\in[1,\infty]$, consider the restriction norm:
	\begin{align*}
		\|F\|_{X_{p,q,r}^{s,\gamma}}:=\|\lambda_{n}^s\langle\tau+\lambda_n^2\rangle^{\gamma}\widehat{\pi_n F}(\tau,\cdot)\|_{\ell_n^pL_{\tau}^qL_{x}^r}
		=
		\|\lambda_{n}^s\lg \kappa\rg^{\gamma}\pi_n\widehat{ \e^{-it(\Delta-1)}F}(\kappa,\cdot) \|_{\ell_n^pL_{\kappa}^qL_x^r}.
	\end{align*}
	It will be more convenient to introduce the notation of  twisted time-Fourier transform.  For a function $F_n\in \mathcal{S}'(\mathbb{R};E_n)$, we denote by 
	\[ \widetilde{F}_n(\kappa,x):=\widehat{F}_n(\kappa-\lambda_n^2,x)=\widehat{\e^{-it(\Delta-1)}F_n}(\kappa,x).
	\]
	For a general distribution $F\in\mathcal{S}'(\R\times\mathbb{S}^2)$, we denote by 
	\[  \widetilde{F}(\kappa,\cdot):=\widehat{\e^{-it(\Delta-1)}F}(\kappa,\cdot) =\sum_{n=1}^{\infty}\widetilde{\pi_nF}(\kappa,\cdot).
	\]
	Though $E_n$ is finite-dimensional space, different norms are not asymptotically equivalent as $n\rightarrow\infty$. By abusing the notation, we denote by $L^r(E_n)$ the space $E_n$ endowed with the $L^r(\S^2)$ norm, with the convention $L^2(E_n)=E_n$. Set
	\[ \|F_n\|_{X_{q,r}^{s,\gamma}(E_n)}
		:=
		\lambda_{n}^{s}
		\|\lg\kappa\rg^{\gamma}\pi_n\widehat{\e^{-it(\Delta-1)} F_n}(\kappa,x)
	\|_{L_{\kappa}^qL_x^r}=\lambda_{n}^{s}
	\|\lg\kappa\rg^{\gamma}
	\widetilde{\pi_nF_n}(\kappa,x) 
	\|_{L_{\kappa}^qL_x^r}\,.
	\]
	With these notations, we have
	\[ \|F\|_{X_{p,q,r}^{s,\gamma}}=\|\pi_nF\|_{\ell_n^pX_{q,r}^{s,\gamma}(E_n) }.
	\]
	Note that $X_{2,2,2}^{0,b}$ is just the usual $X^{0,b}$ space in the literature. One keeps in mind that in our trilinear estimates, we will put remainders in $X_{2,2,2}^{0,\frac{1}{2}+}$ and terms like linear evolutions in $X_{\infty,\infty,\infty}^{\alpha-\frac{1}{2},1+}$ essentially. Given an interval $I\subset\R$, we define the norm $X_{p,q,r}^{s,b}(I)$ by
	\[ \|F\|_{X_{p,q,r}^{s,b}(I)}=\inf\{\|F_1\|_{X_{p,q,r}^{s,b}}:\; F_1|_I=F \}.
	\]
	Let $\chi\in C_c^{\infty}((-1,1))$ and $\chi_T(t):=\chi(T^{-1})$ be a time cutoff function, we have the following localization property:
	\begin{proposition}\label{timecutoff} 
		Let $1\leq p,r\leq \infty,$ and $1\leq q<\infty$, $0<\gamma\leq \gamma_1<1$. For any $u\in X_{p,q,r}^{s,\gamma_1}$ such that $u(t=0,\cdot)=0$, we have
		\[ \|\chi_T(t)u\|_{X_{p,q,r}^{s,\gamma}}\lesssim T^{\gamma_1-\gamma}\|u\|_{X_{p,q,r}^{s,\gamma_1}}
		\]
		for all $0<T\leq 1$.
	\end{proposition}
	We provide the proof in Appendix (see Corollary \ref{timecutoffappendix}). We have the embedding property:
	\begin{lemme}\label{Embedding}
		Let $1<q<\infty, 1\leq r\leq \infty$ and $\gamma\in\big(\frac{1}{q'},1\big)$. Then, for any $\alpha~\in~(0,\gamma-~\frac{1}{q'}]$,
		\[
		X_{q,r}^{0,\gamma}(E_n)\hookrightarrow C^{\alpha}(\R;L^r(E_n))\,.
		\]
	\end{lemme}
	The proof is a direct consequence of Lemma \ref{FLtoHolder}. Denote by
	\[ \mathcal{I}F(t):=\int_0^t\e^{i(t-t')(\Delta-1)}F(t')dt'
	\]
	the Duhamel operator. We have the following linear inhomogeneous estimate:
	\begin{proposition}\label{inhomogeneous}
		Let $\chi\in C_{c}^{\infty}(\R)$, $T>0$ and $\chi_{T}:=\chi(T^{-1}\cdot)$. For any $s\in\R$, for $q\in(1,\infty)$ and $\gamma\in(\frac{1}{q'},1)$, for any $\theta\in(0,1-\gamma)$, we have that for all $n$,
		\begin{align*}
			\|\chi_T(t)\mathcal{I}\pi_nF\|_{X_{q,r}^{s,\gamma}(E_n)}
			&\lesssim
			T^{\theta}\|\pi_nF\|_{X_{q,r}^{s,\gamma-1+\theta}(E_n)}\,.
			\intertext{In particular, taking the $\ell_{n}^{p}$-norm on both sides gives} 
			\|\chi_T(t)\mathcal{I}F\|_{X_{p,q,r}^{s,\gamma}}
			&\lesssim
			T^{\theta}\|F\|_{X_{p,q,r}^{s,\gamma-1+\theta}}\,.
		\end{align*}
	\end{proposition}
	\begin{proof}
	The proof is a direct consequence of Corollary~\ref{timecutoffappendix} and Lemma \ref{inhomogenousabstract}.
	\end{proof}
	 We end this Section by recalling the semi-classical Strichartz estimate from~\cite{BGT} and its formulation in Fourier restriction spaces. We emphasize that this estimate holds on any compact surface. 
	
	\begin{proposition}[Semi-classical Strichartz estimate, \cite{BGT}]\label{prop:str}
	For all $2\leq p,q\leq+\infty$ with $q\neq+\infty$ admissible
	\[
	\frac{1}{p}+ \frac{1}{q} = \frac{1}{2}\,,
	\]
	and for all compact interval $I\subset\R$, there exists $C(I,p,q)>0$ such that for all~$u\in~H^{\frac{1}{p}}(\S^{2})$\,,
	\[
	\|\e^{it(\Delta-1)}\phi\|_{L^{p}L^{q}(I\times\S^{2})} \leq C(I,p,q) \|\phi\|_{H^{\frac{1}{p}}(\S^{2})}\,.
	\]
	Moreover, if $\chi$ is a test function on $\R$, then for all $b>\frac{1}{2}$ and $u\in X^{\frac{1}{p},b}$, 
	\[
	\|\chi(t)u\|_{L^{p}L^{q}(\R\times\S^{2})}\leq C(\chi,p,q)\|u\|_{X^{\frac{1}{p},b}}\,.
	\]
	\end{proposition}
	The bound in Fourier restriction spaces follows from the transference principle~\cite{Gin}.  Semi-classical Strichartz estimates are useful in the deterministic analysis of Section~\ref{sec:det}.

	\subsection{Fourier restriction norms for operators}

	Let $\mathcal{T}_n(t):E_n\rightarrow E_n$ be a time-dependent linear operator on the eigenspace $E_{n}$. Denote by
	$(T_{n;k,k'}(t))_{|k|,|k'|\leq n}$ the matrix element of $\mathcal{T}_n$ under a given orthonormal basis $(\mathbf{b}_{n,k})_{|k|\leq n}$ of $E_n$:
	\[ \mathcal{T}_n(t)(\mathbf{b}_{n,k})=\sum_{|\ell|\leq n}T_{n;\ell,k}(t)\mathbf{b}_{n,\ell}\,.
	\]
	Denote by
	\[ \widehat{\mathcal{T}}_n(\tau):=\int_{\R}\mathcal{T}_n(t)\e^{-it\tau}dt
	\]
	the time Fourier transform~\footnote{we recall that \(\mathcal{T}_{n}(t)\) is a finite dimensional object so the integral makes sense.} of the linear operator $\mathcal{T}_n(t)$, and set	
	\[\widetilde{\mathcal{T}}_n(\kappa):=\widehat{\e^{it\lambda_n^2}\mathcal{T}_n(t)}(\kappa)
	\]
	the twisted time Fourier transform of the operator $\mathcal{T}_n(t)$. We  define the spaces $S_n^{q,\gamma}$ and $S_n^{q,\gamma,\ast}$ of  linear operators $\mathcal{T}_n(t),\mathcal{L}_n(t)\in\mathcal{S}'(\R;\mathcal{L}(E_n))$ associated to the norms:
	\begin{align}\label{Snnorm} \|\mathcal{T}_n\|_{S_n^{q,\gamma}}:=\|\langle\tau+\lambda_n^2\rangle^{\gamma}\widehat{\mathcal{T}}_n(\tau)\|_{E_n\rightarrow L_{\tau}^qE_n},
	\end{align}
	\begin{align}\label{Sn*norm}
		\|\mathcal{L}_n\|_{S_n^{q,\gamma,*}}:=\|\lg\tau-\lambda_n^2\rg^{\gamma}\widehat{\mathcal{L}}_n(\tau) \|_{E_n\rightarrow L_{\tau}^qE_n} \,.
	\end{align}
In terms of twisted Fourier transform,
	\begin{align}
&	\|\mathcal{T}_n\|_{S_n^{q,\gamma}}=\|\lg\kappa\rg^{\gamma} \mathcal{F}_t(\e^{i\lambda_n^2 t}\mathcal{T}_n(t) )(\kappa) \|_{E_n\rightarrow L_{\kappa}^qE_n}, \label{Snnorm'} \\ &\|\mathcal{L}_n\|_{S_n^{q,\gamma,*}}=\|\lg\kappa\rg^{\gamma}
	\mathcal{F}_t(\e^{-i\lambda_n^2 t}\mathcal{L}_n(t)  )(\kappa) \|_{E_n\rightarrow L_{\kappa}^qE_n}.	 \label{Sn*norm'}
	\end{align}
	In terms of matrix element,
	\[ \|\mathcal{T}_n\|_{S_n^{q,\gamma}}=\|\lg\kappa\rg^{\gamma}\widehat{T}_{n,\ell,k}(\kappa-\lambda_n^2) \|_{l_k^2\rightarrow L_{\kappa}^ql_{\ell}^2},\quad\|\mathcal{L}_n\|_{S_n^{q,\gamma,*}}=
	\|\lg\kappa\rg^{\gamma}\widehat{L}_{n,\ell,k}(\kappa+\lambda_n^2) \|_{l_{k}^2\rightarrow L_{\kappa}^ql_{\ell}^2}.
	\]
	In the applications, given $\mathcal{H}_n(t)\in \mathcal{S}'(\R;\mathcal{L}(E_n))$, we will control the $S_n^{q,\gamma}$-norm for $\mathcal{H}_n(t)$ and the $S_n^{q,\gamma,*}$-norm for its adjoint $(\mathcal{H}_n(t))^*$, respectively.
	Note that
	\[ \widetilde{\mathcal{H}}_n(\kappa)=\mathcal{F}_t(\e^{i\lambda_n^2t}\mathcal{H}_n(t))(\kappa) =\int_{\R}\e^{i\lambda_n^2t-it\kappa}\mathcal{H}_n(t)dt.
	\]
	For fixed $\kappa\in\R$, as an operator acting on $E_n$, the adjoint of $\widetilde{\mathcal{H}}_n(\kappa)$ is given by
	\[ (\widetilde{\mathcal{H}}_n(\kappa))^*=\mathcal{F}_t(\e^{-i\lambda_n^2 t}(\mathcal{H}_n(t))^*)(-\kappa) =\int_{\R}\e^{-i\lambda_n^2t+it\kappa}(\mathcal{H}_n(t))^*dt.
	\]
	The matrix elements of $\widetilde{\mathcal{H}}_n(\kappa)$ is
	$ \widetilde{H}_{n;\ell,k}(\kappa)=\widehat{H}_{n;\ell,k}(\kappa-\lambda_n^2),
	$ and the matrix element of $(\widetilde{\mathcal{H}}_n(\kappa))^*$ is
	\[ (\widetilde{\mathcal{H}}_n(\kappa))^*_{\ell,k}= \widehat{\ov{H}}_{n;k,\ell}(\lambda_n^2-\kappa)=\ov{\widehat{H}}_{n;k,\ell}(\kappa-\lambda_n^2)=\ov{\widetilde{H}}_{n;k,\ell}(\kappa).
	\]	
	It follows from the definitions that
	\[ \|\mathcal{H}_n\|_{S_n^{q,\gamma}}=\|\lg\kappa\rg^{\gamma}\widetilde{\mathcal{H}}_n(\kappa)\|_{E_n\rightarrow L_{\kappa}^qE_n }=\|\mathcal{H}_n\|_{E_n\rightarrow X_{q,2}^{0,\gamma}(E_n) }=\|\lg\kappa\rg^{\gamma} \widetilde{H}_{n;\ell,k}(\kappa)\|_{l_{k}^2\rightarrow L_{\kappa}^ql_{\ell}^2 }\,.
	\]
	Since the matrix element satisfies 
	\[ \mathcal{F}_t(\e^{-it\lambda_n^2}(\mathcal{H}_n(t))^* )_{\ell,k}(-\kappa)=\widehat{\ov{H}}_{n;k,l}(-\kappa+\lambda_n^2)=\ov{\widehat{H}}_{n;k,\ell}(\kappa-\lambda_n^2)\,.
	\]
	we have
	\begin{align}
	 \|\mathcal{H}_n^*\|_{S_n^{q,\gamma,*}}=&\|\lg\kappa\rg^{\gamma}\mathcal{F}_t(\e^{-it\lambda_n^2}(\mathcal{H}_n(t))^*)(\kappa)\|_{E_n\rightarrow L_{\kappa}^qE_n}\notag \\
	 =& \|\lg\kappa\rg^{\gamma} 
	 \mathcal{F}_t(\e^{-it\lambda_n^2}(\mathcal{H}_n(t) )^* )(-\kappa)
	 \|_{E_n\rightarrow L_{\kappa}^qE_n}
	  \\ =&\|\lg\kappa\rg^{\gamma}\ov{\widetilde{H}}_{n;\ell,k}(\kappa) \|_{l_{\ell}^2\rightarrow L_{\kappa}^ql_k^2}. \label{Hn*norm}
	\end{align}
	
The $S_n^{q,\gamma}$ norm allows us to control the $X_{q,\infty}^{0,\gamma}(E_n)$ norm of $e_n^{N,\dag}$, while the $S_n^{q,\gamma,*}$ norm will be used in the probabilistic trilinear estimate for purely colored-Gaussian terms (see  Section \ref{sec:CCC}). In particular, we need the following type of pointwise bound:
	\begin{lemme}\label{typicalS_nnormbound}
		Assume that $\mathcal{T}_n(t)\in S_n^{q,\gamma,*}$, then for any  $f_n\in E_n$,
		\[ \sup_{x\in\mathbb{S}^2}\|\lg\kappa\rg^{\gamma}\widetilde{\mathcal{T}_n(t)(f_n)}(\kappa,x)\|_{L_{\kappa}^q}\leq \sqrt{2n+1}\|\mathcal{T}_n\|_{S_n^{q,\gamma,*}}\|f_n\|_{E_n}. 
		\]
	\end{lemme}
	\begin{proof}
		Given $f_n=\sum_{|k|\leq n}f_{n,k}\mathbf{b}_{n,k} \in E_n$ written in the basis $(\mathbf{b}_{n,k})$, we have
		\[ \mathcal{T}_n(t)(f_n)(x)=\sum_{|\ell|,|k|\leq n} T_{n;\ell,k}(t)\cdot f_{n,k}\mathbf{b}_{n,\ell}(x).
		\]
		Hence by Cauchy-Schwarz, for fixed $x\in\mathbb{S}^2$,
		\begin{align*}
			\|\lg\kappa\rg^{\gamma}\widetilde{\mathcal{T}_n(t)(f_n)}(\kappa,x)\|_{L_{\kappa}^q}=&\Big\|\sum_{|\ell|,|k|\leq n }\lg\kappa\rg^{\gamma}\cdot {\widetilde{T}}_{n;\ell,k}(\kappa)\cdot {f}_{n,k} {\mathbf{b}}_{n,\ell}(x)\Big\|_{L_{\kappa}^q}\\
			\leq & \|f_{n,k}\|_{l_k^2}\Big\|
			\Big(\sum_{|k|\leq n}
			\Big|\sum_{|\ell|\leq n }\lg\kappa\rg^{\gamma}\widetilde{T}_{n;\ell,k}(\kappa)\mathbf{b}_{n,\ell}(x)
			\Big|^2
			\Big)^{\frac{1}{2}}
			\Big\|_{L_{\kappa}^q}\\
			=& \|f_n\|_{E_n} 
			\Big\|
			\Big(\sum_{|k|\leq n}
			\Big|\sum_{|\ell|\leq n }\lg\kappa\rg^{\gamma}\overline{\widetilde{T}_{n;\ell,k}}(\kappa)\overline{\mathbf{b}_{n,\ell}}(x)
			\Big|^2
			\Big)^{\frac{1}{2}}
			\Big\|_{L_{\kappa}^q}.
		\end{align*}
We conclude by the following pointwise bound: 
		\begin{equation}
		\label{pointwiseproba}
		\Big\|
		\Big(\sum_{|k|\leq n}
		\Big|\sum_{|\ell|\leq n }\lg\kappa\rg^{\gamma}\overline{\widetilde{T}_{n;\ell,k}}(\kappa)\overline{\mathbf{b}_{n,\ell}}(x)
		\Big|^2
		\Big)^{\frac{1}{2}}
		\Big\|_{L_{\kappa}^q}\leq 
		\sqrt{2n+1}\|\lg\kappa\rg^{\gamma} (\widetilde{\mathcal{T}}_n(\kappa))^{*} \|_{E_n\rightarrow L_{\kappa}^qE_n}.
	\end{equation}
Indeed, by Weyl \eqref{Weyl}, the left hand side can be bounded by
\begin{align*}
 \|\lg\kappa\rg^{\gamma}\ov{\widetilde{T}}_{n;\ell,k}(\kappa)\|_{l_{\ell}^2\rightarrow L_{\kappa}^ql_k^2}\Big(\sum_{|\ell|\leq n}|\ov{\mathbf{b}_{n,\ell}}(x)|^2 \Big)^{\frac{1}{2}}= \sqrt{2n+1}\|\lg\kappa\rg^{\gamma}(\widetilde{\mathcal{T}}_n(\kappa))^{*} \|_{E_n\rightarrow L_{\kappa}^qE_n}\,.
\end{align*}
 This completes the proof.
	\end{proof}

	\section{Key induction steps}\label{sec:5}
	
	In this section, we present the precise statement of Theorem \ref{main:rigorous} and the induction steps. For the reason of clarity, we will only prove the convergence for dyadic sequences $(u_N)_{N\in 2^{\mathbb{N}}}$. Slight modification leads to the convergence of the full sequence, see \cite{DNY1} or \cite{Sun-Tz} for details.
	
	\subsection{Equations for the RAO and remainders}\label{formula}
	
	The Cauchy problem \eqref{ODE:rotation} can be reformulated as a linear equation for $\mathcal{H}_n^{N}(t)$:
	\begin{align}\label{Duhamel:Hn^N}
		\mathcal{H}_n^{N}(t)=\e^{-it\lambda_n^2}\mathrm{Id}_{E_n}+\frac{2}{i}\int_0^t\e^{-i(t-t')\lambda_n^2}\pi_n(\mathcal{H}_n^{N}(t')(\cdot)\cdot |u_{\frac{N}{2}}(t')|^{\diamond 2} )\d t'.
	\end{align}
	To solve the fixed-point problem for the norm $S_n^{q,\gamma}$, we need also to estimate the adjoint of the twisted-Fourier transform of $\mathcal{H}_n^{N}(t)$. Define the potential operator $\mathcal{P}^{N}_n$
	\begin{equation}
	\label{def:Pn} \mathcal{P}^{N}_n(t):F(t)\mapsto \pi_n(\pi_nF(t)\cdot |u_\frac{N}{2}(t)|^{\diamond 2})\,.
	\end{equation}
	The Duhamel formula for the adjoint operator reads:
	\begin{align}\label{adjointformula} 
		(\mathcal{H}^N_n(t))^*
		=&\e^{i\lambda_n^2 t}\mathrm{Id}_{E_n}-\frac{2}{i}\int_0^t \e^{i\lambda_n^2t'} \mathcal{P}^N_n(t')\circ \mathcal{H}_n^N(t')\circ(\mathcal{H}_n^N(t))^*\d t'.
	\end{align}
	To see this, using the unitary property of $\mathcal{H}_n^N(t)$, we write
	\begin{align*}
		(\mathcal{H}_n^{N}(t))^*=&\e^{i\lambda_n^2t}\mathrm{Id}_{E_n}+(\mathcal{H}_n^N(t))^*-\e^{i\lambda_n^2t}\mathcal{H}_n^N(t)(\mathcal{H}_n^N(t))^*\\
		=& \e^{i\lambda_n^2t}
		\mathrm{Id}_{E_n}+(\mathcal{H}_n^N(t))^*-\Big(\mathrm{Id}_{E_n}+\frac{2}{i}\int_0^t\e^{i\lambda_n^2t'}\mathcal{P}_n^N(t')\circ\mathcal{H}_n^N(t')\d t' \Big)(\mathcal{H}_n^N(t))^*\\
		=&\e^{i\lambda_n^2 t}\mathrm{Id}_{E_n}-\frac{2}{i}\int_0^t \e^{i\lambda_n^2t'} \mathcal{P}^N_n(t')\circ \mathcal{H}_n^N(t')\circ(\mathcal{H}_n^N(t))^*\d t'.
	\end{align*}
	
	\medskip

	Recalling \eqref{eq:wN}, the equation verified by $w_N$ reads
	\begin{align}\label{fixed-pointwN}
		w_N(t)=&2\mathcal{I}\mathcal{N}_{(0,1)}(w_N,u_{\frac{N}{2}},u_{\frac{N}{2}})+2\mathcal{I}\mathcal{N}_{[0,1]}(\psi_{ N}+w_N,u_{\frac{N}{2}},u_{\frac{N}{2}})\notag  \\+&\mathcal{I}\mathcal{N}(u_{\frac{N}{2}},\psi_{ N}+w_N,u_{\frac{N}{2}}) 
		+2\mathcal{I}\mathcal{N}(\psi_{ N}+w_N,\psi_{N}+w_N,u_{\frac{N}{2}})\notag \\+&\mathcal{I}\mathcal{N}(\psi_{ N}+w_N,u_{\frac{N}{2}},\psi_{N}+w_N)   +\mathcal{I}\mathcal{N}(\psi_{ N}+w_N),
	\end{align} 
where we recall that
\begin{align}\label{Nnr}
	\mathcal{N}_{[0,1]}(f,g,h):=\sum_{\substack{
		n,n_1,n_2,n_3\\
	n\neq n_1} }  \pi_n(\pi_{n_1}f\cdot \pi_{n_2}\ov{g}\diamond \pi_{n_3}h ).
\end{align}
	Let us analyze the structure of equation \eqref{fixed-pointwN}. 
	\begin{itemize}
	\item The source terms (e.g. with no term \(w_{N}\)) gain small factor \(N^{-s}\) in \(L^{2}(\S^{2})\)-type norms, where $s=\alpha-\frac{1}{2}-\frac{2}{q}-\delta-100\sigma$ was defined in \eqref{smallparameters}. There are several reasons for this: 
	\begin{enumerate}
	\item First, they contain at least one \(\psi_{N}\), which is localized at frequency \(\sim N\).
	\item When there is only one term \(\psi_{N}\), the interaction is always non-resonant (namely of type \(\mathcal{N}_{[0,1]}(\psi_{N},\ast,\ast)\) or \(\mathcal{N}(\ast,\psi_{N},\ast)\)). 
	\item Otherwise, there are at least two \(\psi_{N}\) in which case we leverage the probabilistic structure via conditional Wiener chaos estimates. 
	\end{enumerate}
		
	\item The other contributions contain a term \(w_{N}\). As usual in this type of perturbative analysis, the most relevant contributions are the linear ones. There are several ingredients used to control these interactions. First, the linear and bilinear stochastic objects belong to \(L^{\infty}(\S^{2})\), possibly after Wick ordering. Second, if deterministic components (e.g. \(w_{M}\) for some \(M\leq\frac{N}{2}\)) are involved, we use dispersive considerations as detailed in Section \ref{sec:det}.
	\end{itemize}
	
	\subsection{The extensions and local well-posedness statements}
	
	In practice, we have to restrict the random averaging operators and remainders on some short time interval of size $O(T)$ and extend these restrictions that are defined for all time. As our analysis involves the adjoint of RAO, the extension procedure becomes slightly more delicate than in \cite{DNY1}. These extensions will be defined inductively as follows:
	
	\medskip

	\begin{flushleft}
		
		\textbf{A. Initialization:} We fix a small number $T\in(0,\frac{1}{2})$ and a bump function $\chi\in C_c^{\infty}((-1,1))$ such that $\chi(t)\equiv 1$ for $|t|\leq \frac{1}{2}$.
		First for the given initial frequency $N_0=1$, we define 
		\[
		u_{N_0}^{\dag}(t)=\chi_T(t)u_{N_0}(t)\,,
		\]
		where $u_{N_0}$ is the solution of \eqref{NLS-N} with $N=N_{0}$. 
		For $0\leq n\leq N_0$, we define
			\[ \mathcal{H}_{n}^{N_0,\dag}(t)=\chi(t)\e^{-it(\lambda_{n}^{2}+1)}\mathrm{Id}_{E_{n}}\,,\quad \mathcal{G}_n^{N_0,\dag}(t):=\chi(t)\e^{it(\lambda_{n}^{2}+1)}\mathrm{Id}_{E_{n}}\,,
			\]
			and we set simply 
			\[
			\psi_{ N_{0}}^{\dag}(t):=\chi(t)\e^{it(\Delta-1)}\Pi_{N_{0}}\phi,\quad w_{N_{0}}^{\dag} = u_{N_{0}}^{\dag} - \psi_{N_{0}}^{\dag}.
			\]	
\end{flushleft}

\begin{flushleft}
			\textbf{B. Heredity:} Assuming that for some $N\geq N_0$, the objects  $\mathcal{H}_{m}^{M,\dag}(t),$ $\mathcal{G}_m^{M,\dag}(t)$,$w_M^{\dag}(t)$ and $\psi_M^{\dag}(t)$ are defined for $N_0\leq M\leq N$, compactly supported in $(-1,1)$ and they  coincide with $\mathcal{H}_m^M(t),(\mathcal{H}_m^M(t))^*, w_M(t), \psi_M(t)$ on $\big[-\frac{T}{2},\frac{T}{2}\big]$, respectively.

\medskip

	\textbf{Step 1: The potential operator:} For $n\approx2N$ (meaning that $N<n\leq2N$) we define the extension of the linear operator $\mathcal{P}_n^{2N,\dag}(t)$ via
	\begin{align}\label{def:P2Nextension} \mathcal{P}_n^{2N,\dag}(t)(F(t)):=\pi_n(\pi_nF(t)\cdot |u_{N}^{\dag}(t)|^{\diamond 2} ).
	\end{align}
	We have $\mathcal{P}_n^{2N,\dag}(t)=\mathcal{P}_n^{2N}(t)$ for $|t|\leq \frac{T}{2}$, and for all $t\in\R$, the operator $\mathcal{P}_n^{2N,\dag}(t)$ is self-adjoint on $E_n$. 
	\\
	\medskip

	\textbf{Step 2: Random averaging operators and colored Gaussians:} For $n\approx 2N$, we define~$\mathcal{H}_n^{2N,\dag}$  by the equation
	\begin{align}\label{def:H2Nextension} \mathcal{H}_n^{2N,\dag}(t)=\chi(t)\e^{-it\lambda_n^2}\mathrm{Id}_{E_n}+\frac{2}{i}\chi_{2T}(t)\int_0^t\e^{-i\lambda_n^2(t-t')}\mathcal{P}_n^{2N,\dag}(t')\circ\mathcal{H}_n^{2N,\dag}(t')dt'.
	\end{align}
	Note that the existence and uniqueness of \(\mathcal{H}_{n}^{2N,\dag}\) follows from the linear nature of the problem \eqref{def:H2Nextension}. In addition, for all $n\approx2N$ and $t\in\R$, the operator $\mathcal{H}_n^{2N,\dag}$ is $\mathcal{B}_{\leq N}$-measurable. Moreover, when $|t|\leq \frac{T}{2}$, we have
	\[
	\mathcal{H}_n^{2N,\dag}(t)=\mathcal{H}_n^{2N}(t).
	\]
	For the extension of the adjoint operator, instead of taking directly the adjoint, we introduce the operator $\mathcal{G}_n^{2N,\dag}(t)$ as the solution of the linear equation
	\begin{align}\label{def:G2N}
		\mathcal{G}_n^{2N,\dag}(t)
		=&\chi(t)\e^{i\lambda_n^2 t}\mathrm{Id}_{E_n}-\frac{2}{i}\chi_{2T}(t)\int_0^t \e^{i\lambda_n^2t'} \mathcal{P}^{2N,\dag}_n(t')\circ \mathcal{H}_n^{2N,\dag}(t')\circ\mathcal{G}_n^{2N,\dag}(t)dt'.
	\end{align}
	We observe that for $|t|\leq \frac{T}{2}$, 
	\[
	\mathcal{G}_n^{2N,\dag}(t)=(\mathcal{H}^{2N}_n(t))^*,
	\]
	but $\mathcal{H}_n^{2N,\dag}(t),\mathcal{G}_n^{2N,\dag}(t)$ differ from $\mathcal{H}_n^{2N}(t),(\mathcal{H}_n^{2N}(t))^*$ on $[-T,T]$. Nevertheless we deduce from Lemma \ref{lem:unitaryextend} that $\mathcal{H}_n^{2N,\dag}(t)$ is still unitary \footnote{Here we only need the fact that $\mathcal{P}_n^{2N,\dag}(t)$ is self-adjoint for all $t\in\R$. } on $[-T,T]$, and that
	\[
	\mathcal{G}_n^{2N,\dag}(t)=(\mathcal{H}_n^{2N,\dag}(t))^*\,.
	\]
	We split the perturbative components of $\mathcal{H}_n^{N,\dag}(t),\mathcal{G}_n^{N,\dag}(t)$ into
	\begin{equation}
	\label{def:h2Nextension} 
	\mathbf{h}_n^{2N,\dag}(t)
	:=\mathcal{H}_n^{2N,\dag}(t)
	-\chi(t)\e^{-i\lambda_n^2t}\mathrm{Id}_{E_n},\quad
	 \mathbf{g}_n^{2N,\dag}(t)
	 :=\mathcal{G}_n^{2N,\dag}(t)
	 -\chi(t)\e^{i\lambda_n^2t}\mathrm{Id}_{E_n}.
	\end{equation}
	For $|t|\leq T$, 
	\[
	\mathbf{g}_n^{2N,\dag}(t)=(\mathbf{h}_n^{2N,\dag}(t))^*
	\]
	and the operators $\mathbf{h}_n^{2N,\dag}(t),\mathbf{g}_n^{2N,\dag}(t)$ are compactly supported on $|t|\leq 2T$. Next, we define 
	\begin{equation}
	\label{def:e2Nextension} 
	e_n^{2N,\dag}(t):=
		\mathcal{H}_n^{2N,\dag}(t)(e_n^{\omega})
		,\quad 
	\psi_{2N}^{\dag}(t):=
		\chi_T(t)\sum_{n\approx2N}\frac{e_n^{2N,\dag}(t)}{\lambda_n^{\alpha-\frac{1}{2}}}.
	\end{equation}
We have from the above notations that for $n\approx2N$,
\begin{equation}
	\label{def:pinpsi2N}
	\pi_n\psi_{2N}^{\dag}(t)
	=\frac{\chi_T(t)\e^{-it\lambda_n^2}e_n^{\omega}
	+\mathbf{f}_n^{2N,\dag}(t) }{\lambda_n^{\alpha-\frac{1}{2}}}\,,
\end{equation}
where 
	\[
	\mathbf{f}_n^{2N,\dag}(t)=\chi_T(t)\mathbf{h}_n^{2N,\dag}(t)(e_n^{\omega})\,.
	\]
	We also define 
	\[
	\mathbf{f}_{2N}^{\dag}(t):=\sum_{N<n\leq 2N}\frac{\mathbf{f}_n^{2N,\dag}(t)}{\lambda_n^{\alpha-\frac{1}{2}}}\,.
	\]
	Note that $\mathbf{f}_{2N}^{\dag}(t)$ and $\psi_{2N}^{\dag}(t)$ are compactly supported on $|t|\leq T$ and coincide on $|t|\leq \frac{T}{2}$ with $\psi_{2N}(t)-\e^{it(\Delta-1)}\mathbf{P}_N\phi$ and $\psi_{2N}(t)$ respectively. We stress out that $\psi_{2N}^{\dag}(t)$ solves the equation
	\[ 
	(i\partial_t+\Delta-1)\psi_{2N}^{\dag}
	=
	2\Pi_{2N}\mathcal{N}_{(0,1)}(\psi_{2N}^{\dag},u_{N}^{\dag},u_N^{\dag})
	\]
	only on $[-T/2,T/2]$.
	\medskip

The advantage of defining $\psi_{2N}^{\dag}(t)$ only on $[-T,T]$ is that $\pi_n\psi_{2N}^{\dag}(t)=\chi_T(t)e_n^{2N,\dag}(t)$ and we could apply the invariance property (which is a global information) for $e_n^{2N,\dag}$ when solving the operator equation \eqref{def:H2Nextension} and the equation \eqref{eq:wNdag} for the remainder.   
   The motivation of separating the operator $\mathbf{h}_n^{2N,\dag}(t)$ from $\mathcal{H}_n^{2N,\dag}(t)$ is that, when global information of the colored Gaussian $\pi_n\psi_{2N}^{\dag}(t)$ is not needed, we decompose it as the sum of  
   	\[
	\lambda_n^{-\frac{1}{2}}\chi_T(t)\mathcal{T}_n^{2N,\dag}(t)(e_n^{\omega}), \quad 
 \mathcal{T}_n^{2N,\dag}\in \{\chi(t)\e^{-it\lambda_n^2}\mathrm{Id}_{E_n}, \; \varphi_T(t)\mathbf{h}_n^{2N,\dag}(t) \}\,,   
    	\]
where $\varphi_T(t)\chi_T(t)\mathbf{h}_n^{2N}(t)=\chi_T(t)\mathbf{h}_n^{2N,\dag}(t)$ and supp$(\varphi_T)\subset [-T,T]$. In practice, the output of our estimate for operators $\mathcal{T}_n^{2N}(t)$ will only depend on $\|\mathcal{T}_n^{2N}\|_{S_n^{q,\gamma}}$ and $\|(\mathcal{T}_n^{2N})^*\|_{S_n^{q,\gamma,*}}$ norms.

	\medskip
	
	We provide some large deviation bounds on the stochastic objects in Section \ref{sec:large-dev}. 
	
	\medskip

	\textbf{Step 3: Remainder: }
	Finally we define $w_{2N}^{\dag}$ by solving the integral equation\footnote{Here we truncate the Duhamel formula by $\chi_T(t)$ instead of $\chi_{2T}(t)$. The reason is that on supp$(\chi_T)$, the extended operators $(\mathcal{H}_n^{2N,\dag},\mathcal{G}_n^{2N,\dag})$ form an unitary pair.}
	\begin{align}
		w_{2N}^{\dag}(t)=&2\chi_T(t)\big[\mathcal{I}\mathcal{N}_{\mathrm{(0,1)}}(w^{\dag}_{2N},u_{N}^{\dag},u_{N}^{\dag})+\mathcal{I}\mathcal{N}_{[0,1]}(\psi_{2N}^{\dag}+w^{\dag}_{2N},u_{N}^{\dag},u_{N}^{\dag}) \big]  \label{eq:wNdag} \\+&\chi_T(t)\big[\mathcal{I}\mathcal{N}(u_N^{\dag},\psi_{2N}^{\dag}+w_{2N}^{\dag},u_N^{\dag}) 
		+2\mathcal{I}\mathcal{N}(\psi^{\dag}_{2N}+w^{\dag}_{2N},\psi^{\dag}_{2N}+w^{\dag}_{2N},u^{\dag}_{N})\big]\notag 
		\\+&\chi_T(t)\big[\mathcal{I}\mathcal{N}(\psi_{2N}^{\dag}+w_{2N}^{\dag},u^{\dag}_{N},\psi_{2N}^{\dag}+w_{2N}^{\dag})   +\mathcal{I}\mathcal{N}(\psi_{2N}^{\dag}+w^{\dag}_{2N})\big]. \notag  
	\end{align} 
 Note that \eqref{eq:wNdag} coincides with $\eqref{eq:wN}$ only on the interval $[-\frac{T}{2},\frac{T}{2}]$. At the end of this step, we define
	\[ u_{2N}^{\dag}(t):=u_{N}^{\dag}(t)+\psi_{2N}^{\dag}(t)+w_{2N}^{\dag}(t)\,,
	\]
\end{flushleft}
	and we denote 
	\[
	 \psi_{\leq 2N}^{\dag}:=\sum_{M\leq 2N}\psi_{M}^{\dag},\quad w_{\leq 2N}^{\dag}:=\sum_{M\leq 2N}w_M^{\dag}.
	\]
	We summarize the iterative scheme:
	\begin{itemize}
		\item The extensions of $u_N,\psi_{N}$,$w_N, \mathcal{H}_n^{2N}(t)$ and $(\mathcal{H}^{2N}_n(t))^*$ are denoted by
		\[
		u_{N}^{\dag},\,\psi_{N}^{\dag},\,
		w_{N}^{\dag},\mathcal{H}_n^{2N,\dag}(t),\,
		\mathcal{G}_n^{2N,\dag}(t)\,.
		\]
		\item The functions $u_N^{\dag},\psi_{N}^{\dag},w_N^{\dag}$ are supported in $(-T,T)$, while $\mathcal{H}_n^{2N,\dag}(t),\mathcal{G}_n^{2N}(t)$ are supported on $(-2T,2T)$. They coincide with the corresponding objects without $\dag$ on $[-\frac{T}{2},\frac{T}{2}]$.\\
		\item  The operators $(\mathcal{H}_n^{2N,\dag}(t),\mathcal{G}_n^{2N,\dag}(t))$ form a unitary pair when $|t|\leq T$.\\
		\item The random operators $\mathcal{H}_n^{N,\dag}(t),\mathcal{G}_n^{N,\dag}(t)$  are $\mathcal{B}_{\leq \frac{N}{2}}$ measurable, and they preserve the law of the Gaussian spherical harmonic $e_{n}^{\omega}$ for $|t|\leq T$ and 
		\(n\approx N\) (see Lemma~\ref{loiequiavlent}).
		\\
		\item The random functions $w_{\leq N}^{\dag},u_N^{\dag},\psi_{\leq N}^{\dag}$ are $\mathcal{B}_{\leq N}$ measurable.
	\end{itemize}
	\medskip
	We define	
\begin{equation}
\label{def:hexagonal}
 f\hexagonal{\neq}g:=\sum_{n_2\neq n_3}\pi_{n_2}f\cdot\pi_{n_3}g,\quad f\hexagonal{=}g:=\sum_{n}\big(\pi_{n}f\cdot\pi_ng-\lg\pi_n f|\pi_n g\rg \big)\,.
	\end{equation}
We have that
	\begin{multline}
	\label{def:hexagon'}
	\psi^{\dag}_{\leq N}\hexagonal{=}\overline{\psi^{\dag}_{\leq N}}(t,x) = \sum_{1\leq L \leq N}\psi^{\dag}_L\hexagonal{=}\overline{\psi^{\dag}_L}(t,x)
	= \sum_{1\leq L\leq N}\sum_{\frac{L}{2}\leq n< L}(|\pi_n\psi^{\dag}_L(t,x)|^2 - \|\pi_n\psi^{\dag}_L(t,\cdot)\|_{L^2}^2)
	\\
	=\chi_T(t)^2\sum_{1\leq L\leq N}\sum_{\frac{L}{2}\leq n < L}\lambda_n^{-2(\alpha-\frac{1}{2})}(|e^{L,\dag}_n(t,x)|^2-\|e_n^{L,\dag}(t)\|_{L^2}^2)\,.
	\end{multline}

\begin{remarque}
The contribution $\psi_{L}^{\dag}\hexagonal{=}\ov{\psi_L^{\dag}}$ will be roughly estimated in $L_{t,x}^{\infty}$ spaces (see Lemma \ref{lem:wick}) where we will make use of the fact that for \(n\approx L\), we have 
\[
\pi_n\psi_L^{\dag}=\lambda_n^{-(\alpha-\frac{1}{2})}\chi_T(t)\mathcal{H}_n^{L,\dag}(t)(e_n^{\omega})
\] 
and that $\mathcal{H}_n^{L,\dag}(t)$ is unitary when \(t\in \operatorname{supp}\chi_{T}\). 
\end{remarque}
	We now state an a priori parameter-dependent local bounds statement:
	\medskip
	
	\begin{definition}[$\mathrm{Loc}(N)$~\footnote{The statement Loc(\(N\)) depends on \(R\) and on the small parameter \(\sigma\). The parameter \(T=R^{-10}\) is only related to~\(R\).} ]\label{loc(N)}
		Let $N\geq 2$ be a dyadic number, $0<T\ll 1$, $R\geq 1$ and $0<\sigma\ll 1$, $(q,\gamma,\gamma_1,\delta,s,\theta,b)=(q_{\sigma},\gamma_{\sigma},\gamma_{1,\sigma},\delta_{\sigma},s_{\sigma},\theta_{\sigma},b_{\sigma})$ the $\sigma$-dependent parameters given by~\eqref{smallparameters}. The statement $\mathrm{Loc}(N)$ holds when for all $2\leq M\leq N$ the following properties are satisfied:
		\medskip

		\noindent If $N=2$, then with $T=\frac{1}{2}$,
		\[
		\|\psi_{\leq 2}^{\dag}\|=\|u_{2}^{\dag}\|_{L_t^{\infty}H_x^{10}}\leq R\,.
		\]
		If $N>2$, then $T= R^{-\frac{10}{\theta}}$ and:
		\begin{enumerate}
		\item For any $\frac{M}{2}<n\leq M$,
			\begin{align*}
			\|\mathbf{h}_n^{M,\dag}\|_{S_n^{q,\gamma}}+\|\mathbf{g}_n^{M,\dag}\|_{S_n^{q,\gamma,*}}
			&\leq R^{-1}\,,
			\intertext{and}
	 \|\mathcal{H}_n^{M,\dag}\|_{S_n^{q,\gamma}}+\|\mathcal{G}_{n}^{M,\dag}\|_{S_n^{q,\gamma,*}}
	 	&\leq R\,.
			\end{align*}
		\item Recalling the definition \eqref{def:hexagonal} of the symbol $\hexagonal{=}$,  $\psi_{M}^{\dag}$ satisfies 
		\begin{equation}
		\label{eq:p-psi}
		\|\psi_{M}^{\dag}\|_{L_{t}^{q}L_{x}^{\infty}} \leq RT^{\frac{1}{q}}M^{-(\alpha-1)+\frac{1}{q}+\delta}\,,
		\quad \|\psi_{M}^{\dag}\hexagonal{=}\overline{\psi_{M}^{\dag}}\|_{L_{t}^{q}L_{x}^{\infty}} 
		\leq R^{2}T^{\frac{1}{q}}M^{-2(\alpha-\frac{1}{2})+\frac{1}{2}+\frac{1}{q}+\delta}.
		\end{equation}
		In Fourier restriction spaces, 
		\begin{equation}
		\label{eq:p-psi-}
		\|\psi_{M}^{\dag}\|_{X_{q,q,\infty}^{0,\gamma}}
		\leq RT^{-\gamma+\frac{1}{q'}}
		M^{-(\alpha-\frac{1}{2})+\frac{2}{q}+\delta}\,.
		\end{equation}
		\item For all $M\leq N$ and $L\geq 2M$,
		\begin{equation}
		\label{eq:bo-r}
		 \|w_M^{\dag}\|_{X^{0,b}}\leq M^{-s}R^{-1},\quad \|\Pi_L^{\perp}w_M^{\dag} \|_{X^{0,b}}\leq \big(\frac{M}{L}\big)^{10}M^{-s}R^{-1}. 
		\end{equation}
		\end{enumerate}
	\end{definition}
\begin{remarque} Note that when $\mathrm{Loc}(N)$ holds with parameters $(T,R,\sigma)$, then we also have that for all $2\leq M\leq N$, 
\begin{equation}
\label{eq:chi-H}
\|(\chi_{T}\mathcal{H}_{n}^{M,\dag})^{\ast}\|_{S_{n}^{q,\gamma,\ast}} = \|\chi_{T}\mathcal{G}_{n}^{M,\dag}\|_{S_{n}^{q,\gamma,\ast}} \lesssim T^{\frac{1}{q'}-\gamma}\|\mathcal{G}_{n}^{M,\dag}\|_{S_{n}^{q,\gamma,\ast}} \lesssim T^{\frac{1}{q'}-\gamma}R\,.
\end{equation}
We used Proposition \ref{timelocalization} to obtain the first inequality. 
\end{remarque}

\begin{remarque}
Estimate \eqref{eq:bo-r} implies that the sequence \((w_{N}^{\dag})\) is uniformly bounded in~$X^{s,b}$. It is reminiscent of semi-classical energy estimates employed in \cite[Appendix A]{BGT_ENS} to show ill-posedness.  
\end{remarque}

\begin{remarque} Note that even if $T= R^{-\frac{10}{\theta}}$, we keep writing the $T$-dependence in the estimates. 
\end{remarque}

	We now state the key induction Proposition. 
	\begin{proposition}
		[Key induction Proposition]\label{keyinduction}
	Assume that $\alpha>1$. There exist a sufficiently small $\sigma\in(0,1)$ and accordingly the parameters $(q,\gamma,\delta,s,\theta,b)=(q_{\sigma},\gamma_{\sigma},\delta_{\sigma},s_{\sigma},\theta_{\sigma},b)$ given by~\eqref{smallparameters},  $C_0>1>c_0>0$, depending on parameters $\sigma,\alpha$, such that for any $R\geq1$ sufficiently large and $T=R^{-\frac{10}{\theta}}$, the following statement holds:
\medskip

\begin{adjustwidth}{0.5cm}{0.5cm}	 If $\; \mathrm{Loc}(N)$ holds for any $\omega$ in a $\mathcal{B}_{\leq N}$-measurable set $\Omega_{N}$, then $\mathrm{Loc}(2N)$ holds for $\omega\in \Omega_{2N}$, a $\mathcal{B}_{\leq 2N}$-measurable set such that 
	\[
	\mathbb{P}(\Omega_N\setminus\Omega_{2N})<C_0\e^{-N^{c_0}R^{\delta}}\,.
	\]
\end{adjustwidth}
	\end{proposition}
	Note that the parameters $(\sigma,q,\gamma,\delta,s,\theta,b,C_0,c_0)$ are indeed absolute constants, and the key induction property does not require any constraint on $R$ which measures the large deviation of the desired events.  Consequently, we have:
	\begin{corollaire}\label{local well-posedness}
		Let $(\sigma,q,\gamma,\delta,s,\theta, b,C_0,c_0)$ be the same set of parameter in Proposition~\ref{keyinduction}.
	Then for any sufficiently large $R\geq 1$, $\mathrm{Loc}(N)$ holds for all $N\in 2^{\mathbb{N}}$, outside a set of probability measure $C_0\e^{-cR^{\delta_0} }$, where $C_0,c,\delta_0>0$ are absolute constants depending only on the single parameter $\sigma$.
	\end{corollaire}

		
	\begin{figure}
	\begin{center}
		\begin{tikzpicture}[node distance=2cm and 2cm]
			\node (start) [startstop] {Start:  Loc($N$)};
			\node (in1) [io, below=1cm of start, text width=5.5cm, align=left] {
  \textbf{Inputs:}\\ \medskip
  
--  \(\quad u_N^{\dag} = \psi_{\leq N}^{\dag} + w_{\leq N}^{\dag}\)\\ \medskip

--  \(\quad \mathcal{H}_n^{N,\dag}:\, \eqref{def:H2Nextension}, \quad \mathcal{G}_n^{N,\dag}:\, \eqref{def:G2N}\)
};			\node (pro1) [process, below=1cm of in1] {$\mathcal{P}_{2N}^{\dag}$: \eqref{def:P2Nextension}};
			\node (pro2) [process, left=.7cm of pro1, yshift=-1cm] {
					\(\mathcal{H}_n^{2N,\dag}:\,\eqref{def:H2Nextension},\quad \mathcal{G}_n^{2N,\dag}:\, \eqref{def:G2N}\)
			};
			\node (pro3) [process, below=1cm of pro1] {$\psi_{2N}^{\dag}$: \eqref{def:e2Nextension}};
			\node (pro4) [process, below=1cm of pro3] {
				$w_{2N}^{\dag}$: \eqref{eq:wNdag}
			};
			\node (stop)
			[startstop, below=1cm of pro4]
			{ Renew $N=2N$
			};
			
			\draw [arrow] (start) -- (in1);
			\draw [arrow] (in1) -- (pro1);
			\draw [arrow] (pro1) -| (pro2);
			\draw [arrow] (pro2) |- (pro3);
			\draw [arrow] (pro3) -- (pro4);
			\draw [arrow] (pro4) -- (stop);
			\draw [arrow] (stop) to [out=1, in=-1] (start);
		\end{tikzpicture}
	\end{center}
\caption{Induction scheme from scale \(N\) to \(2N\)..}
	\end{figure}
	
	\subsection{Key estimates for random averaging operators (RAO)} 
	
	The proof of Proposition \ref{keyinduction} relies on a set of multi-linear estimates for functions and operators, both of deterministic type and stochastic type. We begin with deterministic estimates for the Duhamel integral of related RAOs that cover the range $\alpha\geq 1$. 
	Recall that in the RAO ansatz,
		\[
		 u^{\dag}_N(t)=\psi^{\dag}_{\leq N}(t)+w^{\dag}_{\leq N}(t),\quad w^{\dag}_{\leq N}(t)\in X^{s_0,b},\; \psi_{\leq N}^{\dag}\in  X_{q,q,\infty}^{\alpha-\frac{1}{2}-\frac{1}{q}-3\delta,\gamma_1}
		\]
		for $s_0=s_{\sigma}-\sigma,\ \delta=\delta_{\sigma}$ and $\gamma_1=\gamma_{1,\sigma}\in(\frac{1}{q'},1)$.
		To estimate the norm of $\mathcal{P}^{2N,\dag}_n(t)$, we expand $|u^{\dag}_N(t)|^{\diamond 2}$, the wick square of $u^{\dag}_N$ by
		\begin{multline}\label{Wicksquare} 
		|u_{N}(t)|^{\diamond2} := |u^{\dag}_N(t)|^2-\|u^{\dag}_N(t)\|_{L_x^2}^2=|\psi^{\dag}_{\leq N}(t)|^{\diamond 2}+|w^{\dag}_{\leq N}(t)|^{\diamond 2} \\
			+2\Re(\psi^{\dag}_{\leq N}(t)\ov{w}^{\dag}_{\leq N}(t))-2\Re\lg\psi^{\dag}_{\leq N}(t) |w_{\leq N}^{\dag}(t)\rg,
		\end{multline}
		where 
		\[
		 |\psi^{\dag}_{\leq N}(t)|^{\diamond 2}:=|\psi^{\dag}_{\leq N}(t)|^2-\|\psi^{\dag}_{\leq N}(t)\|_{L_x^2}^2, \quad 
		|w_{\leq N}^{\dag}(t)|^{\diamond 2}=|w_{\leq N}^{\dag}(t)|^2-\|w_{\leq N}^{\dag}(t)\|_{L^2}^2.
		\]
		We further split
		\begin{align}\label{splitingwicksquare} 
			|\psi^{\dag}_{\leq N}(t)|^{\diamond 2}=\ov{\psi^{\dag}}_{\leq N}(t)\hexagonal{\neq} \psi^{\dag}_{\leq N}(t)+\ov{\psi^{\dag}}_{\leq N}(t)\hexagonal{=}\psi^{\dag}_{\leq N}(t).
		\end{align}
	We note that 
	\[
	\Pi_N\psi^{\dag}_{\leq N}=\psi_{\leq N}^{\dag},\quad  \Pi_Nw_{\leq N}^{\dag}\neq w_{\leq N}^{\dag}\,.
	\]
	Let us start with the estimates for $S_n^{q,\gamma}$ norm of the RAOs $\mathbf{h}_n^{2N,\dag}$ and $\mathbf{g}_n^{2N,\dag}$ in the fixed-point problems:
	\begin{proposition}\label{multilinear:op1}
		There exists a sufficiently small $\sigma\in(0,1)$, such that if the statement $\mathrm{Loc}(N)$ holds with respect to arguments $(T,R,\sigma)$, then 
		for any $N<n\leq 2N$ and  $\mathcal{T}_n\in S_n^{q,\gamma}$, 
		\begin{align*}
			\Big\|\chi_{2T}(t)\int_0^t\e^{-i\lambda_n^2(t-t')}\mathcal{P}_n^{2N,\dag}(t')\circ\mathcal{T}_n(t')\d t' \Big\|_{S_n^{q,\gamma}}
			\lesssim CT^{\theta}\|\mathcal{T}_n\|_{S_n^{q,\gamma}}
			\cdot 
					T^{-2(\gamma-\frac{1}{q'})}\			R^2.
		\end{align*}
	\end{proposition}

\begin{proof}
Thanks to Proposition \ref{inhomogeneous}, the definition of the norm $S_n^{q,\gamma}$ as well as Loc($N$), it suffices to prove
\begin{multline}\label{eq:multilinear:op1-1}
\|\mathcal{P}_n^{2N,\dag}(t)(F_n(t))\|_{X_{q,2}^{0,\gamma-1+\theta}(E_n)}
	\\
	\lesssim 
	\|F_n\|_{X_{q,2}^{0,\gamma}(E_n)}
	\big(\| \ov{\psi}^{\dag}_{\leq N}\hexagonal{=} \psi^{\dag}_{\leq N}\|_{L_{t}^{q'}L_{x}^{\infty} }  +\|\psi^{\dag}_{\leq N}\|_{X_{q,q,\infty}^{\beta,\gamma}}^2+\|w^{\dag}_{\leq N}\|_{X^{s,b} }^2
	\big),
\end{multline}
where $\beta=\frac{1}{2}-\frac{2}{q}-3\delta$. Indeed, for any $L^2$-normalized function $f_n\in E_n$ and $F_n=\mathcal{T}_n(t)(f_n)$ by the bounds \eqref{eq:p-psi}, \eqref{eq:p-psi-}, the right hand side of \eqref{eq:multilinear:op1-1} is bounded by
\begin{align*}
\|\mathcal{T}_n\|_{S_{n}^{q,\gamma}} R^2\sum_{M\leq N}\big(M^{-2(\alpha-\frac{1}{2})+\frac{1}{2}+\frac{1}{q}+\delta} +T^{-2(\gamma-\frac{1}{q'})}\cdot M^{-2(\alpha-\frac{1}{2}-\frac{2}{q}-\delta)}\cdot M^{2\beta}+M^{-2s} \big),
\end{align*}
which is bounded by $T^{-2(\gamma-\frac{1}{q'})}R^2\|\mathcal{T}_n\|_{S_n^{q,\gamma}}$ as desired.

To prove \eqref{eq:multilinear:op1-1}, we decompose $\mathcal{P}_n^{2N}(t)(F_n)$ into $\mathrm{I}+\mathrm{II}+\mathrm{III}+\mathrm{IV}$, where
\begin{align*} 
	&\mathrm{I}:=
	\e^{-i\lambda_n^2t}\pi_n(\e^{i\lambda_n^2t}\pi_nF_n(t)\ov{\psi}^{\dag}_{\leq N}\hexagonal{=} \psi^{\dag}_{\leq N} ),
	\\ 
	&\mathrm{II}:=
	\e^{-i\lambda_n^2t}\pi_n(\e^{i\lambda_n^2t}\pi_nF_n(t)\Re(\psi^{\dag}_{\leq N}(t)\ov{w^{\dag}}_{\leq N}(t) ))
	-\Re
	\lg \psi_{\leq N}^{\dag}(t) | w_{\leq N}^{\dag}(t) \rg
	\cdot  \pi_nF_n(t) ,
	\\
	&\mathrm{III}:=
	\e^{-i\lambda_n^2t}\pi_n(\e^{i\lambda_n^2t}\pi_nF_n(t)|w^{\dag}_{\leq N}(t)|^{2} )-\|w^{\dag}_{\leq N}(t)\|_{L^2}^2\cdot \pi_nF_n(t) ,\\
	&\mathrm{IV}:=
	\e^{-i\lambda_n^2t}\pi_n(\e^{i\lambda_n^2t}\pi_nF_n(t)\ov{\psi}^{\dag}_{\leq N}\hexagonal{\ne}\psi^{\dag}_{\leq N} ).
\end{align*}
Using that $\gamma-1+\theta<0$ (see \eqref{smallparameters}) and then Lemma \ref{vectorvaluedlemma} with the Hilbert space $E_n$, we get
	\begin{align*}
	\|\mathrm{I}\|_{X_{q,2}^{0,\gamma-1+\theta}(E_n)}
	\leq 
	&\|\mathrm{I}\|_{X_{q,2}^{0,0}(E_n)}=\big\|\mathcal{F}_t\big(\pi_n(\e^{i\lambda_n^2t}\pi_nF_n(t)\ov{\psi^{\dag}}_{\leq N}\hexagonal{=} \psi^{\dag}_{\leq N} )\big)(\lambda) \big\|_{L_{\lambda}^qE_n}
	\\
	\lesssim
	&
	\big\|
	\e^{i\lambda_n^2t}\pi_n(\pi_nF_n(t)\ov{\psi^{\dag}}_{\leq N}\hexagonal{=} \psi^{\dag}_{\leq N} )
	\big\|_{L_{t}^{q'}E_n}
	\\
	\lesssim
	&\|F_n\|_{L_t^{\infty}L_x^2}\|\ov{\psi^{\dag}}_{\leq N}\hexagonal{=} \psi^{\dag}_{\leq N} \|_{L_t^{q'}L_x^{\infty}}
	\\
	\lesssim
	&\|F_n\|_{X_{q,2}^{0,\gamma}(E_n)}\|\ov{\psi^{\dag}}_{\leq N}\hexagonal{=} \psi^{\dag}_{\leq N} \|_{L_t^{q'}L_x^{\infty}},
\end{align*}
where in the last step, we used Lemma \ref{Embedding}, thanks to the fact that $\gamma\in\big(\frac{1}{q'},1\big)$.
\medskip

The estimate of the first contribution in II is a direct consequence of Proposition~\ref{lem:resonanttrilinear}. Moreover, for fixed $t$, using the embedding property 
\[
X^{s,b}\hookrightarrow C_tH_x^s\,,\quad X_{q,\infty}^{0,\gamma}(E_n)\hookrightarrow C_tL^{\infty}(E_n)\,,
\]
we obtain
\begin{align*}
 \big|\big\lg \psi_{\leq N}^{\dag}(t) | w_{\leq N}^{\dag} (t)
 \big\rg\big|\leq  &
 \sum_{n\leq N} \|\pi_n\psi_{\leq N}^{\dag}\|_{L_t^{\infty}L_x^{\infty}(E_n)}\|\pi_nw_{\leq N}^{\dag}\|_{L_t^{\infty}L_x^2}\\
 \leq & \sum_{n\leq N} \langle n\rangle^{-(s+\beta )}
 \|\langle n\rangle^{\beta} \pi_n\psi_{\leq N}^{\dag}\|_{X_{q,\infty}^{0,\gamma}(E_n)} \|\langle n\rangle^s\pi_nw_{\leq N}^{\dag}\|_{X^{0,b}}\\
 \lesssim & \|\psi_{\leq N}^{\dag}\|_{X_{q,q,\infty}^{\beta,\gamma}}\|w_{\leq N}^{\dag}\|_{X^{s,b}}.
\end{align*}
Hence for the second term of II, its $X_{q,2}^{0,\gamma-1+\theta}(E_n)$ norm can be bounded by its $L_{t,x}^2$ norm, leading to the bound 
	\[
	\|F_n\|_{X_{q,2}^{0,\gamma}(E_n)}\|\psi_{\leq N}^{\dag}\|_{X_{q,q,\infty}^{\beta,\gamma}}\|  w_{\leq N}^{\dag}
\|_{X^{s,b}}\,.
	\]
The estimate of the first term in III is implied by Proposition \ref{lem:resonanttrilinear'}. For the second term in III, we first control it by the norm $X_{q,2}^{0,0}(E_n)$ and then by Hausdorff-Young (Corollary \ref{Hausdorff-Young}), we have
	\begin{align*}
	\big\| \|w^{\dag}_{\leq N}(t)\|_{L_x^2}^2 \pi_nF(t)
	\big\|_{X_{q,2}^{0,0}(E_n)}\lesssim & 
	\big\|  \|w^{\dag}_{\leq N}(t)\|_{L_x^2}^2\cdot \|F_n(t)\|_{L_x^2}
	\big\|_{L_t^{q'}}\\
	\leq & \|F_n\|_{L_t^2L_x^2}
	\|w^{\dag}_{\leq N}  \|_{L^{\frac{4q}{q-2}	
	  }_tL_x^2}^2\\
  \lesssim & \|\pi_nF_n\|_{X_{q,2}^{0,\gamma}(E_n)}\|w^{\dag}_{\leq N}\|_{X^{0,b}}^2,
	\end{align*}
since $w_{\leq N}^{\dag}$ is compactly supported in time.
	Finally, the estimate for IV is implied by Proposition \ref{lem:resonanttrilinear''}. This proves \eqref{eq:multilinear:op1-1} and completes the proof of  Proposition \ref{multilinear:op1}.
\end{proof}

	A similar proposition claims that the Duhamel operator above can be extended as a bounded operator on $X_{q,2}^{q,\gamma}(E_n)$. This will allow us to control the equation for the operator $\mathcal{G}_n^{2N,\dag}$.

	\begin{proposition}\label{multilinear:op2}
		There exists a sufficiently small $\sigma\in(0,1)$, such that if the statement $\mathrm{Loc}(N)$ holds with respect to arguments $(T,R,\sigma)$, then for any $N<n\leq 2N$, $\mathcal{T}_n\in S_n^{q,\gamma}$ and $\mathcal{S}_n\in S_n^{q,\gamma,\ast}$, 
		\begin{align*}
			\Big\|\chi_{2T}(t)\Big(\int_0^t\e^{i\lambda_n^2t'}\mathcal{P}_n^{2N,\dag}(t')\circ&\mathcal{T}_n(t')\d t'\Big)\circ \mathcal{S}_n(t) \Big\|_{S_n^{q,\gamma,*}}\\
			\lesssim & T^{\theta-\gamma+\frac{1}{q'}}\|\mathcal{T}_n\|_{S_n^{q,\gamma}}
			\|\mathcal{S}_n\|_{S_n^{q,\gamma,*}}\cdot 
					T^{-2\gamma+\frac{2}{q'}}
			R^2.
		\end{align*}
	\end{proposition}

\begin{proof}
By definition \eqref{Sn*norm'}, we need to show that for any $f_n\in E_n$,
\begin{align*}
&	\Big\|\chi_{2T}(t)\Big(\int_0^t\mathcal{P}_n^{2N,\dag}(t')\circ \e^{i\lambda_n^2t'}\mathcal{T}_n(t')\d t'\Big)\circ \e^{-i\lambda_n^2t}\mathcal{S}_n(t)(f_n) \Big\|_{\mathcal{F}L_q^{\gamma}(E_n)  }\\
\lesssim 
&T^{\theta-\gamma+\frac{1}{q'}}
\|\mathcal{T}_n\|_{S_n^{q,\gamma}}
\|\e^{-i\lambda_n^2t}\mathcal{S}_n(t)(f_n)\|_{\mathcal{F}L_q^{\gamma}(E_n) }
(
		T^{-2\gamma+\frac{2}{q'}}
R^2+\|\ov{\psi^{\dag}_{\leq N}}\hexagonal{=}\psi^{\dag}_{\leq N}\|_{L_t^qL_{x}^{\infty}}),
\end{align*}
where we refer to \eqref{def:FLabstract} for the definition of Fourier-Lebesgue spaces.
Set 
	\[
	\mathbf{A}(t')=\e^{i\lambda_n^2t'}\mathcal{P}_n^{2N,\dag}(t')\circ\mathcal{T}_n(t'), \quad \mathbf{F}(t)=\e^{-i\lambda_n^2t}\mathcal{S}_n(t)(f_n).
	\]
By Lemma \ref{inhomogeneousabstract2},
\begin{align*}
\Big\|\chi_{2T}(t)\int_0^t\mathbf{A}(t')(\mathbf{F}(t))\d t'\Big\|_{\mathcal{F}L_q^{\gamma}(E_n)}\lesssim T^{\theta-\gamma+\frac{1}{q'}}\|\mathbf{A}(t) \|_{E_n\rightarrow \mathcal{F}L_q^{\gamma-1+\theta}(E_n)}\|\mathbf{F}(t)\|_{\mathcal{F}L_q^{\gamma}(E_n)}.
\end{align*}
As a consequence of \eqref{eq:multilinear:op1-1},
we have
\begin{align*} \|\mathbf{A}(t)\|_{E_n\rightarrow \mathcal{F}L_q^{\gamma-1+\theta}(E_n)}
\lesssim  &\|\mathcal{T}_n\|_{S_n^{q,\gamma}}\big(  \| \ov{\psi}^{\dag}_{\leq N}\hexagonal{=} \psi^{\dag}_{\leq N}\|_{L_{t}^{q'}L_{x}^{\infty} }  +\|\psi^{\dag}_{\leq N}\|_{X_{q,q,\infty}^{\beta,\gamma}}^2+\|w^{\dag}_{\leq N}\|_{X^{s,b} }^2 \big)\\
\lesssim &
\|\mathcal{T}_n\|_{S_n^{q,\gamma}}(
		T^{-2\gamma+\frac{2}{q'}}
R^2+  \| \ov{\psi}^{\dag}_{\leq N}\hexagonal{=} \psi^{\dag}_{\leq N}\|_{L_{t}^{q}L_{x}^{\infty} } ).
\end{align*}
This completes the proof of Proposition \ref{multilinear:op2}.
\end{proof}

\subsection{Key trilinear estimates}	
	Next we state multi-linear estimates for remainders. We separate them into different classes. 
	\begin{definition}\label{Type}
		We say that a space-time function\footnote{Here the notation $v_{2N}$ should not be confused with the $u_{2N}-u_{N}$ in the previous section. } $v_{2N}$ is of type $\mathrm{(C)}$ if 
		\[
		v_{2N}=\chi_{T}(t)\sum_{N<n\leq 2N}\frac{\mathcal{T}_n^{2N}(t)(e_n^{\omega}) }{\lambda_{n}^{\alpha-\frac{1}{2}}}
		=:\chi_{T}(t)\mathcal{T}^{2N}(t)(\phi_{2N}^{\omega})\,,
		\]
		where, for all \(N<n\leq 2N\), the operator $\mathcal{T}_{n}^{2N}\in S_n^{q,\gamma}$ is $\mathcal{B}_{\leq N}$-measurable and satisfies
		\[
		\| \mathcal{T}_n^{2N}\|_{S_n^{q,\gamma}}\leq R\,,
		\]
		and if $\mathcal{T}^{2N}:=\sum_{N<n\leq2N}\mathcal{T}_{n}^{2N}$ satisfies
		\[
		\|\mathcal{T}^{2N}(t)(\phi_{2N}^{\omega}) \|_{X_{q,q,\infty}^{0,\gamma}}\leq (2N)^{-\alpha+\frac{1}{2}+\frac{2}{q}+\delta}
		R.
		\]
		It is of type $\mathrm{(D)}$ if $v_{2N}=z_{2N}\in X^{0,b}$, if
		\begin{equation}
		\label{eq:tD}
		\|z_{2N}\|_{X^{0,b}}\leq (2N)^{-s}R^{-1}\,,
		\end{equation}
		and if for all $L\geq4N$, 
		\[
		\|\Pi_{L}^{\perp}z_{2N}\|_{X^{0,b}}\leq (\frac{2N}{L})^{10}(2N)^{-s}R^{-1}\,.
		\]
	\end{definition}
	
	\begin{remarque}
	In the rest of the text, the functions of type \(\mathrm(C)\) or \(\mathrm(D)\) only satisfy the bounds of Definition \ref{Type} when restricted to some sets of large probability measure, precisely quantified with respect to the parameter \(R\), \(N\). The intersection of these sets with \(\Omega_{N}\) constitute the set \(\Omega_{2N}\) appearing in Proposition~\ref{keyinduction}. 
	
	\end{remarque}

A typical type (C) term is $\psi_{2N}^{\dag}$ for which the corresponding operator is $\mathcal{H}_n^{2N,\dag}$. We enlarge the definition of type (C) term to include RAO $\mathbf{h}_n^{2N,\dag}$ and $\chi(t)\e^{-it\lambda_n^2}\mathrm{Id}_{E_n}$ appearing in the decomposition of $\mathcal{H}_n^{2N,\dag}$. 
A typical type (D) term is the remainder~$w_{2N}^{\dag}$, solution to \eqref{eq:wNdag}. However, we introduce the notation to distinguish $w_{2N}^{\dag}$. In practice, knowing $\mathrm{Loc}(N)$, we will solve the fixed-point problem for $w_{2N}^{\dag}$ in the ball $\{w:\; \|w\|_{X^{0,b}}\leq (2N)^{-s} \}$. For this reason, it is more convenient to establish estimates for general type (D) terms.

\medskip

		We begin with some deterministic estimates that allow to control the resonant high-low-low component:
	
	\begin{proposition}\label{multilinear:det-1}
		There exists a sufficiently small $\sigma\in(0,1)$, such that if the statement $\mathrm{Loc}(N)$ holds with arguments $(T,R,\sigma)$, then for any $N<n\leq 2N$,
		\begin{align*}
		\|\chi_T(t)\mathcal{I}\mathcal{P}_n^{2N,\dag}\|_{\mathcal{L}(X^{0,b})}
		&\lesssim T^{\theta-2(\gamma-\frac{1}{q'}) }
		R^2\,.
		\intertext{ 
		Namely, for any $z_{2N}\in X^{0,b}$,}
		 \|\chi_T(t)\mathcal{I}\mathcal{N}_{(0,1)}(z_{2N},u_N^{\dag},u_N^{\dag})\|_{X^{0,b}}
		 &\lesssim T^{\theta-2(\gamma-\frac{1}{q'})}\|z_{2N}\|_{X^{0,b}}R^2\,.
		\end{align*}
	\end{proposition}
	Proposition \ref{multilinear:det-1} allows to control the first term of \eqref{eq:wNdag}, and the proof follows in the same way as for Proposition \ref{multilinear:op1}, hence we omit it. 
	\medskip

	Let us analyze the remaining type of nonlinearities in the equation of $w_{2N}^{\dag}$. In what follows $N_{(1)}=2N$. 
	\begin{itemize}
		\item (C)(C)(C) type $\mathcal{N}(\psi_{N_1}^\dag,\psi_{N_2}^{\dag},\psi_{N_3}^{\dag})$: 
		\medskip
		
		\begin{itemize}
			\item If $N_1=2N, N_2,N_3\leq N$, only $\mathcal{N}_{[0,1]}(\psi_{N_1}^{\dag},\psi_{N_2}^{\dag},\psi_{N_3}^{\dag})$ are presented.
			
			\item Otherwise $N_2=2N$ or $N_{(1)}=N_{(2)}=2N$, all possible combinations.
		\end{itemize}
		\medskip
		
		\item (C)(D)(D) types  $\mathcal{N}(\psi_{N_1}^{\dag},z_{N_2}^{\dag},z_{N_3}^{\dag}),
		\mathcal{N}(z_{N_1}^{\dag},\psi_{N_2}^{\dag},z_{N_3}^{\dag})$ and $\mathcal{N}(z_{N_1}^{\dag},z_{N_2}^{\dag},\psi_{N_3}^{\dag})$:
		\medskip
		
		\begin{itemize}
			\item If $N_1=2N, N_2,N_3\leq N$, only   $\mathcal{N}_{[0,1]}(\psi^{\dag}_{N_1},z^{\dag}_{N_2},z^{\dag}_{N_3})$ 
			\item Otherwise $N_2=2N$ or $N_{(1)}=N_{(2)}=2N$, all possible combinations.
		\end{itemize}
		\medskip
		
		\item (C)(C)(D) types $\mathcal{N}(\psi_{N_1}^{\dag},\psi_{N_2}^{\dag},z_{N_3}^{\dag})$, $\mathcal{N}(\psi_{N_1}^{\dag},z_{N_2}^{\dag},\psi_{N_3}^{\dag})$ and $\mathcal{N}(z_{N_1}^{\dag},\psi_{N_2}^{\dag},\psi_{N_3}^{\dag})$:
		\medskip
		
		\begin{itemize}
			\item If $N_1=2N, N_2,N_3\leq N$, only $\mathcal{N}_{[0,1]}(z_{N_1},z_{N_2},z_{N_3})$.
			\item Otherwise $N_2=2N$ or $N_{(1)}=N_{(2)}=2N$, all possible combinations. 
		\end{itemize}
		\medskip
		
		\item (D)(D)(D) types $\mathcal{N}(z_{N_1}^{\dag},z_{N_2}^{\dag},z_{N_3}^{\dag})$: all possible combinations.
	\end{itemize}
	\medskip
	
	We collect the deterministic trilinear estimates in the following proposition, which treats (D)(D)(C), (D)(C)(C) and (D)(D)(D) terms together. Note that we do not cover the case of the Gibbs measure $(\alpha=1)$ here, and refined probabilistic estimates together with a modified resolution scheme are needed. These will be developed in our forthcoming work \cite{BCST2}.  
	\begin{proposition}\label{multilinear:det-2}
	Let $\alpha>1$,  \(\sigma\) small to ensure the bound~\eqref{smallparameters3}, and
	\[
	0<\epsilon_{0}<\min(\frac{1}{2}(s-\frac{1}{2}),\alpha-1-\frac{2}{q}-\delta, 100\sigma)\,.
	\]
	We suppose that $\mathrm{Loc}(N)$ is true with parameters $(T,R,\sigma)$ and that $\psi_{2N}^{\dag}$ also satisfies the bound \eqref{eq:p-psi}. Let $N_{1},N_{2},N_{3}$ be dyadic integers such that $2N=N_{(1)}\geq N_{(2)}\geq N_{(3)}$, and $v_{N_{j}}^{\dag}$ is either of type $\mathrm{(C)}$ or of type $\mathrm{(D)}$, with at least one of type $\mathrm{(D)}$. We have:
	\begin{itemize}
	\item If there is at least one term $z_{2N}^{\dag}$ then 
	\[
	\|\chi_{T}\mathcal{I}\mathcal{N}(z_{2N}^{\dag},v_{N_{2}}^{\dag},v_{N_{3}}^{\dag})\|_{X^{0,b}} 
	+
	\|\chi_{T}\mathcal{I}\mathcal{N}(v_{N_{1}}^{\dag},z_{2N}^{\dag},v_{N_{3}}^{\dag})\|_{X^{0,b}}
	\leq 
	CT^{\frac{\theta}{2}}R
	\|z_{2N}^{\dag}\|_{X^{0,b}}N_{(2)}^{-\epsilon_{0}}\,.
	\]
	\item Otherwise there is at least one term $\psi_{2N}^{\dag}$, and
	\begin{multline*}
	\|\chi_{T}\mathcal{I}\mathcal{N}(v_{N_{1}}^{\dag},\psi_{2N}^{\dag},v_{N_{3}}^{\dag})\|_{X^{0,b}}
	+
	\|
	\chi_{T}\mathcal{I}\mathcal{N}_{[0,1]}(\psi_{2N}^{\dag},v_{N_{2}}^{\dag},
	v_{N_{3}}^{\dag})
	\|_{X^{0,b}}
	\leq CT^{\frac{\theta}{2}}RN^{-s}N_{(2)}^{-\epsilon_{0}}\,.
	\end{multline*}
 	\item For the interactions $\mathcal{N}_{(0,1)}(\psi_{2N}^{\dag},\cdot,\cdot)$ we suppose that $N_{(2)}=N_{(1)}=2N$, and we get  
  \begin{align*}
  	\|\chi_{T}\mathcal{I}\mathcal{N}_{(0,1)}(\psi_{2N}^{\dag},v_{N_{2}}^{\dag},v_{N_{3}}^{\dag})\|_{X^{0,b}} \leq CT^{\frac{\theta}{2}}RN^{-s-\epsilon_{0}}.
  \end{align*}
	\end{itemize}
 
	\end{proposition}

We prove Proposition \ref{multilinear:det-2} in Section \ref{sec:det}. Let us now state the trilinear probabilistic estimates, which we prove in Section \ref{sec:CCC}. In contrast to the above deterministic estimate, we also cover the Gibbs case corresponding to $\alpha=1$.

	\begin{proposition}[(C)(C)(C) type interactions]\label{multilinear:random} 
	Let $\alpha\geq 1$.	Suppose that $\mathrm{Loc}(N)$ holds on $\Omega_{N}$ with parameters $(T,R,\sigma)$, for some sufficiently small $\sigma\in(0,1)$, and that $\psi_{2N}^{\dag}$ also satisfies the bound \eqref{eq:p-psi}. Let $N_1,N_2,N_3$ be dyadic integers such that $2N=N_{(1)}\geq N_{(2)}\geq N_{(3)}$. Let $\psi_{N_j}^{\dag}$ be of type $\mathrm{(C)}$. Then there exits a $\mathcal{B}_{\leq 2N}$-measurable set $\Xi$ with $\mathbb{P}(\Omega\setminus \Xi)<C_0\e^{-N^{c_0}R^{\delta_0}}$, such that the following estimates hold on $\Omega_{N}\cap\Xi$: 
	\begin{itemize}
	\item If $N_1=2N$ and $N_2,N_3\leq N$, then 
		\[
		  \|
		  \chi_T(t)\mathcal{I}\mathcal{N}_{[0,1]}(\psi_{N_1}^{\dag},\psi_{N_2}^{\dag},\psi_{N_3}^{\dag})
		  \|_{X^{0,b}}
		  \leq 
		  T^{\frac{\theta}{2}}
		  R^4N^{-(\alpha-\frac{1}{2})+\frac{1}{q}+2\delta+100\theta}.
		\]
	\item If $N_{2}=2N$ or $N_{(1)}=N_{(2)}=2N$, then
		\[
		\|
		\chi_T(t)\mathcal{I}\mathcal{N}(\psi_{N_1}^{\dag},\psi_{N_2}^{\dag},\psi_{N_3}^{\dag})
		\|_{X^{0,b}}
		\leq 
		T^{\frac{\theta}{2}}R^4
		N^{-(\alpha-\frac{1}{2})+\frac{1}{q}+2\delta+100\theta}\,.
		\]
	\end{itemize}		

	\end{proposition}
	\subsection{Proof of Proposition \ref{keyinduction} by assuming multi-linear estimates }
	\label{sec:s1}
	In this subsection, we show that Proposition \ref{multilinear:op1}, Proposition~\ref{multilinear:op2},  Proposition~\ref{multilinear:det-1}, Proposition~\ref{multilinear:det-2} and Proposition~\ref{multilinear:random} imply the key-induction stated in Proposition~\ref{keyinduction}.
	
Let $R\gg 1$ be a fixed parameter. Assuming $\mathrm{Loc}(N)$ (see Definition \ref{loc(N)}) holds for all $\omega\in \Omega_N$ which is $\mathcal{B}_{\leq N}$-measurable. 

 \medskip
 
\noi
$\bullet${\bf Step 1 : Uniform bounds for the RAOs }	

\medskip

Equation \eqref{def:H2Nextension} can be written in term of $\mathbf{h}_n^{2N,\dag}$:
\begin{align}\label{eq:hn2N}
\mathbf{h}_n^{2N,\dag}(t)=-2i\chi_{2T}(t)\int_0^t\e^{-i\lambda_n^2(t-t')}\mathcal{P}_n^{2N,\dag}(t')\circ\mathcal{H}_n^{2N,\dag}(t')\d t'\,.
\end{align}
Thanks to Proposition \ref{multilinear:op1}, 	the $S_n^{q,\gamma}$ norm of the right hand side is bounded by
	\[
	 CT^{\theta-2\gamma+\frac{2}{q'}}R^2\|\mathcal{H}_n^{2N,\dag}(t)\|_{S_n^{q,\gamma}}\leq CT^{\frac{\theta}{2}}R^2(1+\|\mathbf{h}_n^{2N,\dag}\|_{S_n^{q,\gamma}}),
	\]
where we used the fact that  $\theta-2\gamma+\frac{2}{q'}>\frac{\theta}{2}$, thanks to the choices of parameters. 

Similarly, the equation~\eqref{def:G2N} can be written in term of $\mathbf{g}_n^{2N,\dag}$:
\begin{align}\label{eq:gn2N}
\mathbf{g}_n^{2N,\dag}(t)=2i\chi_{2T}(t)\Big(\int_0^t\e^{i\lambda_n^2t'}\mathcal{P}_n^{2N,\dag}(t')\circ\mathcal{H}_n^{2N,\dag}(t')\d t'\Big)\circ\mathcal{G}_n^{2N,\dag}(t),
\end{align}
by Proposition \ref{multilinear:op2}, the $S_n^{q,\gamma,*}$-norm of the right hand side is bounded by
	\[
	 CT^{\frac{\theta}{2}}R^2(1+\|\mathbf{g}_{n}^{2N,\dag}\|_{S_n^{q,\gamma,*}} ),
	\]
provided that $\theta-3\gamma+\frac{3}{q'}>\frac{\theta}{2}$.

Consequently, for \(T\sim R^{-\frac{10}{\theta}}\), and for any $N<n\leq 2N$, we obtain that
	\[ \|\mathbf{h}_n^{2N,\dag}\|_{S_n^{q,\gamma}}+\|\mathbf{g}_n^{2N,\dag}\|_{S_n^{q,\gamma,*}} \leq R^{-1},\; 
	\]
and
	\[
	\|\mathcal{H}_n^{2N,\dag}\|_{S_n^{q,\gamma}}+\|\mathcal{G}_n^{2N,\dag}\|_{S_n^{q,\gamma,*}}\leq R\,.
	\]
		From Lemma \ref{bd:colored} and Lemma \ref{lem:wick}, by deleting a subset $\Xi_N$ from $\Omega_N$ of probability smaller than $C_0\e^{-c_0(2NR)^{\delta_0}}$, the colored term $\psi_{2N}^{\dag}$ satisfies the bounds \eqref{eq:p-psi} and~\eqref{eq:p-psi-}. 
\medskip

\noi
$\bullet${\bf Step 2: Fixed-point problem for the remainder.} We now solve the fixed-point problem \eqref{eq:wNdag} for $w_{2N}^{\dag}$ in the closed set
\[
\mathbf{B}_{2N} :=\big\{
	z :\; \|z\|_{X^{0,b}}\leq (2N)^{-s}R^{-1}, \sup_{L\geq 4N}\|\Pi_L^{\perp}z\|_{X^{0,b}}\leq (\frac{L}{2N})^{-10}(2N)^{-s}R^{-1}  
	\big\}\,,
\]	
where we recall from \eqref{smallparameters} that $s=\alpha-\frac{1}{2}-100\sigma>\frac{1}{2}+100\sigma$ (we supposed $\alpha>1$). 
Below we show that with the choice $T\sim R^{-\frac{10}{\theta}}$, the mapping
\begin{equation}
\label{fixed-pointz2N}
	\begin{split}
 \mathbf{\Phi}&(z_{2N})
 	:=2\chi_T(t)\big[\mathcal{I}\mathcal{N}_{(0,1)}(z_{2N},u_{N}^{\dag},u_{N}^{\dag})+2\mathcal{I}\mathcal{N}_{[0,1]}(\psi_{2N}^{\dag}+z_{2N},u_{N}^{\dag},u_{N}^{\dag}) \big]  \\
 	&+2\chi_T(t)\big[\mathcal{I}\mathcal{N}(u_N^{\dag},\psi_{2N}^{\dag}+z_{2N},u_N^{\dag}) 
	+\mathcal{I}\mathcal{N}(\psi^{\dag}_{2N}+z_{2N},\psi^{\dag}_{2N}+z_{2N},u^{\dag}_{N})\big]
	\\
	&+\chi_T(t)\big[\mathcal{I}\mathcal{N}(\psi_{2N}^{\dag}+z_{2N},u^{\dag}_{N},\psi_{2N}^{\dag}+z_{2N})   
	+\mathcal{I}\mathcal{N}(\psi_{2N}^{\dag}+z_{2N})\big]
	\end{split}
\end{equation}
 maps $\mathbf{B}_{2N}$ to itself. The contraction property can be proved in a similar way upon elementary algebraic manipulation and we do not detail this part. 
\medskip

First, from Proposition \ref{multilinear:det-1},  we have that for any type (D) term $z_{2N}\in \mathbf{B}_{2N}$
\begin{align}\label{remainder1} \|\chi_T(t)\mathcal{I}\mathcal{N}_{(0,1)}(z_{2N},u_N^{\dag},u_N^{\dag})\|_{X^{s_1,b}} \leq CT^{\theta-2\gamma+\frac{2}{q'}}R^2\|z_{2N}\|_{X^{0,b}}\leq CT^{\frac{\theta}{2}}R(2N)^{-s},
\end{align}
thanks to the condition $\theta-2\gamma+\frac{2}{q'}>\frac{\theta}{2}$.

For other contributions on the right hand side of \eqref{fixed-pointz2N}, we decompose $u_N^{\dag}=\sum_{L\leq N}(\psi_L^{\dag}+w_{L}^{\dag})$ and estimate each dyadic piece. The (C)(C)(C) type interactions contribute to the source term, and are linear combination of the following terms 
\begin{align*}
 	&\sum_{N_2,N_3\leq N}
	\chi_T(t)\mathcal{I}\mathcal{N}_{[0,1]}(\psi_{2N}^{\dag},\psi_{N_2}^{\dag},\psi_{N_3}^{\dag}),\; \sum_{N_1,N_3\leq N}\chi_T(t)\Pi_{2N}\mathcal{I}\mathcal{N}(\psi_{N_1}^{\dag},\psi_{2N}^{\dag},\psi_{N_3}^{\dag})\\ 
	&\sum_{N_2\leq 2N}\chi_T(t)\mathcal{I}\mathcal{N}(\psi_{2N}^{\dag},\psi_{N_2}^{\dag},\psi_{2N}^{\dag}),\; \sum_{N_3\leq N}\chi_T(t)\mathcal{I}\mathcal{N}(\psi_{2N}^{\dag},\psi_{2N}^{\dag},\psi_{N_3}^{\dag})\,.
\end{align*}
According to Proposition \ref{multilinear:random}, the $X^{0,b}$-norm of these terms can be bounded by
\[
CT^{\frac{\theta}{2}}R^4(2N)^{-(\alpha-\frac{1}{2})+\frac{1}{q}+3\delta+100\theta}(\log N)^3
	\leq CT^{\frac{\theta}{2}}R^{4}(2N)^{-s}\,,
\]
 outside a subset of probability smaller than $C_0\e^{-(\log N)^{c_0}R^{c_0}}$,
where the $\log N$ factor comes from the dyadic summation. We used that, according to the hierarchy of parameters (see \eqref{smallparameters}),
	\[
	s=\alpha-\frac{1}{2}-100\sigma < \alpha-\frac{1}{2}-\frac{1}{q}-3\delta-100\theta\,.
	\]
For other terms ((C)(C)(D), (C)(D)(D) and (D)(D)(D) interactions), we apply Proposition \ref{multilinear:det-2} to control the  $X^{0,b}-$norm of the sum of these dyadic pieces by
 \[
 	\sum_{N_{(3)}\leq N_{(2)}\leq 2N}CT^{\frac{\theta}{2}}
	R^2(2N)^{-s}N_{(2)}^{-\epsilon_0}
	\leq CT^{\frac{\theta}{2}}R^2(2N)^{-s}.
\]
For $T=R^{-\frac{10}{\theta}}$ and sufficiently large $R>1$, we deduce that
\[
CT^{\frac{\theta}{2}}R^4(2N)^{-s}\leq (2N)^{-s}R^{-1}.
\]

It remains to estimate $\|\Pi_L^{\perp}\mathbf{\Phi}(z_{2N})\|_{X^{0,b}}$ for all $L\geq 4N$. The proof is exactly the same as in the previous paragraphs. Here we sketch the main point to gain of the extra factor $(\frac{L}{2N})^{-10}$. Indeed, since $\psi_M^{\dag}$ is exactly supported on frequencies $m\in(\frac{M}{2},M]$, we have  $\Pi_L^{\perp}\mathcal{N}(\psi_{N_1}^{\dag},\psi_{N_2}^{\dag},\psi_{N_3}^{\dag})=0$ for $L\geq2^{4}N$.  
\medskip

For terms $\Pi_L^{\perp}\mathcal{N}(v_{N_1}^{\dag},v_{N_2}^{\dag},v_{N_3}^{\dag})$ with at least one $v_{N_j}^{\dag}$ of type $(\mathrm{D})$, with $\frac{L}{4}\geq 4N$, there is at least one factor $\Pi_{\frac{L}{4}}^{\perp}v_{N_j}^{\dag}$ present in the non-zero contributions. This term generates a factor $C(\frac{L}{4N})^{-10}\cdot (2N)^{-s}R^{-1}$. 
With these additional arguments, we conclude the proof of the key-induction Proposition~\ref{keyinduction}.



	\section{Some deterministic estimates for operators}\label{sec:6}
	
	Recall the notation
	\begin{align}\label{resonant:trilinear} 
		\mathcal{N}_{(0,1)}(f_1,f_2,f_3):=\sum_{n,n_2,n_3}\pi_{n}(\pi_{n}f_1\pi_{n_2}\ov{f}_2\pi_{n_3}f_3)
	\end{align}
	and the slightly different variant (in order to deal with the product $f_2\hexagonal{\neq}f_3$)
	\begin{align}\label{resonant:trilinear'} 
		\mathcal{N}_{(0,1)[2,3]}(f_1,f_2,f_3):=\sum_{\substack{n,n_2, n_3\\ n_2\neq n_3} }\pi_{n}(\pi_{n}f_1\pi_{n_2}\ov{f}_2\pi_{n_3}f_3).
	\end{align}
In this section, we provide some deterministic estimates that are needed for the key estimates for RAOs. Since we allow the terms of type $(\mathrm{C})$ in $X_{q,q,\infty}^{\frac{1}{2}-\frac{2}{q}-3\delta,\gamma}$, our estimates are also valid when $\alpha=1$ and will be useful to treat the Gibbs measure case in \cite{BCST2}. 
To state the following propositions, we assume that parameters $q,\gamma,\gamma_0,\theta,\delta$ satisfies the following constraint
\begin{align}\label{constraint:parametersq}
&1-\gamma=\frac{1}{q}-\frac{1}{q^{10}},\; \gamma_0=\frac{1}{q}+\frac{1}{q^{10}},\; \delta=\frac{1}{q^{20}},\;\theta=\frac{1}{q^5}.
\end{align}
	\begin{proposition}\label{lem:resonanttrilinear} 
		Let and $\varphi\in \mathcal{S}(\R;\R)$. There exist $q\in[2,\infty)$ large enough, and $\gamma,\gamma_0,\delta,\theta$ obeying \eqref{constraint:parametersq},  such that for any $s>\frac{1}{4}+\frac{3}{q}$, for $\frac{1}{2}<b<1$ and for any  $F_2\in X_{q,q,\infty}^{\frac{1}{2}-\frac{2}{q}-3\delta,\gamma}\,,\ F_3\in X^{s,b}$ and $n\in\N$ we have 
		\[
		\|
		\varphi(t)\pi_n\mathcal{N}_{(0,1)}(\,\cdot\,,F_{2},F_{3})\|_{X_{q,2}^{0,\gamma}(E_n) \to X_{q,2}^{0,\gamma-1+\theta}(E_n)}
		\lesssim \|F_2\|_{
			X_{q,q,\infty}^{\frac{1}{2}-\frac{2}{q}-3\delta,\gamma}
		}
		\|F_3\|_{X^{s,b}}\,,
		\]
 where the implicit constant is independent of $n$ and $0<T<1$.
	\end{proposition}
	Similarly, we have:
	\begin{proposition}\label{lem:resonanttrilinear'} 
		Let $\varphi\in \mathcal{S}(\R;\R)$. There exist $q\in[2,\infty)$ large enough, and $\gamma,\gamma_0,\delta,\theta$ obeying \eqref{constraint:parametersq},  such that for any $s>\frac{1}{4}+\frac{3}{q}$, for $\frac{1}{2}<b<1$, and for any  $F_2,F_3\in X^{s,b}$ and $n\in\N$, 
		\[
		\|\varphi(t)\pi_n\mathcal{N}_{(0,1)}(\,\cdot\,,F_{2},F_{3,})\|_{X_{q,2}^{0,\gamma}(E_n)\to X_{q,2}^{0,\gamma-1+\theta}(E_n)}
			\lesssim 
			\|F_2\|_{X^{s,b}}
			\|F_3\|_{X^{s,b}},
		\]
		where the implicit bound is independent of $n$ and $0<T<1$.
	\end{proposition}
	For non-pairing part, we have similar estimate:
	\begin{proposition}\label{lem:resonanttrilinear''} 
	Let $\varphi\in \mathcal{S}(\R;\R)$. There exist $q\in[2,\infty)$ large enough, and $\gamma,\gamma_0,\delta,\theta$ obeying \eqref{constraint:parametersq},  such that for any  $F_2\,, F_3\ \in X_{q,q,\infty}^{\frac{1}{2}-\frac{2}{q}-3\delta,\gamma}
			$,
		\[
			\|\varphi(t)\mathcal{N}_{(0,1)[2,3]}(\,\cdot\,,F_{2},F_{3})\|_{X_{q,2}^{0,\gamma}(E_n)\to X_{q,2}^{0,\gamma-1+\theta}(E_n)} \lesssim 
			\|F_2\|_{X_{q,q,\infty}^{\frac{1}{2}-\frac{2}{q}-3\delta,\gamma}}
			\|F_3\|_{X_{q,q,\infty}^{\frac{1}{2}-\frac{2}{q}-3\delta,\gamma} },
		\]
		where the implicit bound is independent of $n$ and $0<T<1$.
	\end{proposition}
	Before proving the above propositions, we state an elementary Lemma useful in the analysis.

	\begin{lemme}\label{sum1} 
		Assume that 
		\[
		\alpha_2,\alpha_3\in(0,1),\; \alpha_2+\alpha_3>\frac{1}{2},
		\]
		then for any $\widetilde{\kappa}\in \R$, $1\leq r\leq 2$
		\[ \sum_{n_3\geq 1}\lambda_{n_3}^{-2\alpha_3}\Big(\sum_{n_2\geq 1}\lambda_{n_2}^{-r\alpha_2}\lg\widetilde{\kappa}+\lambda_{n_2}^2-\lambda_{n_3}^2\rg^{-10r}\Big)^{\frac{2}{r}}\lesssim 1,
		\] 
		where the implicit bound is independent of $\widetilde{\kappa}\in\R$.
	\end{lemme}
	
	\begin{proof}
		We fix \(n_{3}\) and separate the sum over \(n_{2}\) into two parts depending on the relative size of \(n_{2}\) compared to \(n_{3}\). In each part, we make the change of variable $m=|\lambda_{n_3}^2-\lambda_{n_2}^2|$:
		\[
			\sum_{n_{2}=1}^{n_{3}}\lambda_{n_2}^{-r\alpha_2}\lg\widetilde{\kappa}+\lambda_{n_2}^2-\lambda_{n_3}^2\rg^{-10r}
			\leq\sum_{m=0}^{\lambda_{n_{3}}^{2}-1}\lg\widetilde{\kappa}-m\rg^{-10r}(\lambda_{n_3}^2-m)^{-\frac{r\alpha_2}{2}}\,.
		\]
		Then we split the sum into two contributions: on the one hand, 
		\[
		\sum_{m=0}^{{\lambda_{n_{3}}}^{2}-1}\mathbf{1}_{|\lambda_{n_{3}}^{2}-m|>\frac{1}{2}|\widetilde{\kappa}-\lambda_{n_{3}}^{2}|}
		\lg\widetilde{\kappa}-m\rg^{-10r}(\lambda_{n_3}^2-m)^{-\frac{r\alpha_2}{2}} 
		\lesssim \langle \widetilde{\kappa}-\lambda_{n_{3}}^{2}\rangle^{-\frac{r\alpha_{2}}{2}}\sum_{m=0}^{\lambda_{n_{3}}^{2}-1}\lg\widetilde{\kappa}-m\rg^{-10r}
		\lesssim \langle \widetilde{\kappa}-\lambda_{n_{3}}^{2}\rangle^{-\frac{r\alpha_{2}}{2}}\,.
		\] 
		On the other hand, when 
		\(|\lambda_{n_{3}}^{2}-m|>\frac{1}{2}|\widetilde{\kappa}-\lambda_{n_{3}}^{2}|\) 
		then 
		\(|\widetilde{\kappa}-m|\geq\frac{1}{2}|\widetilde{\kappa}-\lambda|\) and therefore 
		\[
		\sum_{m=0}^{{\lambda_{n_{3}}}^{2}-1}
		\mathbf{1}_{|\lambda_{n_{3}}^{2}-m|>\frac{1}{2}|\widetilde{\kappa}-\lambda_{n_{3}}^{2}|}
		\lg\widetilde{\kappa}-m\rg^{-10r}(\lambda_{n_3}^2-m)^{-\frac{r\alpha_2}{2}} 
		\lesssim 
		\langle \widetilde{\kappa}-\lambda_{n_{3}}^{2}\rangle^{-r}\sum_{m=0}^{\lambda_{n_{3}}^{2}-1}\lg\widetilde{\kappa}-m\rg^{-9r}
		\lesssim 
		\langle \widetilde{\kappa}-\lambda_{n_{3}}^{2}\rangle^{-\frac{r\alpha_{2}}{2}}\,,
		\] 
		where we used that \(\frac{\alpha}{2}\leq 1\).  This gives
		\[
		\sum_{n_{2}=1}^{n_{3}}\lambda_{n_2}^{-r\alpha_2}\lg\widetilde{\kappa}+\lambda_{n_2}^2-\lambda_{n_3}^2\rg^{-10r}
		\lesssim
		\lg\widetilde{\kappa}-\lambda_{n_3}^2\rg^{-\frac{r\alpha_2}{2}}\,.
		\]
		Similarly, we prove that 
		\[
		\sum_{n_{2}=n_{3}+1}^{+\infty}
		\lambda_{n_2}^{-r\alpha_2}\lg\widetilde{\kappa}+\lambda_{n_2}^2-\lambda_{n_3}^2\rg^{-10r}
			\leq
			\sum_{m\geq 0}\lg\widetilde{\kappa}+m\rg^{-10r}\lg m+\lambda_{n_3}^2\rg^{-\frac{r\alpha_2}{2}}
			\lesssim  \lg\widetilde{\kappa}-\lambda_{n_3}^2\rg^{-\frac{r\alpha_2}{2}}.
		\]
		Then, the sum over \(n_{3}\) can be estimated by
		\begin{align*}
			\sum_{n_3\geq 1}
			\lambda_{n_3}^{-2\alpha_3}\lg\widetilde{\kappa}-\lambda_{n_3}^2\rg^{-\alpha_2}\lesssim_{\kappa} 1\,,
		\end{align*}
		thanks to the fact that $2\alpha_2+2\alpha_3>1$.
	\end{proof}

\subsection{Multilinear estimates for $\mathcal{N}_{(0,1)}$ and $\mathcal{N}_{(0,1)[2,3]}$ }	
	  The proof will follow from an interpolation argument. First we prove some general estimates which leave some flexibility of indices.
	\begin{lemme}\label{lem:weaker}
		Assume that $q_1,q_2\in[2,\infty)$ and $\gamma_j\in\big(\frac{1}{q_j'},1\big)$. For any $\gamma_0\in\big(\frac{1}{q_1},1\big)$ and $\delta>0, b>\frac{1}{2}$ such that 
	$\frac{1}{2}-\frac{4}{q_2}-3\delta>0$,	
	uniformly in $n\in\mathbb{N}$ and $Z_{1}\in X_{q_{1},2}^{0,\gamma_{1}}(E_{n})$,
		\begin{align}\label{weaker1}
			\|\varphi(t)\pi_n\mathcal{N}_{(0,1)}&(Z_1,F_{2},F_{3})\|_{X_{q_1,2}^{0,-\gamma_0}(E_n)}\notag  \\ \lesssim_{q_j,\gamma_j,\delta} &
			\|Z_1\|_{X_{q_1,2}^{0,\gamma_1}(E_n)}
			\|F_2\|_{
				X_{q_2,q_2,\infty}^{\frac{1}{2}-\frac{4}{q_2}-3\delta,\gamma_2} 
			}
			\|F_3\|_{
				X^{\frac{1}{4}+\frac{4}{q_2}+4\delta,b}
			},
		\intertext{and}
		\label{weaker2} \|\varphi(t)\pi_n\mathcal{N}_{(0,1)}&(Z_{1},F_{2},F_{3})\|_{X_{q_1,2}^{0,0}(E_n)} \notag \\ \lesssim_{q_j,\gamma_j} &
		\|Z_1\|_{X_{q_1,2}^{0,\gamma_1}(E_n)}
		\|F_2\|_{ 
			X_{q_2,q_2,\infty}^{1,\gamma_2}}
		\|F_3\|_{X^{2,b}}.
		\end{align}
	\end{lemme} 
	\begin{proof} 
		First we prove \eqref{weaker1}. By duality, it suffices to show that for any $G_n\in X_{q_1',2}^{0,\gamma_0}(E_n)$ with $\|G_n\|_{X_{q_1',2}^{0,\gamma_0}(E_n) }\leq 1,$ there holds
		\begin{multline}
		\label{eq:resonanttrilinear1} 
			\big|\int_{\R}\int_{\S^2}\pi_n\,\mathcal{N}_{(0,1)}(Z_{1},F_{2},F_{3})\cdot  \varphi(t)\ov{G}_n(t,x) 
			\mathrm{d}x\mathrm{d}t
			\big| 
			\\
			\lesssim 
			\|Z_{1}\|_{X_{q_1,2}^{0,\gamma_1}(E_n)}
			\|F_2\|_{X_{q_2,q_2,\infty}^{\beta_2,\gamma_2}}
			\|F_3\|_{X^{\beta_3,b}}\,,
		\end{multline}
	where throughout the proof
	\[
	\beta_2=\frac{1}{2}-\frac{4}{q_2}-3\delta,
	\quad \beta_3=\frac{1}{4}+\frac{4}{q_2}+4\delta\,.
	\]
	By Plancherel in time, we have
		\begin{align*}
			&\int_{\R}\int_{\S^2}\pi_n\mathcal{N}_{(0,1)}(Z_{1},F_{2},F_{3})\cdot\varphi(t)\ov{G}_n(t,x)\d x \d t\\
			=&\int_{\R}\int_{\S^2}
			\sum_{n_2,n_3}Z_{1}(t)\pi_{n_2}\ov{F}_2(t)\pi_{n_3}F_3(t)\varphi(t)\ov{G}_n(t)\d t \d x\\
			=&\int_{\R^4}\d \vec{\tau} 
			\int_{\S^2}
			\sum_{n_2, n_3}
			\widehat{Z}_1(\tau_1,x)\pi_{n_2}\ov{\widehat{F}}_2(\tau_2,x)\pi_{n_3}\widehat{F}_{3}(\tau_3,x)
			\ov{\widehat{G}}_n(\tau_0,x)
			{\widehat{\varphi}}(\tau_1-\tau_2+\tau_3-\tau_0)\d x\\
			=&\int_{\R^4}d\vec{\kappa}
			\int_{\S^2}\sum_{n_2,n_3}
			\widetilde{Z}_1(\kappa_1,x)\pi_{n_2}\ov{\widetilde{F}}_2(\kappa_2,x)\pi_{n_3}\widetilde{F}_{3}(\kappa_3,x)\ov{\widetilde{G}}_n(\kappa_0,x)
			{\widehat{\varphi}}(\widetilde{\kappa}-\Omega(\vec{n}))\d x,
		\end{align*}
		where $\kappa_j=\tau_j+\lambda_{n_j}^2$, $\d \vec{\tau}=\d \tau_0 \d \tau_1 \d \tau_2 \d \tau_3$,  $\d\vec{\kappa}=\d\kappa_1\d\kappa_2\d\kappa_3\d\kappa_0$, $\widetilde{\kappa}=\kappa_1-\kappa_2+\kappa_3-\kappa_0,$ and $\Omega(\vec{n})=-\lambda_{n_2}^2+\lambda_{n_3}^2$.
		Set
		\begin{align*}
		\mathbf{a}^{(1)}_{n}(\kappa_1)
		&=\|\lg\kappa_1\rg^{\gamma_1}\widetilde{Z}_1(\kappa_1,\cdot)\|_{L_x^2}
		\,,\quad
		\mathbf{a}_{n_2}^{(2)}(\kappa_2)
		=\lambda_{n_2}^{\beta_2}\|\lg\kappa_{2}\rg^{\gamma_2}\pi_{n_2}\widetilde{F}_2(\kappa_2,\cdot)\|_{L_x^{\infty}},
		\\
		\mathbf{a}_{n_3}^{(3)}(\kappa_3)
		&=\lambda_{n_3}^{\beta_3}
		\|\lg\kappa_3\rg^{b}\pi_{n_3}\widetilde{F}_3(\kappa_3,\cdot) \|_{L_x^2}
		\,,\quad 
		\mathbf{a}_n^{(0)}(\kappa_0)
		=\|\lg\kappa_0\rg^{\gamma_0}\widetilde{G}_n(\kappa_0,\cdot) \|_{L_x^2}.
		\end{align*}
		Applying the bilinear eigenfunction estimate from Proposition \ref{bilinearBGT}, we obtain
		\begin{multline*}
		\big|\int_{\S^2}\widetilde{Z}_1(\kappa_1,x)\pi_{n_2}\ov{\widetilde{F}}_2(\kappa_2,x)\pi_{n_3}\widetilde{F}_{3}(\kappa_3,x)\ov{\widetilde{G}}_n(\kappa_0,x)\mathrm{d}x\big|
		\\ 
		\lesssim \lambda_{n_3}^{\frac{1}{4}}
		\lambda_{n_2}^{-\beta_2}
		\lambda_{n_3}^{-\beta_3} \frac{
				\mathbf{a}_n^{(0)}(\kappa_0)
				\mathbf{a}_{n}^{(1)}(\kappa_1)
				\mathbf{a}_{n_2}^{(2)}(\kappa_2)
				\mathbf{a}_{n_3}^{(3)}(\kappa_3)
			}{\lg\kappa_0\rg^{\gamma_0}
			\lg\kappa_1\rg^{\gamma_1}\lg\kappa_2\rg^{\gamma_2}\lg\kappa_3\rg^b}.
		\end{multline*}
	It remains to control the expression
		\begin{align}\label{ankappa1}
			\int_{\R^4}
			\mathrm{d} \vec{\kappa}
			\frac{\mathbf{a}_n^{(0)}(\kappa_0) \mathbf{a}_n^{(1)}(\kappa_1)}{
				\lg\kappa_0\rg^{\gamma_0}
				\lg \kappa_1\rg^{\gamma_1}\lg \kappa_2\rg^{\gamma_2}\lg\kappa_3\rg^{b} }\sum_{n_2,n_3}
			\mathbf{a}_{n_2}^{(2)}(\kappa_2)\mathbf{a}_{n_3}^{(3)}(\kappa_3)
			|\ov{\widehat{\varphi}}(\widetilde{\kappa}-\Omega(\vec{n}))|
			\lambda_{n_2}^{-\beta_2}
			\lambda_{n_3}^{-\beta_3+\frac{1}{4}}.
		\end{align}
		By H\"older, the sum in the integrand can be bounded by
		\[
		\|\mathbf{a}_{n_2}^{(2)}(\kappa_2)\|_{\ell_{n_2}^{q_2}}
		\|\mathbf{a}_{n_3}^{(3)}(\kappa_3)\|_{\ell_{n_3}^2}
		\big\|\ov{\widehat{\varphi}}(\widetilde{\kappa}-\Omega(\vec{n}))
		\lambda_{n_2}^{-\beta_2}
		\lambda_{n_3}^{-\beta_3+\frac{1}{4}} 
		\big\|_{\ell_{n_3}^2\ell_{n_2}^{q_2'}}
		\\
		\lesssim_{\delta,q_2} \|\mathbf{a}_{n_2}^{(2)}(\kappa_2)\|_{\ell_{n_2}^{q_2}}
		\|\mathbf{a}_{n_3}^{(3)}(\kappa_3)\|_{\ell_{n_3}^2},  
		\]
		where, to obtain the second estimate, we applied Lemma \ref{sum1} with parameters~$r=q_{2}'$ in $[1,2]$, $\alpha_{2}=\beta_{2}\in(0,1)$ and $\alpha_{3}=\beta_{3}-\frac{1}{4}\in(0,1)$, satisfying $\alpha_{2}+\alpha_{3}=\beta_2+\beta_3-\frac{1}{4}=\frac{1}{2}+\delta$.  Then, by H\"older and the fact 
that $\gamma_0q_{1}>1$ and $\gamma_jq_j'>1, j=1,2$, we obtain that
		\[
		\eqref{ankappa1}\lesssim
		\|\mathbf{a}_{n}^{(0)}(\kappa_0)\|_{L_{\kappa_0}^{q_1'}}\|\mathbf{a}_{n}^{(1)}(\kappa_1)\|_{L_{\kappa_1}^{q_1}}
		\|\mathbf{a}_{n_2}^{(2)}(\kappa_2)\|_{L_{\kappa_2}^{q_2}\ell_{n_2}^{q_2}}
		\|\mathbf{a}_{n_3}^{(3)}(\kappa_3)\|_{L_{\kappa_3}^2\ell_{n_3}^2\,.}
		\]
		This proves \eqref{weaker1}.
\medskip

		Next, to prove \eqref{weaker2}, by Corollary \ref{Hausdorff-Young}, 
		\begin{align*}
			\|\varphi(t)\pi_n\mathcal{N}_{(0,1)}(Z_{1},F_{2},F_{3})\|_{X_{q_1,2}^{0,0}(E_n)}\lesssim & \|\varphi(t)\pi_n\mathcal{N}_{(0,1)}(Z_{1},F_{2},F_{3})\|_{L_t^{q_1'}L_x^2}.
		\end{align*}
		The right hand side can be controlled by
		\[
		 \|Z_{1}\|_{L_t^{\infty}L_x^2}\|F_{2}\|_{L_t^{\infty}L_x^{\infty}}\|F_{3}\|_{L_t^{\infty}L_x^{\infty}}. 
		\]
		By Sobolev embedding and the embedding property stated in Lemma \ref{Embedding},
		\[ \|Z_{1}\|_{L_t^{\infty}L_x^2}\lesssim \|\pi_nZ_{1}\|_{X_{q_1,2}^{0,\gamma_1}(E_n)},\;\|F_{3}\|_{L_t^{\infty}L_x^{\infty}}\lesssim \|F_3\|_{L_t^{\infty}H_x^2}\lesssim \|F_3\|_{X^{2,b}}.
		\]
		Moreover,
		\[
		 \|F_{2}\|_{L_x^{\infty}}=
		\|\sum_{n_2}\pi_{n_2}F_2
		\|_{L_x^{\infty}}
		 \leq 
		\big(\sum_{n_2\geq0}\lambda_{n_{2}}^{q_2}\|\pi_{n_2}F_2\|_{L_x^{\infty}}^{q_2}
		\big)^{\frac{1}{q_2}}
		 \big(\sum_{n_2\geq0}\frac{1}{\lambda_{n_{2}}^{q_2'}}
		\big)^{\frac{1}{q_2'}}
		\lesssim_{q_2}
		\|\lambda_{n_2}\pi_{n_2}F_2
		\|_{\ell_{n_2}^{q_2}L_x^{\infty}}.
		\]
		Hence
		\[
		 \|F_{2}\|_{L_t^{\infty}L_x^{\infty}}
		 \lesssim \|\lambda_{n_2}\pi_{n_2}F_2 \|_{L_t^{\infty}\ell_{n_2}^{q_2}L_x^{\infty}}
		 \lesssim \|\lambda_{n_2}\pi_{n_2}F_2\|_{\ell_{n_2}^{q_2}L_t^{\infty}L_x^{\infty}}
		 \lesssim \|F_2\|_{X_{q_2,q_2,\infty}^{1,\gamma_2}},
		\]
		where we used Lemma \ref{Embedding} to the last inequality. The proof of Lemma \ref{lem:weaker} is now complete.
	\end{proof}
	Similarly, we have:
	\begin{lemme}\label{lem:weaker'}
		Assume that $q\in[2,\infty)$ and $\gamma\in\big(\frac{1}{q'},1\big)$. For any $\gamma_0\in\big(\frac{1}{q},1\big)$, for $\delta>0$ and $b>\frac{1}{2}$, uniformly in $n\in\mathbb{N}$ and $Z_{1}\in X_{q_{1},2}^{0,\gamma_{1}}(E_{n})$, we have
		\begin{multline}
		\label{weaker1'}
			\|\varphi(t)\pi_n\mathcal{N}_{(0,1)}(Z_{1},F_2,F_3)\|_{X_{q,2}^{0,-\gamma_0}(E_n)}
			\\
			\lesssim_{q,\gamma,\delta} \|Z_{1}\|_{X_{q,2}^{0,\gamma}(E_n)}\|F_2\|_{X^{\frac{1}{4}+\frac{2}{q}+4\delta,b}}\|F_3\|_{X^{\frac{1}{4}+\frac{2}{q}+4\delta,b}},
		\end{multline}
		and
		\begin{equation}
		\label{weaker2'} \|\varphi(t)\pi_n\mathcal{N}_{(0,1)}(Z_{1},F_2,F_3)\|_{X_{q,2}^{0,0}(E_n)}
		\lesssim_{q_j,\gamma_j} \|Z_{1}\|_{X_{q,2}^{0,\gamma}(E_n)}\|F_2\|_{X^{2,b}}\|F_3\|_{X^{2,b}}.
		\end{equation}
	\end{lemme}
	\begin{proof}
		The proof follows from very similar argument as in the proof of Lemma \ref{lem:weaker}, and we only emphasize the adaptations needed. To prove~\eqref{weaker1'}, we modify the function $\mathbf{a}_{n_2}^{(2)}(\kappa_2)$ as
		\[
		\mathbf{a}_{n_2}^{(2)}(\kappa_2)=\lambda_{n_2}^{\beta_2}\|\lg\kappa\rg^b\pi_{n_2}\widetilde{F}_{2}(\kappa,\cdot)\|_{L_x^2}\,,\quad \text{ where }\quad \beta_2=\frac{1}{4}+\frac{2}{q}+4\delta.
		\]
		We also set $\beta_3=\beta_{2}$.
		Using the bilinear eigenfunction estimate (Proposition \ref{bilinearBGT}), the expression to be estimated becomes
		\begin{align*}
			\int_{\R^4}d\vec{\kappa}\sum_{n_2,n_3}(\lambda_{n_2}\lambda_{n_3})^{\frac{1}{4}-\beta_{2}}\,
		\frac{	\mathbf{a}_n^{(0)}(\kappa_0)
				\mathbf{a}_{n}^{(1)}(\kappa_1)
				\mathbf{a}_{n_2}^{(2)}(\kappa_2)
				\mathbf{a}_{n_3}^{(3)}(\kappa_3)
			}{\lg\kappa_0\rg^{\gamma_0}
			\lg\kappa_1\rg^{\gamma}\lg\kappa_2\rg^{b}\lg\kappa_3\rg^b}|\ov{\widehat{\varphi}}(\widetilde{\kappa}-\Omega(\vec{n}))|.
		\end{align*}
		Here to estimate the sum over $n_2,n_3$, we use Schur's test, by observing that		\[
		\sup_{n_2}\sum_{n_3}|\widehat{\varphi}(\widetilde{\kappa}-\Omega(\vec{n}))|+\sup_{n_3}\sum_{n_2 }|\widehat{\varphi}(\widetilde{\kappa}-\Omega(\vec{n}))|\lesssim 1.
		\]
		Thanks to the fact that $\beta_2=\beta_3>\frac{1}{4}$, we obtain the bound
		\[
		\sum_{n_2,n_3}
		(\lambda_{n_2}\lambda_{n_3})^{\frac{1}{4}-\beta_2}\mathbf{a}_{n_2}^{(2)}(\kappa_2)\mathbf{a}_{n_3}^{(3)}(\kappa_3)|\ov{\widehat{\varphi}}(\widetilde{\kappa}-\Omega(\vec{n}))|
		\lesssim \|\mathbf{a}_{n_2}^{(2)}(\kappa_2)\|_{\ell_{n_2}^2}
		\|\mathbf{a}_{n_3}^{(3)}(\kappa_2)\|_{\ell_{n_3}^2}.
		\]
		As in the previous case, we conclude from H\"older's inequality to control the integral over~$\kappa_j$.
		\medskip
		
		To prove \eqref{weaker2'}, the analysis is exactly the same as before, but we simply control $\|F_{2}\|_{L_t^{\infty}L_x^{\infty}}$ by $\|F_{2}\|_{X^{2,b}}$ using the Sobolev embedding. This completes the proof of Lemma~\ref{lem:weaker'}.
	\end{proof}
	Finally we establish an estimate to deal with the product $F_2\hexagonal{\neq}F_3$ when $F_2,F_3$ are both in the Fourier-Lebesgue space $ X_{q,q,\infty}^{\frac{1}{2}-\frac{4}{q}-3\delta,\gamma}
	$.
	\begin{lemme}\label{lem:weaker''}
		Assume that $q\in[2,\infty)$ and $\gamma\in\big(\frac{1}{q'},1\big)$. For any $\gamma_0\in\big(\frac{1}{q},1\big)$~$\delta~>~0,~b>~\frac{1}{2}$ such that $\frac{1}{2}-\frac{4}{q}-3\delta>0$,uniformly in $n$ and $Z_1\in X_{q,2}^{0,\gamma}(E_n)$, 
		\begin{multline}
		\label{weaker1''}
		\|\varphi(t)\pi_n\mathcal{N}_{(0,,1)[2,3]}(Z_{1},F_2,F_3)\|_{X_{q,2}^{0,-\gamma_0}(E_n)}
		\\
		\lesssim_{q,\gamma,\delta} 
		\|Z_{1}\|_{X_{q,2}^{0,\gamma}(E_n)}\|F_2\|
				_{
					X_{q,q,\infty}^{\frac{1}{2}-\frac{4}{q}-3\delta,\gamma}
				}			
		\|F_3\|_{
			X_{q,q,\infty}^{\frac{1}{2}-\frac{4}{q}-3\delta,\gamma} 
		},
		\end{multline}
		and
		\begin{multline}
		\label{weaker2''}
		 \|\varphi(t)\pi_n\mathcal{N}_{(0,1)[2,3]}(Z_{1},F_2,F_3)\|_{X_{q,2}^{0,0}(E_n)}
		\\
		\lesssim_{q_j,\gamma_j} \|Z_{1}\|_{X_{q,2}^{0,\gamma}(E_n)}\|F_2\|_{X_{q,q,\infty}^{1,\gamma}}\|F_3\|_{X_{q,q,\infty}^{1,\gamma}}.
		\end{multline}
	\end{lemme}

	\begin{proof}
		The proof is also very similar to the proof of Lemma \ref{lem:weaker}. To prove \eqref{weaker1''}, we adjust the function 
		\[ \mathbf{a}_{n_3}^{(3)}(\kappa_3):=\lambda_{n_3}^{\beta_3}\|\lg\kappa_3\rg^{\gamma}\pi_{n_3}\widetilde{F}_3(\kappa_3,\cdot) \|_{L_x^{\infty}},\quad \beta_3=\frac{1}{2}-\frac{4}{q}-3\delta\,,
		\]
		and $\beta_2=\beta_{3}$ is unchanged.
		By Cauchy-Schwarz, 
		\begin{multline*} 
		\big|\int_{\S^2}\widetilde{Z}_1(\kappa_1,x)\pi_{n_2}\ov{\widetilde{F}}_2(\kappa_2,x)\pi_{n_3}\widetilde{F}_{3}(\kappa_3,x)\ov{\widetilde{G}}_n(\kappa_0,x)\mathrm{d}x\big|
		\\ 
		\lesssim (\lambda_{n_2}\lambda_{n_3})^{-\beta_{2}} 
		\frac{
				\mathbf{a}_n^{(0)}(\kappa_0)
				\mathbf{a}_{n}^{(1)}(\kappa_1)
				\mathbf{a}_{n_2}^{(2)}(\kappa_2)
				\mathbf{a}_{n_3}^{(3)}(\kappa_3)
			}{\lg\kappa_0\rg^{\gamma_0}
			\lg\kappa_1\rg^{\gamma}\lg\kappa_2\rg^{\gamma}\lg\kappa_3\rg^{\gamma}},
		\end{multline*}
		and we need to control
		\begin{multline*}
			\int_{\R^4}\frac{\mathbf{a}_n^{(0)}(\kappa_0) \mathbf{a}_n^{(1)}(\kappa_1)d \vec{\kappa}}{
				\lg\kappa_0\rg^{\gamma_0}
				\lg \kappa_1\rg^{\gamma}\lg \kappa_2\rg^{\gamma}\lg\kappa_3\rg^{\gamma} }\sum_{\substack{n_2,n_3\\ n_2\neq n_3 } }
			\mathbf{a}_{n_2}^{(2)}(\kappa_2)\mathbf{a}_{n_3}^{(3)}(\kappa_3)
			|\ov{\widehat{\varphi}}(\widetilde{\kappa}-\Omega(\vec{n}))|
			(\lambda_{n_2}
			\lambda_{n_3})^{-\beta_3}\\
			\leq 
			\int_{\R^4}\frac{\mathbf{a}_n^{(0)}(\kappa_0) \mathbf{a}_n^{(1)}(\kappa_1)}{\lg\kappa_0\rg^{\gamma_0}
				\lg \kappa_1\rg^{\gamma}\lg \kappa_2\rg^{\gamma}\lg\kappa_3\rg^{\gamma}}\|\mathbf{a}_{n_2}^{(2)}(\kappa_2)\|_{\ell_{n_2}^q}
			\|\mathbf{a}_{n_3}^{(3)}(\kappa_3)\|_{\ell_{n_3}^q}
			\big(
			\sum_{\substack{n_2,n_3\leq n\\
					n_2\neq n_3 }}|\ov{\widehat{\varphi}}(\widetilde{\kappa}-\Omega(\vec{n}))|
					(\lambda_{n_2}
					\lambda_{n_3})^{-\beta_{3}}
					\big)
					\mathrm{d}\vec{\kappa}\,,
		\end{multline*}
	where we used $\ell^{q}\hookrightarrow \ell^{\infty}$.
		By the divisor bound (here we make use of the fact $n_2\neq n_3$):
		\[ \sup_{l}\#\{n_2\neq n_3: n_j\sim N_j, \lambda_{n_2}^2-\lambda_{n_3}^2=l \}\lesssim_{\epsilon} (N_2N_3)^{\epsilon},
		\]
		the sum in the second line is uniformly bounded. By H\"older, we obtain \eqref{weaker1''}.
		\medskip

		Next, to prove \eqref{weaker2}, by Lemma \ref{Hausdorff-Young}, 
		\begin{align*}
			\|\varphi(t)\pi_n\mathcal{N}_{(0,1)[2,3]}&(Z_{1},F_2,F_3)\|_{X_{q,2}^ {0,0}(E_n)}\\ \lesssim & \|\varphi(t)\pi_n\mathcal{N}_{(0,1)[2,3] }(Z_{1},F_2,F_3)\|_{L_t^{q'}L_x^2}\\
			\lesssim &\Big\|\sum_{n_2,n_3}\|Z_{1}(t)\|_{L_x^2}\|\pi_{n_2}F_2(t)\|_{L_{x}^{\infty}}\|\pi_{n_3}F_3(t)\|_{L_{x}^{\infty}}\Big\|_{L_t^{\infty}}\\
			\leq & \|Z_{1}(t)\|_{L_t^{\infty}L_x^2}
			\big\|  \|\lambda_{n_2}\pi_{n_2}F_2(t)\|_{\ell_{n_2}^{q}L_{x}^{\infty}} 
			\|\lambda_{n_3}\pi_{n_3}F_3(t)\|_{\ell_{n_3}^qL_x^{\infty}}  \big\|_{L_t^{\infty}},
		\end{align*}
		where to the last step, we used the H\"older inequality for the sum and the fact that 
		\[
		\|\lambda_{n_2}^{-1}\|_{\ell_{n_2}^{q'}}<\infty, \|\lambda_{n_3}^{-1}\|_{\ell_{n_3}^{q'}}<\infty\,.
		\]
		Finally, by Minkowski and the embedding property $X_{q,\infty}^{0,\gamma}(E_n)\hookrightarrow L_t^{\infty}L_x^{\infty}(E_n)$ (Lemma~\ref{Embedding}), we conclude the proof.
		\end{proof}

\subsection{Proof of main propositions} We are now ready to prove the main Propositions of this section, by using an interpolation argument.
	\begin{proof}[Proof of Proposition \ref{resonant:trilinear}, \ref{resonant:trilinear'} and \ref{lem:resonanttrilinear''}]
		It follows from our choice of parameters~\eqref{constraint:parametersq} that for $q$ sufficiently large we must have
		\begin{equation}
		\begin{split}
		\label{constraint:interpolation}
			\frac{1-\gamma-\theta}{\gamma_0}\cdot (s-\frac{1}{4}-\frac{4}{q}-4\delta) 
			&>
			(2-s)\cdot\frac{\gamma_0-(1-\gamma-\theta)}{\gamma_0},\\ 
			\frac{2}{q}
					\cdot\frac{1-\gamma-\theta}{\gamma_0}
			&>
			\big(\frac{1}{2}+\frac{2}{q}+3\delta\big)\cdot\frac{\gamma_0-(1-\gamma-\theta)}{\gamma_0}.
			\end{split}
		\end{equation}
		Indeed, in the above inequalities, the left-hand side are at least of order $O(q^{-2})$ while the right-hand side are of order $O(q^{-4})$.
		
		To prove the desired inequality in Proposition \ref{resonant:trilinear}, we will apply \eqref{weaker1},\eqref{weaker2} to frequency-localized $F_2,F_3$ and then interpolate. More precisely, we decompose
		\[
		 F_j=\sum_{N_j}\mathbf{P}_{N_j}F_j,\quad j=2,3.
		\]
	 For fixed $N_2,N_3$, applying \eqref{weaker1} with $q_1=q_2=q$ gives
		\begin{multline*}
			\|\varphi(t)\pi_n\mathcal{N}_{(0,1)}
			(Z_{1},\mathbf{P}_{N_2}F_2,\mathbf{P}_{N_3}F_3)\|_{X_{q,2}^{0,-\gamma_0}(E_n)}
			\\
			\lesssim 
			\|Z_{1}\|_{X_{q,2}^{0,\gamma}(E_n)}\|\mathbf{P}_{N_2}F_2\|_{
				X_{q,q,\infty}^{\frac{1}{2}-\frac{4}{q}-3\delta,\gamma}
			}\|\mathbf{P}_{N_3}F_3\|_{
			X^{\frac{1}{4}+\frac{4}{q}+4\delta,b}
		}\\
			\lesssim 
			N_2^{-\frac{2}{q}}
			N_3^{-(s-\frac{1}{4}-\frac{4}{q}-4\delta) }
			\|Z_{1}\|_{X_{q,2}^{0,\gamma}(E_n)}
			\|\mathbf{P}_{N_2}F_2\|_{
				X_{q,q,\infty}^{\frac{1}{2}-\frac{2}{q}-3\delta,\gamma}
			}
			\|\mathbf{P}_{N_3}F_3\|_{X^{s,b}}.
		\end{multline*}
		From \eqref{weaker2},
		\begin{multline*}
			\|\varphi(t)\pi_n\mathcal{N}_{(0,1)}(Z_{1},\mathbf{P}_{N_2}F_2,\mathbf{P}_{N_3}F_3)\|_{X_{q,2}^{0,0}(E_n)}
			\\ 
			\lesssim \|Z_{1}\|_{X_{q,2}^{0,\gamma}(E_n)}\|\mathbf{P}_{N_2}F_2\|_{X_{q,q,\infty}^{1,\gamma}}\|\mathbf{P}_{N_3}F_3\|_{X^{2,b}}
			\\
			\lesssim 
			N_2^{\frac{1}{2}+\frac{2}{q}+3\delta }
			N_3^{2-s}\|Z_{1}\|_{X_{q,2}^{0,\gamma}(E_n)}\|\mathbf{P}_{N_2}F_2\|_{
				X_{q,q,\infty}^{\frac{1}{2}-\frac{2}{q}-3\delta,\gamma}
			}\|\mathbf{P}_{N_3}F_3\|_{X^{s,b}}.
		\end{multline*}
		Interpolate the above inequalities (since $-\gamma_{0}<\gamma-1+\theta<0$), we have
		\begin{multline*}
			\|\varphi(t)\pi_n\mathcal{N}_{(0,1)}(Z_{1},\mathbf{P}_{N_2}F_2,\mathbf{P}_{N_3}F_3)\|_{X_{q,2}^{0,\gamma-1+\theta}(E_n)}
			\\
			\leq 
			\|\varphi(t)\pi_{n}\mathcal{N}_{(0,1)}(Z_{1},\mathbf{P}_{N_2}F_2,\mathbf{P}_{N_3}F_3)\|_{X_{q,2}^{0,-\gamma_0}(E_n)}^{\frac{1-\gamma-\theta}{\gamma_0}}
			\|\varphi(t)\pi_{n}\mathcal{N}_{(0,1)}(Z_{1},\mathbf{P}_{N_2}F_2,\mathbf{P}_{N_3}F_3)\|_{X_{q,2}^{0,0}(E_n)}^{1-\frac{1-\gamma-\theta}{\gamma_0}}.
		\end{multline*}
		Thanks to \eqref{constraint:interpolation}, we gain negative powers in $N_2,N_3$ for the right hand side, which is summable. This proves the first estimate of Proposition \ref{resonant:trilinear}. The second one is similar and we omit the detail.

		For Proposition \ref{resonant:trilinear'} and Proposition \ref{lem:resonanttrilinear''}, we only need to adjust the argument above. 	
		By Lemma \ref{lem:weaker'} and interpolation,
		\begin{multline*}
			\|\varphi(t)\pi_n\mathcal{N}_{(0,1)}(Z_{1},\mathbf{P}_{N_2}F_2,\mathbf{P}_{N_3}F_3)\|_{X_{q,2}^{0,\gamma-1+\theta}(E_n)}\\ 
			\lesssim  (N_2N_3)^{-(s-\frac{1}{4}-\frac{2}{q}-4\delta)\frac{1-\gamma-\theta}{\gamma_0}}(N_2N_3)^{(2-s)\frac{\gamma_0-(1-\gamma-\theta)}{\gamma_0}}
			\\
			\|Z_{1}\|_{X_{q,2}^{0,\gamma}(E_n)}
			\|\mathbf{P}_{N_2}F_2\|_{X^{s,b}}
			\|\mathbf{P}_{N_3}F_3\|_{X^{s,b}},
		\end{multline*}	
		where the powers of $N_2,N_3$ are negative. The same argument by Lemma \ref{lem:resonanttrilinear''}, we have
		\begin{multline*}
			\|\varphi(t)\pi_n\mathcal{N}_{(0,1)[2,3]}(Z_{1},\mathbf{P}_{N_2}F_2,\mathbf{P}_{N_3}F_3)\|_{X_{q,2}^{0,\gamma-1+\theta}(E_n)}
			\\
			\lesssim (N_2N_3)^{-\frac{2}{q}\cdot\frac{1-\gamma-\theta}{\gamma_0}}(N_2N_3)^{(\frac{1}{2}+\frac{2}{q}+3\delta)\frac{\gamma_0-(1-\gamma-\theta)}{\gamma_0}}
			\\
			\|Z_{1}\|_{X_{q,2}^{0,\gamma}(E_n)}
			\|\mathbf{P}_{N_2}F_2\|_{d}{
				X_{q,q,\infty}^{\frac{1}{2}-\frac{2}{q}-3\delta,\gamma }
			}
			\|\mathbf{P}_{N_3}F_3\|_{X_{q,q,\infty}^{ 
					\frac{1}{2}-\frac{2}{q}-3\delta,\gamma}
				}.
		\end{multline*}	
		These bounds are conclusive, thanks to \eqref{constraint:interpolation}.
	\end{proof}
		

	\section{Large deviation estimates for linear random object}
	
\label{sec:large-dev}
	
	In this section we prove the bounds claimed in $\mathrm{Loc}(N)$ (Definition~\ref{loc(N)}) at step $2N$ for the random objects. An important ingredient is the equivalence in law of product of terms of type (C). 
	
	\subsection{Equivalence in law}

		Recall that $\mathcal{B}_{\leq N}$ is the $\sigma$-algebra generated by Gaussians~$(g_{n,k})_{|k|\leq n, n\leq N}$, and that for $n\in\N$,  $E_{n}$ is the eigenspace associated to the eigenvalue $\lambda_{n}^2$. Equipped with the $L^2(\S^{2})$ norm, this space is isometric to $\mathbb{C}^{2n+1}$. For~$n\approx N$ and $|t|\leq T$, the colored term is
		\[
		e_n^{N,\dag}(t)=\mathcal{H}_n^{N,\dag}(t)(e_n^{\omega})\,,
		\]	
	where $\mathcal{H}_n^{N,\dag}(t)$ is an unitary operator on $E_n$ for $|t|\leq T$ and is $\mathcal{B}_{\leq N/2}$ measurable. 
	\begin{lemme}[law-equivalence]\label{loiequiavlent} 
		Let $l\in\N^*$ and $n_1,\cdots, n_l\in (N/2,N]$ (not necessarily distinct). Let $F$ be a bounded Borel measurable functional on $\mathbb{C}\times E_{n_{1}}\times\cdots\times E_{n_{l}}$, and $Y$ be $\mathcal{B}_{\leq N/2}$-measurable random variable. Then for all $|t|\leq T$ and almost surely in $\omega$,
		\[
		\mathbb{E}\big[F\big(Y,\mathcal{H}_{n_1}^{N,\dag}(t)(e_{n_1}^{\omega}),\cdots, \mathcal{H}_{n_l}^{N,\dag}(t)(e_{n_l}^{\omega})\big)|\mathcal{B}_{\leq \frac{N}{2}}\big]=\mathbb{E}[F(Y,e_{n_1}^{\omega},\cdots,e_{n_l}^{\omega})|\mathcal{B}_{\leq  \frac{N}{2}}]\,.
		\]
		In particular, if $F$ is of the product form $F(Y,\cdot)=Y\cdot G(\cdot)$, then for all $|t|\leq T$ and almost surely in $\omega$, 
		\[
		\mathbb{E}\big[YG\big(\mathcal{H}_{n_1}^{N,\dag}(t)(e_{n_1}^{\omega}),\cdots, \mathcal{H}_{n_l}^{N,\dag}(t)(e_{n_l}^{\omega})\big)|\mathcal{B}_{\leq \frac{N}{2}}\big]
		=Y\cdot\mathbb{E}[G(e_{n_1}^{\omega},\cdots,e_{n_l}^{\omega})]\,.
		\]
	\end{lemme}
	\begin{proof} 
		We first do some preliminary reductions. We identify the eigenspace $E_{n}$ as $\mathbb{C}^{2n+1}$ endowed with the $L^2(E_{n})$, which is isometric to the canonical norm of $\mathbb{C}^{2n+1}$. Let us denote by $\mathbf{U}(2n+1)$ the unitary group on $E_{n}=\mathbb{C}^{2n+1}$. Up to modifying the functional $F$, we can assume that $n_1,n_2,\cdots,n_l$ are distinct. Denote by
		\[
		X:=(e_{n_1}^{\omega},\cdots,e_{n_l}^{\omega}),\; \mathcal{H}(X):=\big(\mathcal{H}_{n_1}^N(e_{n_1}^{\omega}),\cdots, \mathcal{H}_{n_l}^N(e_{n_l}^{\omega})\big)\,.
		\]
		We need to show that for any bounded Borel function $G$,
		\begin{equation}
		\label{tojustify} 
		\mathbb{E}[F(Y,\mathcal{H}(X))G(Y)]=\mathbb{E}[F(Y,X)G(Y)]\,.
		\end{equation}
		Since $\mathcal{B}_{\leq N/2}$ is generated by $Z:=(e_{n}^{\omega})_{n\leq N/2}$, then almost surely, $Y=\varphi(Z)$ for some Borel measurable function $\varphi$ on the product space $\Pi_{n\leq N/2}E_{n}$.
		Similarly, since $\mathcal{H}_{n_j}^N$ are $\mathcal{B}_{\leq N/2}$ measurable there exist some $\mathbf{U}(2n+1)$-valued Borel functions $U_{n_j}$ such that $\mathcal{H}_{n_j}^{N}=U_{n_j}(Z)$.
		\medskip
		
		Denote by $\mu_Z(\cdot)$ the distribution of $Z$ on $\Pi_{n\leq N/2}\mathbb{C}^{2n+1}$ and $\mu_X(\cdot)$ the  distribution of $X$ on $\Pi_{j=1}^{l}\mathbb{C}^{2n_j+1}$. By independence between $X$ and $Z$, the joint distribution of $(Z,X)$ is $\mu_Z\otimes\mu_X$. Thus
		\begin{multline*}
			\mathbb{E}[F(Y,\mathcal{H}(X))G(Y)]\\
			=\int_{\prod_{n\leq \frac{N}{2}}\mathbb{C}^{2n+1} }
			G(\varphi(\mathbf{z}))
			\Big(
			\int_{\prod_{j=1}^{l}\C^{2n_{j}+1}}F(\varphi(\mathbf{z}),U_{n_1}(\mathbf{z})x_{n_1},\cdots,U_{n_l}(\mathbf{z})x_{n_l})d\mu_X(\mathbf{x})
			\Big)
			d\mu_{Z}(\mathbf{z})
			\,.
		\end{multline*}
		It suffices to show that for almost every  $\mathbf{z}$,
		\begin{multline*}
		\int_{\prod_{j=1}^{l}\C^{2n_{j}+1}}
		F(\varphi(\mathbf{z}),U_{n_1}(\mathbf{z})x_{n_1},\cdots,U_{n_l}(\mathbf{z})x_{n_l})d\mu_X(\mathbf{x})
		\\
		=
		\int_{\prod_{j=1}^{l}\C^{2n_{j}+1}}
		F(\varphi(\mathbf{z}),x_{n_1},\cdots,x_{n_l})d\mu_X(\mathbf{x}).
		\end{multline*}
		This amounts to show that for fixed $\mathbf{z}$, the random variables
		\[
		U(X)=:(U_{n_1}(\mathbf{z})e_{n_1},\cdots,U_{n_l}(\mathbf{z})e_{n_l})\quad  \text{and}\quad X=(e_{n_1},\cdots,e_{n_l})
		\]
		have the same law. This follows from the fact that
		\[
		\mu_X=\otimes_{j=1}^l\mathcal{N}_{\mathbb{C}^{2n_j+1}}(0,1)\,,
		\]
		and from Lemma \ref{lemme:law} (the unitary group leaves the complex normal distribution invariant). This yields \eqref{tojustify} and completes the proof of Lemma~\ref{loiequiavlent}.
		\end{proof}
	
	\subsection{Fourier-Lebesgue norm of the colored term}
	In the first Lemma, we deduce some bounds on the Fourier-Lebesgue norms of colored terms from the random information encoded in the random averaging operators (RAOs). Recall the definitions ~\eqref{def:G2N}, \eqref{def:h2Nextension},~\eqref{def:e2Nextension} and~\eqref{def:pinpsi2N} of
	$\mathcal{G}_n^{2N,\dag}, \mathbf{h}_n^{2N,\dag}, \mathbf{g}_n^{2N,\dag} $ and $ \mathbf{f}_n^{2N,\dag}
	$ respectively. 
	\begin{lemme}[Fourier-Lebesgue bounds]\label{bd:colored}
		Assume that $\Xi$ is a $\mathcal{B}_{\leq N}$-measurable set such that on $\Xi$, 
		\[ \|\mathbf{g}_n^{2N,\dag}(t)\|_{S_n^{q,\gamma,*}}\leq R^{-1},\quad 
		\|\mathcal{G}_n^{2N,\dag}(t)\|_{S_n^{q,\gamma,*}}\leq R\,.
		\]
		Then there exists a $\mathcal{B}_{\leq 2N}$-measurable set $\Xi'$ such that $\mathbb{P}(\Omega\setminus\Xi')<C_0\e^{-c_0(NR)^{\delta_0}}$, such that on $\Xi\cap\Xi'$,
		\begin{align*} 
			\|\mathbf{f}_n^{2N,\dag}\|_{X_{q,\infty}^{0,\gamma}(E_n)}\leq N^{\frac{2}{q}+\delta}R,
			\quad	
			\|\psi_{2N}^{\dag}\|_{X_{q,q,\infty}^{0,\gamma}}\leq T^{-\gamma+\frac{1}{q'}}R
			N^{-(\alpha-\frac{1}{2})+\frac{2}{q}+\delta},
		\end{align*}
where $C_0, c_{0},\delta_{0}>0$ are parameters depending only on $q,\delta,\gamma$.
	\end{lemme}

\begin{remarque}
Assuming \(\mathrm{Loc}(N)\), we will need to verify the bounds for the objects  \(\mathbf{g}_{n}^{2N,\dag}\) and \(\mathcal{G}_{n}^{2N,\dag}\) supposed in the statement of the Lemma. This is done in paragraph \ref{sec:s1}, Step 1. 
\medskip
\end{remarque}

\begin{remarque}
Note that this Lemma will be used to propagate the bound \eqref{eq:p-psi-} in the induction scheme, to go from $\mathrm{Loc}(N)$ to $\mathrm{Loc}(2N)$. 
\end{remarque}

	\begin{proof}
		Thanks to Lemma \ref{concentration} on the \(L^{p}\)-bounds of random spherical harmonics and to the Sobolev embedding, we have 
		\[
		\|\chi_{T}(t)\e^{-it\lambda_{n}^{2}}e_{n}^{\omega}\|_{X_{q,q,\infty}^{0,\gamma}}\lesssim_{q} T^{-\gamma+\frac{1}{q'}}RN^{\frac{2}{q}+\delta}\,,
		\] 
		outside a set of probability smaller than $\mathcal{O}(\e^{-c_0N^{\theta}R^{2}})$. Then, according to the decomposition \eqref{def:pinpsi2N}, it suffices to prove the bound for $\mathbf{f}_n^{2N,\dag}$ to obtain the bound for~$\psi_{2N}^{\dag}$. 
		\medskip
		
		 Given an orthonormal basis $(\mathbf{b}_{n,k})_{|k|\leq n}$ of $E_{n}$, recall that $\mathbf{h}_{n;\ell,k}^{2N,\dag}(t)$ denotes the matrix-element of $\mathbf{h}_{n}^{2N,\dag}(t)$.
		\medskip

		Let us first make two observations. Thanks to the eigenfunction estimate, by losing $N^{\frac{1}{q}}$, it suffices to estimate the $X_{q,q}^{0,\gamma}(E_n)$-norm of $\mathbf{f}_{n}^{\dag}$ and to prove the bound
		\[ \|\mathbf{f}_{n}^{\dag} \|_{X_{q,q}^{0,\gamma}(E_n)}\leq N^{\delta}R
		\]
		outside of the set with probability smaller than $\mathcal{O}(\e^{-c_0(NR)^{\theta}})$. 
		Moreover, since $\mathbf{f}^{\dag}$ is compactly supported on $[-T,T]$, we may replace the operator $\mathbf{h}^{\dag}$ (which is supported on $[-2T,2T]$) by $\eta(t/T)\mathbf{h}^{\dag}$ for some $\eta\in C_c^{\infty}((-1,1))$ such that $\eta(t/T)\mathbf{f}^{\dag}=\mathbf{f}^{\dag}$. 
		Since $(\mathbf{h}^{\dag})^*=\mathbf{g}^{\dag}$ on $|t|\leq T$, we have
		$\big(\mathbf{h}^{\dag} \big)^{*}=\eta(t/T)\mathbf{g}^{\dag}$.
		This simple observation allows us to replace the operator norm of the adjoint $(\mathbf{h}^{\dag})^*$ by the operator norm $\mathbf{g}^{\dag}$ in the arguments below.
\medskip

	Recall the notations introduced in Section~\ref{sec:4}. From the discussion above, we implicitly insert $\eta(t/T)$ in the operators $\mathbf{h}^{\dag}(t),\mathbf{g}^{\dag}(t)$ and obtain
		\begin{align*}
			\lg\lambda\rg^{\gamma}	\widetilde{\mathbf{f}}_{n}^{\dag}(\lambda,x)=\sum_{|k|,|\ell|\leq n}\langle\lambda\rangle^{\gamma}\widetilde{\mathbf{h}}^{\dag}_{\ell,k}(\lambda)\frac{g_{n,k}(\omega)}{\sqrt{2n+1}}\mathbf{b}_{n,\ell}(x).
		\end{align*}
		By Chebyshev and Minkowski, for $q<\infty$ and $p\geq q$, 
		\begin{align*}
			\mathbb{P}\big[\|\mathbf{1}_{\Xi}\cdot \lg\lambda\rg^{\gamma}	\widetilde{\mathbf{f}}_{n}^{\dag}(\lambda,x) \|_{L_{\lambda,x}^q}>N^{\delta}R \big]\leq & \frac{1}{(N^{\delta}R)^p}\|\mathbf{1}_{\Xi}\cdot\lg\lambda\rg^{\gamma}	\widetilde{\mathbf{f}}_{n}^{\dag}(\lambda,x) \|_{L_{\omega}^pL_{\lambda,x}^q}^p\\
			=&\frac{1}{(N^{\delta}R)^p}\Big\|
			\mathbf{1}_{\Xi}\cdot\big\|\|\lg\lambda\rg^{\gamma}\widetilde{\mathbf{f}}_{n}^{\dag}(\lambda,x)\|_{L_{\lambda,x}^q}\big\|_{L_{\omega}^p|\mathcal{B}_{\leq N}}
			\Big\|_{L_{\omega}^p}^p
			\\
			\leq & \frac{1}{(N^{\delta}R)^p}
			\Big\|
			\mathbf{1}_{\Xi}\cdot
			\big\|
			\|
			\lg\lambda\rg^{\gamma}\widetilde{\mathbf{f}}_{n}^{\dag}(\lambda,x)\|_{L_{\omega}^p|\mathcal{B}_{\leq N}}\big\|_{L_{\lambda,x}^q}
			\Big\|_{L_{\omega}^p}^p
		\end{align*}
		Recall that the matrices $\widetilde{\mathbf{h}}_{\ell,k}^{\dag}(\lambda)$ are independent of the Borel $\sigma$-algebra generated by Gaussians $\mathcal{B}_{2N}$. Therefore, for fixed $\lambda,x$ the conditional Wiener-chaos estimate yields
		\begin{align*}
			\|\mathbf{1}_{\Xi}\cdot\lg\lambda\rg^{\gamma}	\widetilde{\mathbf{f}}_{n}^{\dag}(\lambda,x) \|_{L_{\omega}^p|\mathcal{B}_{\leq N}}\leq &  Cp^{\frac{1}{2}}\Big\|\mathbf{1}_{\Xi}\langle\lambda\rangle^{\gamma}\cdot\sum_{|k|,|\ell|\leq n}\widetilde{\mathbf{h}}^{\dag}_{\ell,k}(\lambda)\frac{g_{n,k}(\omega)}{\sqrt{2n+1} }\mathbf{b}_{n,\ell}(x) \Big\|_{L_{\omega}^2|\mathcal{B}_{\leq N}}.
		\end{align*}
		Taking $L_{\lambda}^q$ of the inequality above, we have for fixed $x$,
		\begin{align}\label{eq:LDV1} 
			\|\mathbf{1}_{\Xi}\lg\lambda\rg^{\gamma}\widetilde{\mathbf{f}}_{n}^{\dag}(\lambda,x)\|_{L_{\lambda}^qL_{\omega}^p|\mathcal{B}_{\leq N}}
			\leq
			\frac{Cp^{\frac{1}{2}}}{\sqrt{2n+1}}\Big\|\Big(\sum_{|k|\leq n}\Big| 
			\sum_{|\ell|\leq n}\langle\lambda\rangle^{\gamma}\widetilde{\mathbf{h}}^{\dag}_{\ell,k}(\lambda)\mathbf{b}_{n,\ell}(x)\Big|^2
			\Big)^{\frac{1}{2}}\Big\|_{L_{\lambda}^q}.
		\end{align}
		Applying \eqref{pointwiseproba} to $\mathcal{T}_n=\mathcal{\mathbf{h}}^{\dag}$, we obtain the bound
		\[
		\Big\| \Big(\sum_{|k|\leq n}\Big| 
		\sum_{|\ell|\leq n}\langle\lambda\rangle^{\gamma}\widetilde{\mathbf{h}}^{\dag}_{\ell,k}(\lambda)\mathbf{b}_{n,\ell}(x)\Big|^2
		\Big)^{\frac{1}{2}}\Big\|_{L_{\lambda}^{q}}
		\leq
		C\sqrt{2n+1}\|\lg\lambda\rg^{\gamma}(\widetilde{\mathcal{\mathbf{h}}^{\dag}}(\lambda))^*\|_{E_n\rightarrow L_{\lambda}^qE_n}.
		\]
		Plugging into \eqref{eq:LDV1} and using the fact that $(\mathbf{h}^{\dag}(t))^*=\eta(t/T)\mathbf{g}^{\dag}(t)$,
		we finally obtain the bound
		\begin{align*}
			\eqref{eq:LDV1}\leq &Cp^{\frac{1}{2}}\|\lg\lambda\rg^{\gamma}(\widetilde{\mathbf{h}}^{\dag}(\lambda))^* \|_{E_n\rightarrow L_{\lambda}^qE_n}\\=&Cp^{\frac{1}{2}}
			\|\lg\lambda\rg^{\gamma}\mathcal{F}_t(\e^{-it\lambda_n^2}\mathbf{g}^{\dag}(t) )(-\lambda)  \|_{E_n\rightarrow L_{\lambda}^qE_n}
			=
			Cp^{\frac{1}{2}}\|\eta\big(\frac{t}{T}\big)\mathbf{g}^{\dag}(t)\|_{S_n^{q,\gamma,*}}.
		\end{align*}
		Since $\mathbf{g}^{\dag}(0)=0$, by Corollary \ref{timecutoffappendix} , the right hand side is bounded by 
		\[
		Cp^{\frac{1}{2}}\|\mathbf{g}^{\dag}\|_{S_n^{q,\gamma,*}} \leq Cp^{\frac{1}{2}}R^{-1}
		\]
		on $\Xi$. Taking the $L_x^q$-norm, optimizing the choice of $p$, and then taking the \(\ell_{n}^{q}\)-norm we obtain the desired estimate.
	\end{proof}

	\subsection{Pointwise bound on the colored terms} 
	
In this subsection, we prove large deviation estimates for norms defined in the physical space. We use the equivalence in law of Lemma \ref{lemme:law} to reduce the analysis to time $t=0$, when the random objects are much simpler. 
	\begin{lemme}[Pointwise bounds]
		\label{lem:wick} 
		Suppose that on a $\mathcal{B}_{\leq N}$-measurable set $\Omega_{N}$, the colored term $\psi_{2N}^{\dag}$ is well defined and $\mathrm{Loc}(N)$ is true with parameters $(T,R,\sigma)$ with $\sigma$ sufficiently small, and $(q,\gamma,\delta) = (q_{\sigma},\gamma_{\sigma},\delta_{\sigma})$. There exists a $\mathcal{B}_{\leq 2N}$-measurable set $\Xi'$ with $\mathbb{P}(\Omega\setminus\Xi')<C_{0}\e^{-c_{0}(NR)^{\delta_{0}}}$ such that on $\Omega_{N}\cap\Xi'$, 
		\begin{align}
			\label{eq:wick}
			\|\psi_{2N}^{\dag}\|_{L_{t}^{q}L_{x}^{\infty}}
			&\leq
			RT^{\frac{1}{q}}N^{-(\alpha-1)+\frac{1}{q}+\delta}\,,
			\\
			\|\psi^{\dag}_{2N}\hexagonal{=}\overline{\psi^{\dag}_{2N}}\|_{L_{t}^{q}L_{x}^{\infty}}
			&\leq
			R^{2}T^{\frac{1}{q}}N^{-2(\alpha-\frac{1}{2})+\frac{1}{2}+ \frac{1}{q} + \delta}\,,
			\end{align}
		where the constants $C_{0},c_{0},\delta_{0}>0$ only depend on $q,\delta,\gamma$. 
	\end{lemme}
	\begin{proof}

		As in the proof of Lemma \ref{bd:colored} (and using a Sobolev embedding) we reduce the proof to moment bounds: there exists $C>0$ such that for all $r\geq q$, 
		\begin{align}
		\label{eq:m0}
		\|\psi^{\dag}_{2N}\|_{L_{\omega}^{r}L_{t,x}^{q}}
		&\leq
		\sqrt{r}CT^{\frac{1}{q}}N^{-(\alpha-1)}\,,
		\\
			\label{eq:moment-wick}
			\| \psi^{\dag}_{2N}\hexagonal{=}\overline{\psi_{2N}^{\dag}}\|_{L_\omega^rL_{t,x}^q} 
			&\leq r CT^{\frac{1}{q}}N^{-2(\alpha-\frac{1}{2})+\frac{1}{2}}\,. 
		\end{align}
		We first prove \eqref{eq:m0}. Recall the decomposition 
		\[
		\psi_{2N}^{\dag}(t,x) = 
		\sum_{n\approx2N}
		\pi_n\psi_{2N}^{\dag}(t)
		=
		\sum_{n\approx2N}
		\lambda_n^{-(\alpha-\frac{1}{2})}e_n^{2N,\dag}(t)\,.
		\]
		According to the equivalence in law of Lemma \ref{loiequiavlent} that for fixed $t,x$
		\[
		\|\psi_{2N}^{\dag}(t,x)\|_{L_{\omega}^{r}}
		=
		\big\|
		\|
		\psi_{2N}^{\dag}(t,x)
		\|_{L_{\omega}^{r}\mid\mathcal{B}_{\leq N}}
		\big\|_{L_{\omega}^{r}}
		=
		\chi_{T}^{2}(t)
		\|
		\sum_{n\approx2N}\lambda_{n}^{-(\alpha-\frac{1}{2})}
		e_{n}^{\omega}(x)
		\|_{L_{\omega}^{r}}\,.
		\]
		Then, as in the proof of Lemma \ref{concentration},  we conclude by the Khinchin's inequality:
	\begin{align*}
	\|
	\sum_{n\approx2N}\lambda_{n}^{-(\alpha-\frac{1}{2})}e_{n}^{\omega}(x)
	\|_{L_{\omega}^{r}}
	&\lesssim
	\sqrt{r}
	\mathbb{E}
	\big[
	\big|
	\sum_{n\approx2N}\lambda_{n}^{-(\alpha-\frac{1}{2})}
	\frac{1}{\sqrt{2n+1}}
	\sum_{|k|\leq n}
	g_{n,k}(\omega)\mathbf{b}_{n,k}(x)
	\big|^{2}
	\big]^{\frac{1}{2}} \\
	& \lesssim \sqrt{r}
	\big(
	\sum_{n\approx2N}\lambda_{n}^{-2(\alpha-\frac{1}{2})}
	\frac{1}{2n+1}
	\sum_{|k|\leq n}
	|\mathbf{b}_{n,k}(x)|^{2}
	\big)^{\frac{1}{2}} \\
	&\lesssim\sqrt{r} N^{-(\alpha-1)}\,,
	\end{align*}
		where we used the Weyl's law \eqref{Weyl}. 
		\medskip
		
		Let us now turn to the proof of \eqref{eq:moment-wick}. According to equivalence of laws, for fixed~$t,x$,
		\begin{align*}
		\|\psi^{\dag}_{2N}\hexagonal{=}\ov{\psi_{2N}^{\dag}}(t,x)\|_{L_{\omega}^r}
		=&
		\big\|
		\|
		\sum_{n\approx2N} 
		\lambda_n^{-2(\alpha-\frac{1}{2})}
		(|e_n^{2N,\dag}(t,x)|^2-\|e_n^{2N,\dag}\|_{L^2(\S^2)}^2)
		\|_{L_{\omega}^{r}|\mathcal{B}_{\leq N}}
		\big\|_{L_{\omega}^{r}}
		 \\ 
		 = &\chi_T(t)^{2}
		 \big\|
		 \sum_{n\approx2N}
		 \lambda_n^{-2(\alpha-\frac{1}{2})}
		 (|e^{\omega}_n(x)|^2-\|e^{\omega}_n\|_{L^2(\S^2)}^2)
		 \big\|_{L_{\omega}^{r}}\,,
		\end{align*}
		Then, for $n\approx 2N$ we fix an orthonormal basis $(\mathbf{b}_{n,k})_{|k|\leq n}$ of $E_n$ and we decompose into two parts:
		\begin{align*}
			|e_n^{\omega}(x)|^2-\|e^{\omega}_n\|_{L^2(\S^2)}^2 &= \frac{1}{(2n+1)}\sum_{\substack{k\neq k'}}g_{n,k}(\omega)\overline{g}_{n,k'}(\omega)\mathbf{b}_{n,k}(x)\overline{\mathbf{b}}_{n,k'}(x)\\
			&+ \frac{1}{(2n+1)}\sum_{k}|g_{n,k}(\omega)|^2(|\mathbf{b}_{n,k}(x)|^2-1)\,.
		\end{align*}
		According to this decomposition, we write 
		\[
		\sum_{n\approx L}\lambda_n^{-2(\alpha-\frac{1}{2})}(|e_n^{\omega}(x)|^2-\|e_n^{\omega}\|_{L^2}^2) = \operatorname{I}(x)+\operatorname{II}(x)\,.
		\]
		Note that we have no dependence on $t$, thanks to the invariance of the law. Note also that on $\mathbb{T}^d$ the term $\mathrm{II}(x)$ would not appear since the amplitude of a plane wave is always one.
		$\\$
		
		\noi$\bullet${\bf Term $\operatorname{I}(x)$}: We deduce from hypercontractivity that 
		\begin{align*}
			\mathbb{E}[\operatorname{I}(x)^r]^\frac{1}{r} &= 
			\mathbb{E}\Big[ \Big| \sum_{n\approx 2N}\lambda_n^{-2(\alpha-\frac{1}{2})}\sum_{\substack{k,k'\\ k\neq k'}}g_{n,k}(\omega)\overline{g}_{n,k'}(\omega)\frac{\mathbf{b}_{n,k}(x)\overline{\mathbf{b}}_{n,k'}(x)}{2n+1} \Big|^r \Big]^\frac{1}{r} \\
			&\lesssim r \mathbb{E}\Big[ \Big| \sum_{n\approx 2N}\lambda_n^{-2(\alpha-\frac{1}{2})}\sum_{\substack{k,k'\\ k\neq k'}}g_{n,k}(\omega)\overline{g}_{n,k'}(\omega)\frac{\mathbf{b}_{n,k}(x)\overline{\mathbf{b}}_{n,k'}(x)}{2n+1} \Big|^2 \Big]^\frac{1}{2}  \\
			& \lesssim r \Big( \sum_{n\approx 2N}\lambda_n^{-4(\alpha-\frac{1}{2})}\sum_{\substack{k,k'\\ k\neq k'}}  \frac{|\mathbf{b}_{n,k}(x)|^2|\mathbf{b}_{n,k'}(x)|^2}{(2n+1)^2} \Big)^\frac{1}{2}\,,
		\end{align*}
		where we used the independence on the last line. We conclude from Weyl's law \eqref{Weyl} that 
		\[
		\mathbb{E}[|\operatorname{I}(x)|^r]^\frac{1}{r} \leq r \Big(\sum_{n\approx 2N}\lambda_n^{-4(\alpha-\frac{1}{2})} \big(\sum_{|k|\leq n}\frac{|{\bf b}_{n,k}(x)|^2}{2n+1}\big)^2\Big)^\frac{1}{2} 
		\lesssim rN^{-2(\alpha-\frac{1}{2})+\frac{1}{2}}\,.
		\]
		Note that the above bound does not depend on $x$. 
		$\\$
		
		\noi$\bullet${\bf Term $\operatorname{II}(x)$}: For the second term, we also use hypercontractivity to obtain
		\[
		\mathbb{E}
		[|\operatorname{II}(x)|^r]^\frac{1}{r}\leq r\mathbb{E}
		\Big[\Big(
		\sum_{n\approx 2N}\lambda_n^{-2(\alpha-\frac{1}{2})}\sum_{|k|\leq n}|g_{n,k}(\omega)|^2\frac{|\mathbf{b}_{n,k}(x)|^2-1}{2n+1}\Big)^2\Big]^\frac{1}{2}\,.
		\]
		We observe from Weyl's law~\eqref{Weyl} that for every $n$ and $x\in\S^2$, 
		\[
		\sum_{|k|\leq n}|g_{n,k}(\omega)|^2
		\frac{|\mathbf{b}_{n,k}(x)|^2-1}{2n+1} 
		= \sum_{|k|\leq n}(|g_{n,k}(\omega)|^2-1)\frac{|\mathbf{b}_{n,k}(x)|^2-1}{2n+1}\,.
		\]
		This allows us to use random oscillations and to deduce from the independence that  
		\begin{align*}
			\mathbb{E}[|\operatorname{II}(x)|^r]^\frac{1}{r}
			&\leq r
			\Big(\sum_{n\approx 2N}\lambda_n^{-4(\alpha-\frac{1}{2})}\sum_{|k|\leq n}\mathbb{E}[(|g_{n,k}(\omega)|^2-1)^2]\Big(\frac{|\mathbf{b}_{n,k}(x)|^2-1}{2n+1}\Big)^2\Big)^\frac{1}{2} \\
			&\leq r\Big(\sum_{n\approx 2N}\lambda_n^{-4(\alpha-\frac{1}{2})}\sum_{|k|\leq n}\Big(\frac{|\mathbf{b}_{n,k}(x)|^2-1}{2n+1}\Big)^2\Big)^\frac{1}{2} \\
			&\leq r\Big(\sum_{n\approx 2N}\lambda_n^{-4(\alpha-\frac{1}{2})}\frac{1}{(2n+1)^{2}}\Big(\sum_{|k|\leq n}\Big||\mathbf{b}_{n,k}(x)|^2-1\Big|\Big)^2\Big)^\frac{1}{2} \,,
		\end{align*}
		Applying the Weyl's law~\eqref{Weyl} we conclude that
		\[
		\mathbb{E}[|\operatorname{II}(x)|^r]^\frac{1}{r}\lesssim r \Big(\sum_{n\approx 2N}\lambda_n^{-4(\alpha-\frac{1}{2})-1}\Big)^\frac{1}{2}\lesssim rN^{-2(\alpha-\frac{1}{2})}\,.
		\]
		This gives~\eqref{eq:moment-wick} and completes the proof of Lemma~\ref{lem:wick}.
	\end{proof}
Summarizing Lemmas \ref{bd:colored} and \ref{lem:wick}, we obtain the bounds on \(\psi_{2N}^{\dag}\) claimed in $\mathrm{Loc(N)}$ (Definition \ref{loc(N)}) at step $2N$, provided we proved the bounds for the RAO's by assuming Loc(\(N\)), as detailed in paragraph~\ref{sec:s1}, Step 1.
	\medskip
	
	The rest of the paper is devoted to prove the bounds \eqref{eq:bo-r} on the remainder $w_{2N}^{\dag}$, which is solution to the equation \eqref{eq:wNdag}. We may implicitly suppose that the random objects of generation $2N$ satisfy the bound \eqref{eq:p-psi} in $\mathrm{Loc(2N)}$. 
\medskip
	


	\section{Large deviation estimates for the (C)(C)(C) terms}\label{sec:CCC}
	Given $N_{1}, N_{2}, N_{3}$ and $\psi_{N_{1}}^{\dag}, \psi_{N_{2}}^{\dag}, \psi_{N_{3}}^{\dag}$ of type (C)(C)(C) we derive some moments estimates,  from which we deduce the large deviation bounds claimed in Proposition~\ref{multilinear:random}. Note that the estimates obtained in this Section are also effective when~$\alpha~=~1$, on the support of the Gibbs measure.
	\medskip
	
	We recall that \(N_{(1)} \geq N_{(2)} \geq N_{(3)}\) denotes the decreasing rearrangement of the dyadic integers \(N_{1}, N_{2}, N_{3}\), and that, in the multilinear estimates at step \(2N\) of the induction scheme, we always assume \(N_{(1)} = 2N\). 
We also remind the reader that the hierarchy of parameters such as \(q\) and \(\delta\) is governed by the common small parameter \(\sigma\), as specified in equations~\eqref{smallparameters},~\eqref{smallparameters1},~\eqref{smallparameters2}, and~\eqref{smallparameters3}.
\begin{proposition}
\label{prop:CCC}
Let \(\alpha \geq 1\) and \(b_{1} > \frac{1}{2}\). Suppose that \(\operatorname{Loc}(N)\) holds on \(\Omega_{N}\) for the parameters \((T, R, \sigma)\), and that \(\Xi'\) is a \(\mathcal{B}_{\leq 2N}\)-measurable set as in Lemmas~\ref{bd:colored} and~\ref{lem:wick}. 

Then, for all \(N_{1}, N_{2}, N_{3} \leq 2N\) with \(N_{(1)} = 2N\), and for all \(r \geq q\), the following estimates hold:
\begin{itemize}
    \item If \(N_{(1)} = N_{(2)}\) or \(N_{(1)} = N_{2}\), then
    \[
    \big\| 
    \chi \mathcal{N}_{(0,1)}(\psi_{N_{1}}^{\dag}, \psi_{N_{2}}^{\dag}, \psi_{N_{3}}^{\dag}) 
    \big\|_{L_{\omega}^{r}(\Omega_{N} \cap \Xi'; X^{0,-b_{1}})} 
    \lesssim_{\sigma, \alpha, b_{1}} 
    \left(r^{\frac{1}{2}} R T^{-\gamma + \frac{1}{q'}}\right)^3 
    N^{-(\alpha - \frac{1}{2}) + \frac{1}{q} + \delta}\,.
    \]
    
    \item Otherwise, assuming only that \(N_{(1)} = 2N\), we obtain
    \[
    \big\| 
    \chi \mathcal{N}_{[0,1]}(\psi_{N_{1}}^{\dag}, \psi_{N_{2}}^{\dag}, \psi_{N_{3}}^{\dag}) 
    \big\|_{L_{\omega}^{r}(\Omega_{N} \cap \Xi'; X^{0,-b_{1}})} 
    \lesssim_{\sigma,\alpha,b_{1}} 
    \left(r^{\frac{1}{2}} R T^{-\gamma + \frac{1}{q'}}\right)^3 
    N^{-(\alpha - \frac{1}{2}) + \frac{1}{q} + \delta}\,.
    \]
\end{itemize}
\end{proposition}

\medskip

\noindent
We emphasize that the constraint \(N_{(1)} = N_{(2)}\) or \(N_{(1)} = N_{2}\) in the case \(\mathcal{N} = \mathcal{N}_{(0,1)}\) is essential to ensure the nonlinear smoothing effect, as shown in~\cite{BCLST}. This condition is naturally handled through our ansatz~\eqref{eq:psiN}.
	
	\begin{proof}[Proof of Proposition \ref{multilinear:random} (assuming Proposition \ref{prop:CCC})]
	
	Recall that  $b\in(\frac{1}{2},\frac{1}{2}+\theta)$. We introduce the parameter $b'=\frac{1}{2}-2\theta$, so that 
	\[
	1-b-b' >\theta\,.
	\]
	According to the standard inhomogeneous estimate recalled in Proposition \ref{inhomogeneous} (applied with \(q=r=2\))
	\[
	\|\chi_{T}\mathcal{I}(\chi\mathcal{N}(\psi_{N_{1}}^{\dag},\psi_{N_{2}}^{\dag},\psi_{N_{3}}^{\dag}))\|_{X^{0,b}} \lesssim T^{1-b-b'} \|\chi\mathcal{N}(\psi_{N_{1}}^{\dag},\psi_{N_{2}}^{\dag},\psi_{N_{3}}^{\dag}) \|_{X^{0,-b'}}\,.
	\]	
	Then, to control the right-hand-side we interpolate the bounds claimed in Proposition~\ref{prop:CCC} with the trivial estimate ($N_{(1)}=2N$)
	\[
	\|\chi\mathcal{N}(\psi_{N_1}^{\dag},\psi_{N_2}^{\dag},\psi_{N_3}^{\dag}) \|_{L_{\omega}^r(\Omega_N\cap \Xi';X^{0,0})}\lesssim_{\alpha} N^{10}R^3\,. 
	\]
This gives
	\[
	\|\chi\mathcal{N}(\psi_{N_1}^{\dag},\psi_{N_2}^{\dag},\psi_{N_3}^{\dag}) \|_{L_{\omega}^r(\Omega_N\cap \Xi
	';X^{0,-b'})}
	\lesssim 
	(2r)^{\frac{3b'}{2b_1}}
	R^{3}
	T^{-3\frac{b'}{b_1}(\gamma-\frac{1}{q'})}N^{-\frac{b'}{b_1}s_1+10(1-\frac{b'}{b_1})}\,, 
	\]
where we denoted $s_{1}:=(\alpha-\frac{1}{2})-\frac{1}{q}-\delta$. It follows from our choice of parameters~\eqref{smallparameters}  that for $b_{1}>\frac{1}{2}$,
	 \[
	 1-b-b' - 3\frac{b'}{b_{1}}(\gamma-\frac{1}{q'}) >\theta - 3\sigma^{10}>\frac{\theta}{2}\,.
	 \]
This gives the desired factor $T^{\frac{\theta}{2}}$ on the right-hand-side.	Then, for \(b_{1}\) close enough to~$\frac{1}{2}$ and satisfying $\frac{1}{2}<b_1$, we have
\[
\frac{b'}{b_1}s_1-10(1-\frac{b'}{b_1})
	>s_{1}-100\theta 
	= \alpha-\frac{1}{2} - \frac{1}{q}-\delta-100\theta\,.
\]
This gives the factor \(N^{-\alpha+\frac{1}{2}+\frac{1}{q}+\delta+100\theta}\). The proof of Proposition \ref{multilinear:random} then follows from the standard argument using Chebyshev. Let us remark that the growth $R^4$ in the statement of Proposition \ref{multilinear:random} is responsible for the tail probability $C_0\exp(-N^{c_0}R^{c_0})$ of the exceptional set,  since the growth in the moment bound is~$R^3$. The same comment applies for the growth $N^{2\delta}$ in Proposition~\ref{multilinear:random} (compared to~$N^{\delta}$ in the moment bound).
	\end{proof}	
	The rest of this section is devoted to the proof of Proposition \ref{prop:CCC}. Let us start with a general probabilistic multilinear estimate used in many places of the proof. 
\begin{lemme}[$k$-linear probabilistic estimates at generation $2N$]
		\label{lem:multi-C} Fix $N\geq1$ such that $\operatorname{Loc}(N)$ holds on $\Omega_{N}$ for parameters $(R,T,\sigma)$, and that the operator bound~\eqref{eq:chi-H} holds at generation $2N$ on a $\mathcal{B}_{\leq2N}$-measurable set $\Xi'$. For all $k\geq 1$, for every  $\mathcal{B}_{\leq N}$-measurable function $f\in L_x^\infty L_\omega^r\ell_{n_1,\cdots,n_k}^{2}L_{\kappa_1,\cdots,\kappa_k}^{q'}$ and every $x$ in $\mathbb{S}^2$, 
		\begin{multline}
		\label{eq:CCC}
			\Big
			\|
			\int_{\R^k}\sum_{\substack{n_1,\cdots, n_k\\ \text{distinct}}} \Big(\prod_{i=1}^k \pi_{n_{i}}\widetilde{\psi_{2N}^{\dag}}^{\iota_{i}}(\kappa_{i},x)\Big)f_{n_1,\cdots,n_k}(\kappa_1,\cdots \kappa_k, x)  
			d\kappa_1\cdots d\kappa_k 
			\Big
			\|_{L_\omega^r(\Omega_{N}\cap\Xi')} \\
			\lesssim_{k} 
			\big\|
			\big(\prod_{i=1}^k\langle\kappa_i\rangle^{-\gamma}\big)
			f_{n_1,\cdots,n_k}
			(\kappa_1,\cdots \kappa_k, x)
			\big\|_{L_\omega^r(\Omega_{N}\cap\Xi'; \ell_{n_1,\cdots,n_k}^2L_{\kappa_1,\cdots,\kappa_k}^{q'})}
			\\
			(r^\frac{1}{2}T^{-\gamma+\frac{1}{q'}} RN^{-(\alpha-\frac{1}{2})})^k\,,
		\end{multline}
		where $\iota_{i}$ stands for the possible comlpex conjugaison bar, which is not important here. 
	\end{lemme}
	This type of estimate allows one to control the time-modulated multilinear interactions between $k$-colored Gaussian variables from the same generation $2N$,  possibly interacting with a random function $f$ constructed at a former generation, which is therefore  $\mathcal{B}_{\leq N}$-measurable.
	\begin{remarque}
		Lemma \ref{lem:multi-C} will be repeatedly applied in several parts of the multilinear probabilistic estimates, with $k\leq 3$. Hence, we neglect the dependence on $k$ of the constants.
	\end{remarque}
	\begin{proof}
It turns out that the conjugaison bar plays no role, and we suppose that $\iota_{i}=1$ for all $i\in\{1,\cdots,k\}$.  Recall that for $n\approx 2N$,
	\[
			\pi_{n}\psi_{2N}^{\dag}(t,x) = \chi_{T}(t)\lambda_{n}^{-(\alpha-\frac{1}{2})}e_n^{2N,\dag}(t,x) 
			=  \frac{\lambda_{n}^{-(\alpha-\frac{1}{2})}}{(2n+1)^\frac{1}{2}}\sum_{|k|, |\ell|\leq n} \chi_{T}(t) H_{n ; k,\ell}^{2N,\dag}(t)g_{n,\ell}(\omega)\mathbf{b}_{n,k}(x)\,,
	\]
	where $(H_{n;k,\ell}^{2N,\dag})_{|k|,|\ell|\leq n}$ are the matrix elements in the basis $(\mathbf{b}_{n,k})_{|k|\leq n}$ of the random averaging operator $\mathcal{H}_{n}^{2N,\dag}$ (acting on $E_{n}$). 
	These coefficients are $\mathcal{B}_{\leq N}$-measurable and therefore independent of the Gaussian variables $(g_{n,\ell})_{|\ell|\leq n}$ when $n\approx 2N$. Thus, we deduce from the conditional Wiener chaos estimate and from the non-pairing condition that  almost surely in $\omega$, 
		\begin{align*}
			\mathrm{l.h.s}\ \eqref{eq:CCC}&=\|\int_{\R^k} \sum_{\substack{n_1,\cdots, n_k\\ \text{distinct}}} 
			\prod_{i=1}^k
			\pi_{n_{i}}\widetilde{\psi_{2N}^{\dag}}(\kappa_i,x) f_{n_1,\cdots,n_k}(\kappa_1,\cdots \kappa_k, x)
			d\kappa_1\cdots d\kappa_k \|_{L_\omega^rL_\omega^r|\mathcal{B}_{\leq N}} \\
			&\lesssim r^\frac{k}{2}\|\int_{\R^k}
			\sum_{\substack{n_1,\cdots, n_k\\ \text{distinct}}}
			\prod_{i=1}^k\pi_{n_{i}}
			\widetilde{\psi_{2N}^{\dag}}(\kappa_i,x) f_{n_1,\cdots,n_k}(\kappa_1,\cdots \kappa_k, x)
			d\kappa_1\cdots d\kappa_k
			\|_{L_\omega^rL_\omega^2|\mathcal{B}_{\leq N}}\,.
		\end{align*}
		Expanding the colored Gaussian variables and using independence along with the non-pairing condition, we conclude that almost surely in $\omega$,
		\begin{multline}
		\label{eq:88}
			\|\int_{\R^k}d\kappa_1\cdots d\kappa_k\sum_{\substack{n_1,\cdots, n_k\\ \text{distinct}}}\prod_{i=1}^k\pi_{n_{i}}
			\widetilde{\psi_{2N}^{\dag}}(\kappa_i,x) f_{n_1,\cdots,n_k}(\kappa_1,\cdots \kappa_k, x,\omega) \|_{L_\omega^2|\mathcal{B}_{\leq N}}^2
			\\
			= \int_{\R^{2k}}d\kappa_1\cdots d\kappa_k d\kappa_1'\cdots d\kappa_k'
			\sum_{\sigma\in\mathfrak{S}_k }\sum_{\substack{n_1,\cdots,n_k\\ \mathrm{distinct} }}
			f_{n_{1},\cdots,n_{k}}(\kappa_1,\cdots,\kappa_k,x,\omega)
			\ov{f_{n_{\sigma(1)},\cdots,n_{\sigma(k)}}}(\kappa_1',\cdots,\kappa_k',x,\omega) 
			\\
			\prod_{i=1}^k\lambda_{n_i}^{-2(\alpha-\frac{1}{2})}
			\sum_{|\ell_i|\leq n_i}
			\Big(\sum_{|k_i|\leq n_i}
			\widetilde{(\chi_{T}H_{n_i;k_i,\ell_i}^{2N,\dag})}(\kappa_i)\frac{\mathbf{b}_{n_i,k_i}(x)}{(2n_i+1)^\frac{1}{2}}
			\Big)
			\\
			\Big(\overline{\sum_{|k_i'|\leq n_i}\widetilde{(\chi_{T}H_{n_i;k_i',\ell_i}^{2N,\dag})}(\kappa_i')\frac{\mathbf{b}_{n_i,k_i'}(x)}{(2n+1)^\frac{1}{2}}}\Big)\,,
		\end{multline}
	where $\mathfrak{S}_k$ is the permutation group of $\{1,2,\cdots,k\}$. We apply Cauchy-Schwarz's inequality in the sums over $\ell_i$ to see that 
		\begin{multline*}
			\eqref{eq:88}\lesssim
			\int_{\R^{2k}}d\kappa_1\cdots d\kappa_k d\kappa_1'\cdots d\kappa_k' \\
		\sum_{\sigma\in\mathfrak{S}_k}	\sum_{n_1,\cdots,n_k}|f_{n_1,\cdots,n_k}(\kappa_1,\cdots,\kappa_k,x,\omega)\overline{f_{n_{\sigma(1)},\cdots,n_{\sigma(k)}}(\kappa_1',\cdots,\kappa_k',x,\omega)}| \\
		\times 	\prod_{i=1}^k\lambda_{n_i}^{-2(\alpha-\frac{1}{2})}
			\Big(\sum_{|\ell_i|\leq n_i}\Big|\sum_{|k_i|\leq n_i}\widetilde{(\chi_{T}H_{n_i ; k_i,\ell_i}^{2N,\dag})}(\kappa_i)\frac{\mathbf{b}_{n_i,k_i}(x)}{(2n_i+1)^\frac{1}{2}}\Big|^2\Big)^\frac{1}{2}\\
			\times\Big(\sum_{|\ell_i|\leq n_i}\Big|\sum_{|k_i|\leq n_i}\widetilde{(\chi_{T}H_{n_i ; k_i,\ell_i}^{2N,\dag})}(\kappa_i')\frac{\mathbf{b}_{n_i,k_i}(x)}{(2n_i+1)^\frac{1}{2}}\Big|^2\Big)^\frac{1}{2}\,.
		\end{multline*}
		We conclude from H{\"o}lder's inequality in the $\kappa_i$'s variables that
		\begin{multline*}
			\eqref{eq:88}\lesssim \sum_{n_1,\cdots,n_k}
			\Big\|
			\big(\prod_{i=1}^k\langle\kappa_i\rangle^{-\gamma}\lambda_{n_i}^{-(\alpha-\frac{1}{2})})  f_{n_1,\cdots,n_k}(\kappa_1,\cdots,\kappa_k,x,\omega
			\big)
			\Big\|_{L_{\kappa_1,\cdots,\kappa_k}^{q'}}^2\\
			\prod_{i=1}^k \Big\|\langle\kappa_i\rangle^{\gamma}\Big(\sum_{|\ell_i|\leq n_i}\Big|\sum_{|k_i|\leq n_i}\widetilde{(\chi_{T}H_{n_i ; k_i,\ell_i}^{2N,\dag})}(\kappa_i)\frac{\mathbf{b}_{n_i,k_i}(x)}{(2n_i+1)^\frac{1}{2}}\Big|^2\Big)^\frac{1}{2} \Big\|_{L_{\kappa_i}^q}^2\,.
		\end{multline*}
		Then, the result easily follows from~\eqref{pointwiseproba} and from the induction assumption $\operatorname{Loc}(N)$. Recalling the definition of the \eqref{Hn*norm}-norm, the above quantity is bounded by
		\begin{multline}
		\label{eq:dim}
		\eqref{eq:88}\lesssim
		\sum_{n_1,\cdots,n_k}
		\big
		\|
		(\prod_{i=1}^k\langle\kappa_i\rangle^{-\gamma}\lambda_{n_i}^{-(\alpha-\frac{1}{2})}) f_{n_1,\cdots,n_k}(\kappa_1,\cdots,\kappa_k,x,\omega)
		\big\|_{L_{\kappa_1,\cdots,\kappa_k}^{q'}}^2 
		\\
		\|(\chi_{T}\mathcal{H}_n^{2N,\dag})^\ast\|_{\ell_n^\infty S_n^{q,\gamma,\ast}}^{2k}\,.
		\end{multline}
		We conclude by using Minkowski ($2\geq q'$) and by taking the $L^r$-norm on the set $\Omega_{N}\cap\Xi'$, where we supposed that the operator bound~\eqref{eq:chi-H} at generation $2N$ is true:
		\begin{multline*}
		\mathrm{l.h.s}\ \eqref{eq:CCC}
		=
		r^{\frac{k}{2}}\|\eqref{eq:dim}^{\frac{1}{2}}\|_{L_{\omega}^{r}(\Omega_{N}\cap\Xi')}
				\lesssim_{k} 
		\\
			\big
			\|
			(\prod_{i=1}^k\langle\kappa_i\rangle^{-\gamma})
			 f_{n_1,\cdots,n_k}(\kappa_1,\cdots,\kappa_k,x,\omega
			)
			\big\|_{L_{\omega}^{r}L_{\kappa_1,\cdots,\kappa_k}^{q'}\ell_{n_{1},\cdots,n_{k}}^{2}}
			(r^{\frac{1}{2}}T^{-\gamma+\frac{1}{q'}}RN^{-(\alpha-\frac{1}{2})})^{k}
\,.
		\end{multline*}
		This completes the proof of~ Lemma~\ref{lem:multi-C}. 
	\end{proof}
We are now ready to prove Proposition \ref{prop:CCC}. 

\begin{proof}[Proof of Proposition \ref{prop:CCC}]

	The proof of Proposition~\ref{prop:CCC} boils down to show that there exists $\delta>0$ such that for all $r\geq 2$, 
	\begin{multline}
		\label{eq:CCC-moment}
		\big\|\langle\kappa\rangle^{-b_1}\int_{\R^3}d\vec{\kappa}
		\sum_{(n_1,n_2,n_3)\in\Gamma(n)}
		\widehat{\chi}(\widetilde{\kappa}-\Omega(\vec{n}))
		\pi_{n}\big(\pi_{n_{1}}\widetilde{\psi_{N_{1}}^{\dag}}(\kappa_{1})
		\pi_{n_{2}}\overline{\widetilde{\psi_{N_{2}}^{\dag}}}(\kappa_{2})
		\diamond\pi_{n_{3}}\widetilde{\psi_{N_{3}}^{\dag}}(\kappa_{3})\big) 
		\big\|_{L_\omega^r(\Omega_{N}\cap\Xi' ; \ell_n^2L_{\kappa,x}^2)}
		\\
		\lesssim (r^\frac{1}{2}RT^{-\gamma+\frac{1}{q'}})^{3}
		N^{-(\alpha-\frac{1}{2})+\frac{1}{q}+\delta}
		\,.
	\end{multline}
	We recall the following notations:
\[
\vec{n} = (n_1, n_2, n_3, n), \quad 
\widetilde{\kappa} := \kappa_1 - \kappa_2 + \kappa_3 - \kappa, \quad 
d\vec{\kappa} = d\kappa_1\, d\kappa_2\, d\kappa_3, \quad 
\Omega(\vec{n}) := \lambda_{n_1}^2 - \lambda_{n_2}^2 + \lambda_{n_3}^2 - \lambda_n^2.
\]
Given \(n\), the set \(\Gamma(n)\) represents a constraint to be specified. The Wick product, denoted by \(\diamond\), was defined in~\eqref{defwicksquare}.

\medskip

For clarity, we omit the constraint \(\omega \in \Omega_N \cap \Xi\) from the notation in what follows.

\begin{remarque} \label{rem:deg}
Before beginning the proof, we emphasize that—due to degree considerations for the spherical harmonics—only the terms with \(n \lesssim N_{(1)}\) contribute to the sum over \(n\) (corresponding to the output frequencies).
\end{remarque}

To prove~\eqref{eq:CCC-moment} and control the moments of order \(r\), we exploit different properties of the colored Gaussian random variables, depending on both the pairing configuration \(\Gamma(n)\) and the relative sizes of \(N_1, N_2, N_3\), and \(2N\).
	\medskip

\noindent{\bf Case 1: Pairing $n_{1}=n$} (corresponding to $\mathcal{N}_{(0,1)}$). In the analysis below, $\psi_{N_{2}}^{\dag}$ and $\psi_{N_{3}}^{\dag}$ play a symmetric role so we may suppose that $N_{2}\geq N_{3}$. Only the terms with $N_{(1)}=N_{(2)}=2N$ contribute since the other frequency interactions are absorbed in the ansatz \eqref{eq:psiN}. We have
	\[
	\mathrm{l.h.s}\ \eqref{eq:CCC-moment} =  \Big\|\langle\kappa\rangle^{-b_1}\int_{\R^3}
		\widehat{\chi}(\widetilde{\kappa}-(\lambda_{n_{2}}^{2}-\lambda_{n_{3}}^{2}))
		\mathcal{N}_{(0,1)}(\psi_{N_{1}}^{\dag},\psi_{N_{2}}^{\dag},\psi_{N_{3}}^{\dag})
		d\vec{\kappa}
		\Big\|_{L_{\omega}^{r}(\Omega_{N}\cap\Xi' ;\ell_n^2L_\kappa^2L_x^2)}\,,
	\]
	\medskip
	
	\noi{\bf $\bullet$ Case 1.1: Partial pairing: $\mathcal{N}_{(0,1)[2,3]}$}. In this case we have $N_{1}\leq 2N\,,$ $N_{2}=2N$. Since the resonant function  $\Omega(\vec{n})=\lambda_{n_{2}}^{2}-\lambda_{n_{3}}^{2}$ only depends on $n_{2}\neq n_{3}$, the strategy is to leverage these time-oscillations to sum over $n_{2}$ and $n_{3}$ together, loosing only $\mathcal{O}(N^{\epsilon})$. 

\medskip

	First, observe that we save a square root in the sum over $n$ thanks to the orthogonal projectors $\pi_{n}$. Then, we get rid of the terms $\pi_{n}\widetilde{\psi_{N_{1}}^{\dag}}(\kappa_1)$ by using the bound \eqref{eq:p-psi-}:
	\begin{multline*}
	\mathrm{l.h.s}\ \eqref{eq:CCC-moment}  
	= 
	\|
	\chi\mathcal{N}_{(0,1)[2,3]}(\psi_{N_{1}}^{\dag},\psi_{2N}^{\dag},\psi_{N_{3}}^{\dag})\|_{L_{\omega}^{r}(\Omega_{N}\cap\Xi',X^{0,-b_{1}})} 
	\\
		\lesssim 
		\|\langle\kappa\rangle^{-b_1}\sum_{n_2\neq n_3}
		\int_{\R^3} \widehat{\chi}
		(\widetilde{\kappa}-(\lambda_{n_{2}}^{2}-\lambda_{n_{3}}^{2})))
		\pi_{n}\widetilde{\psi_{N_{1}}^{\dag}}
		(\kappa_1,x)\pi_{n_{2}}
		\overline{\widetilde{\psi_{N_2}^{\dag}}}(\kappa_2,x)\pi_{n_{3}}\widetilde{\psi_{N_{3}}^{\dag}}(\kappa_3,x)d\vec{\kappa} 
		\|_{L_\omega^rL_\kappa^2L_x^2\ell_{n}^2}\\
		\lesssim
		\| \langle\kappa_1\rangle^{\gamma}\pi_{n}\widetilde{\psi_{N_{1}}^{\dag}} \|_{L_\omega^\infty(\Omega_{N}\cap\Xi_{2N}' ; \ell_{n}^{q}  L_{\kappa_1}^qL_x^\infty)}
		\|\langle\kappa\rangle^{-b_1}\langle\kappa_1\rangle^{-\gamma}f(\kappa,\kappa_1,x)  \|_{L_\omega^rL_{\kappa,x}^2\ell_{n}^{\frac{2q}{q-2}}L_{\kappa_1}^{q'}}
		 \\
		 \lesssim
		 \| \langle\kappa_1\rangle^{\gamma}\pi_{n}\widetilde{\psi_{N_{1}}^{\dag}} \|_{L_\omega^\infty(\Omega_{N}\cap\Xi_{2N}' ; \ell_{n}^{q}  L_{\kappa_1}^qL_x^\infty)}
		N_{1}^{\frac{q-2}{2q}}\|\langle\kappa\rangle^{-b_1}\langle\kappa_1\rangle^{-\gamma}f(\kappa,\kappa_1,x)  \|_{L_\omega^rL_{\kappa,x}^2L_{\kappa_1}^{q'}} \,,
	\end{multline*}
	with 
	\[
	f(\kappa_1,\kappa,x) := \int_{\R^2}
	\sum_{n_2\neq n_3}
	\widehat{\chi}(\widetilde{\kappa}-(\lambda_{n_{2}}^{2}-\lambda_{n_{3}}^{2}))
	\pi_{n_{2}}\overline{\widetilde{\psi_{N_2}^{\dag}}}(\kappa_2,x)
	\pi_{n_{3}}\widetilde{\psi_{N_{3}}^{\dag}}(\kappa_3,x)
	d\kappa_2d\kappa_3\,.
	\]
	In the last line we used that the sum over $n$ runs over $\mathcal{O}(N_{1})$-terms (because we have the condition $n=n_{1}$). We obtain 
	\[
	\mathrm{l.h.s}\ \eqref{eq:CCC-moment}
	\lesssim  
		N_{1}^{\frac{1}{2}-\frac{1}{q}}\|
		\psi_{N_{1}}^{\dag}
		\|_{L_{\omega}^{\infty}(\Omega_{N}\cap\Xi' ; X_{q,q,\infty}^{0,\gamma})}
		\|
		\langle\kappa\rangle^{-b_1}\langle\kappa_1\rangle^{-\gamma}f(\kappa,\kappa_1,x)
		\|_{L_\omega^rL_{\kappa,x}^2L_{\kappa_1}^{q'}}\,.
	\]
	According to the assumption that the bound \eqref{eq:p-psi-} is satisfied for all $\omega\in\Xi'$, we get
	\[
	\mathrm{l.h.s}\ \eqref{eq:CCC-moment}  
	\lesssim 
	RT^{-\gamma+\frac{1}{q'}}
	N_{1}^{-(\alpha-\frac{1}{2})+\frac{1}{2}+\frac{1}{q}+\delta}
	\|
	\langle\kappa\rangle^{-b_1}\langle\kappa_1\rangle^{-\gamma}f(\kappa,\kappa_1,x)
	\|_{L_\omega^rL_{\kappa,x}^2L_{\kappa_1}^{q'}}\,.
	\]
	We are now in a good position to apply Lemma \ref{lem:multi-C}. 
	\medskip
	
	\noindent $\bullet$ {\bf Case 1.1 a):} Suppose first that $N_{2}>N_{3}$, in which case we apply Lemma~\ref{lem:multi-C} twice. We represent
	\[
	f(\kappa,\kappa_1,x) 
	=: \int_{\R}
	\sum_{n_2\approx N_2}
	\pi_{n_{2}}\overline{\widetilde{\psi_{N_2}^{\dag}}}(\kappa_2)
	f_{n_2}(\kappa,\kappa_1,\kappa_2,x)
	d\kappa_2\,,
	\] 
	with 
	\[
	f_{n_2}(\kappa,\kappa_1,\kappa_2,x) 
	:=\int_\R
	\sum_{n_3\approx N_3}
	\widehat{\chi}(\widetilde{\kappa}-(\lambda_{n_{2}}^{2}-\lambda_{n_{3}}^{2}))
	\pi_{n_{3}}
	\widetilde{\psi_{N_{3}}^{\dag}}(\kappa_{3})
	d\kappa_3\,.
	\]
	To control $f$ we first use Lemma~\ref{lem:multi-C} applied at generation $N_2$ with $f_{n_2}$ and $k=1$: using Minkowski's (which is possible since $r\geq 2\geq q'$), we get
	\begin{multline*}
		\|
		\langle\kappa\rangle^{-b_1}
		\langle\kappa_1\rangle^{-\gamma}
		f(\kappa,\kappa_1,x)  
		\|_{L_\omega^rL_{\kappa,x}^2L_{\kappa_1}^{q'}}
		\leq 
		\|
		\langle\kappa\rangle^{-b_1}\langle\kappa_1\rangle^{-\gamma}f(\kappa,\kappa_1,x)  
		\|_{L_\omega^rL_{\kappa,x}^2L_{\kappa_1}^{q'}L_\omega^r|\mathcal{B}_{\leq N_3}} \\
		\lesssim 
		r^\frac{1}{2}RT^{-\gamma+\frac{1}{q'}}N_2^{-(\alpha-\frac{1}{2})}
		\|
		\langle\kappa\rangle^{-b_1}\langle\kappa_1\rangle^{-\gamma}\langle\kappa_2\rangle^{-\gamma}f_{n_2} 
		\|_{L_\omega^rL_{\kappa,x}^2L_{\kappa_1}^{q'}\ell_{n_2}^2L_{\kappa_2}^{q'}  }\,.
	\end{multline*}
	Subsequently, we write
	\[
	f_{n_2}(\kappa,\kappa_1,\kappa_2) = \int_\R\sum_{n_3\approx N_3}
	\pi_{n_{3}}\widetilde{\psi_{N_{3}}^{\dag}}(\kappa_3)f_{n_2,n_3}(\kappa,\kappa_1,\kappa_2,\kappa_3,x)
	d\kappa_3\,,
	\]
	with
	\[
	f_{n_2,n_3}(\kappa,\kappa_1,\kappa_2,\kappa_3) 
	= \widehat{\chi}(\widetilde{\kappa}-(\lambda_{n_{2}}^{2}-\lambda_{n_{3}}^{2}))\,.
	\]
	We apply Lemma~\ref{lem:multi-C} at generation $N_3$ with $f_{n_2,n_3}$ and $k=1$ and Minkowski ($2\geq q'$) to conclude that 
	\begin{multline*}
		\mathrm{l.h.s}\ \eqref{eq:CCC-moment} 
		\lesssim 
		r(RT^{-\gamma+\frac{1}{q'}})^{3}
		N_{1}^{-(\alpha-\frac{1}{2})+\frac{1}{2}+\frac{1}{q}+\delta}
		(N_2N_{3})^{-(\alpha-\frac{1}{2})}
		 \\
		\| \mathbf{1}_{n_{2}\neq n_{3}}
		\langle\kappa\rangle^{-b_1}
(\langle\kappa_1\rangle\langle\kappa_2\rangle\langle\kappa_3\rangle)^{-\gamma}\widehat{\chi}(\widetilde{\kappa}-
	(\lambda_{n_{2}}^{2}-\lambda_{n_{3}}^{2}))
		\|
		_{L_{\kappa}^2
		L_{\kappa_1,\kappa_2,\kappa_3}^{q'}
		\ell_{n_2,n_3}^2}\,.
	\end{multline*}
	Since $n_{2}\neq n_{3}$ in the sums, we can conclude from the divisor bound 
\begin{equation}
\label{eq:small-div}
		\sup_{l}\#
		\big\{
		n_2\neq n_3: n_j\sim N_j, \lambda_{n_2}^2-\lambda_{n_3}^2=l 
		\big\}
		\lesssim_{\epsilon} N_3^{\epsilon}
\end{equation}
		that we can sum over $n_2$ and $n_3$ together. We obtain that for all $\epsilon>0$, 
	\[
		\mathrm{l.h.s}\ \eqref{eq:CCC-moment} 
		\lesssim_{\epsilon} 
		r(RT^{-\gamma+\frac{1}{q'}})^{3}
		N_{1}^{-(\alpha-\frac{1}{2})+\frac{1}{2}+\frac{1}{q}+\delta}
		(N_2N_{3})^{-(\alpha-\frac{1}{2})}N_{3}^{\epsilon}
		\,.
	\]
	Recall that $\alpha\geq1$. In the case when $N_{(1)}=N_{(2)}=2N$, we get the bound
		\[
		\mathrm{l.h.s}\ \eqref{eq:CCC-moment} 
		\lesssim_{\epsilon}
		r(RT^{-\gamma+\frac{1}{q'}})^{3}
		N^{-2(\alpha-\frac{1}{2})+\frac{1}{2}+\frac{1}{q}+\delta}
		N_{(3)}^{-(\alpha-\frac{1}{2})+\epsilon}
		\,.
		\]
When $N_2=2N$, we get
		\[
		\mathrm{l.h.s}\ \eqref{eq:CCC-moment} 
		\lesssim_{\epsilon}
		r(RT^{-\gamma+\frac{1}{q'}})^3
		N^{-(\alpha-\frac{1}{2})+\frac{1}{q}+\delta}
		N_{1}^{-(\alpha-1)}N_{3}^{-(\alpha-\frac{1}{2})+\epsilon}
		\,.
		\]
This is conclusive in both cases (since \(-2(\alpha-\frac{1}{2})+\frac{1}{2} = -(\alpha-\frac{1}{2})-(\alpha-1) \leq -(\alpha-\frac{1}{2}) \)).
		\medskip
	
		\noindent
$\bullet$ \textbf{Case 1.1 b): \(N_{2} = N_{3}\).} In this case, we only need to consider the configuration \(N_{2} = N_{3} = 2N\), since—by assumption—interactions of the form \(\mathcal{N}_{(0,1)}\) require either \(N_{(2)} = N_{(1)} = 2N\) or \(N_{2} = 2N\). Applying Lemma~\ref{lem:multi-C} once, with \(k = 2\), we obtain:
	\[
	\mathrm{l.h.s}\ \eqref{eq:CCC-moment}
	\lesssim
	(r^{\frac{1}{2}}RT^{-\gamma+\frac{1}{q'}})^3 
	N_1^{-(\alpha-\frac{1}{2})+\frac{1}{2}+\frac{1}{q}+\delta}
	N_2^{-2(\alpha-\frac{1}{2})}
	\lesssim
	(r^{\frac{1}{2}}RT^{-\gamma+\frac{1}{q'}})^3
	N^{-3(\alpha-\frac{1}{2})+\frac{1}{2}+\frac{1}{q}+\delta}\,,
	\]
which is also conclusive.
\medskip

	\noindent
$\bullet$ \textbf{Case 1.2 — Full pairing: \(\mathcal{N}_{(0,1)(2,3)}\).} In this resonant case, we need to exploit—at least when \(n_1 \neq n_2\)—the Wick renormalization, particularly Lemma~\ref{lem:wick}. Moreover, we cannot expect any gain from the modulation, as there is no time oscillation. Hence, we do not consider the Fourier transform in time.

\medskip

For the interaction \(\mathcal{N}_{(0,1)}\), we assumed \(N_{(2)} = 2N\) or \(N_2 = 2N\). Since here \(N_2 = N_3\), and thus \(N_2 \geq N_{(2)}\), it follows necessarily that \(N_2 = N_3 = 2N\). Therefore, we will exploit both the Wick ordering and the gain in \(N_2\) provided by Lemma~\ref{lem:wick}. We then obtain:
	\begin{align*}
	\mathrm{l.h.s}\ \eqref{eq:CCC-moment} 
	&=
		\|\chi\sum_{n\approx N_{1}}\pi_{n}(\pi_{n}\psi_{N_{1}}^{\dag}(\psi_{2N}^{\dag}\hexagonal{=}\overline{\psi_{2N}^{\dag}})) \|_{L_\omega^r(\Omega_{N}\cap\Xi' ;X^{0,-b_1})}\\
		&\lesssim 
		\|\chi\sum_{n\approx N_{1}}\pi_{n}(\pi_{n}\psi_{N_{1}}^{\dag}(\psi_{2N}^{\dag}\hexagonal{=}\overline{\psi_{2N}^{\dag}}))
		\|_{L_\omega^r(\Omega_{N}\cap\Xi';X^{0,0})}
		\\
		&\lesssim 
		\|
		\chi \pi_{n}\psi_{N_{1}}^{\dag}(\psi_{2N}^{\dag}\hexagonal{=}\overline{\psi_{2N}^{\dag}})
		\|_{L_{\omega}^{r}(\Omega_{N}\cap\Xi';\ell_{n}^{2}L_{t}^{2}L_{x}^{2})}\,.
	\end{align*}
	Using the assumption that \eqref{eq:p-psi} is true for all $\omega\in\Xi'$, namely
	\[
		\|\psi_{2N}^{\dag}\hexagonal{=}\overline{\psi_{2N}^{\dag}}\|_{L_{t}^{q}L_{x}^{\infty}} \leq CR^{2}
		N^{-2(\alpha-\frac{1}{2})+\frac{1}{2}+\frac{1}{q}+\delta}\,,
	\]
	we obtain
	\begin{align*}
		\mathrm{l.h.s}\ \eqref{eq:CCC-moment} 
		&\lesssim 
		\big\|
		 \|
		 \chi \pi_{n}\psi_{N_{1}}^{\dag} 
		 \|_{L_{t}^{\frac{2q}{q-2}}\ell_{n}^{2}L_{x}^{2}}
		 \|\psi_{2N}^{\dag}\hexagonal{=}\overline{\psi_{2N}^{\dag}}\|_{L_{t}^{q}L_{x}^{\infty}}
		 \big
		 \|_{L_{\omega}^{r}(\Omega_{N}\cap\Xi_{2N}')} 
		 \\
		&\lesssim 
		\|\pi_{n}\psi_{N_{1}}^{\dag}
		\|
		_{L_{\omega}^{r}L_{t}^{\frac{2q}{q-2}}
		\ell_{n}^{2}L_{x}^{2}} R^{2}
		N^{-2(\alpha-\frac{1}{2})+\frac{1}{2}+\frac{1}{q}+\delta}\,. 
	\end{align*}
	Using that $\pi_{n}\psi_{N_{1}}^{\dag}(t,x)=\lambda_{n}^{-(\alpha-\frac{1}{2})}\mathcal{H}_{n}^{N_{1},\dag}(t)(e_{n}^{\omega})(x)$, and that $\mathcal{H}_{n}^{N_{1},\dag}(t)$ is unitary on~$E_{n}$, we obtain that for all $t$
	\begin{equation}
	\label{eq:psiN1}
	\|
	\pi_{n}\psi_{N_{1}}^{\dag}(t)\|_{L_{\omega}^{r}\ell_{n}^{2}L_{x}^{2}} 
	\lesssim 
	\big(
	\sum_{n\approx N_{1}}
	\lambda_{n}^{-2(\alpha-\frac{1}{2})}
	\|e_{n}^{\omega}
	\|_{L_{\omega}^{r}L_{x}^{2}}
	\big)^{\frac{1}{2}} 
	\lesssim r^{\frac{1}{2}}RN_{1}^{1-\alpha}\,,
	\end{equation}
	which is acceptable $\alpha\geq1$. Hence, by applying Minkowski's inequality to invert the integral in \(\omega\) and in\(t\) --- which is possible since $r\geq \frac{2q}{q-2}$, we obtain that
	\[
	\mathrm{l.h.s}\ \eqref{eq:CCC-moment}  \lesssim 
	R^{3}
	N_{1}^{1-\alpha}
	N^{-2(\alpha-\frac{1}{2})+\frac{1}{2}+\frac{1}{q}+2\delta} \lesssim R^{3}N^{-2(\alpha-\frac{1}{2})+\frac{1}{2}+\frac{1}{q}+\delta}\,.
	\]
	This is also conclusive.  
	\medskip

\noindent\textbf{Case 2 — No pairing: \(n_1, n_2, n_3, n\) are distinct.} In this case, we have \(N_{(1)} = 2N\), and without loss of generality, we assume \(N_1 = N_{(1)} \geq N_2 \geq N_3\), since the three terms are symmetric in the analysis of this configuration.

We focus on the more technical situation where the three frequencies belong to distinct generations, namely \(2N = N_1 > N_2 > N_3\). In this case, we apply Lemma~\ref{lem:multi-C} with \(k = 1\) three times. 

The other cases can be handled in a similar way: for instance, if \(N_1 = N_2 = 2N\), we first apply Lemma~\ref{lem:multi-C} at generation \(2N\) with \(k = 2\), and then at generation \(N_3\) with \(k = 1\).\medskip

In Case 2, when $2N=N_{1}>N_{2}>N_{3}$, we have 
		\[
		\mathrm{l.h.s}\ \eqref{eq:CCC-moment} = 
		\|
		\langle\kappa\rangle^{-b_1}
		\sum_{n_{1}\approx 2N}
		\int_\R 
		\pi_{n_{1}}\widetilde{\psi_{2N}^{\dag}}(\kappa_1,x)f_{n_1}(\kappa_1,n,\kappa,x) 
		d\kappa_1 
		\|_{L_\omega^r \ell_n^2L_\kappa^2L_x^2}\,,
		\]
		where, given $n_1\approx 2N$,
		\[
		f_{n_1}(\kappa_1,n,\kappa,x):=\sum_{\substack{n_2\approx N_2, n_3\approx N_3\\n_{2}\neq n_{3}}}
		\int_{\R^2} 
		\widehat{\chi}(\widetilde{\kappa}-\Omega(\vec{n}))
		\pi_{n_{2}}\overline{\widetilde{\psi_{N_{2}}^{\dag}}}(\kappa_2,x)
		\pi_{n_{3}}\widetilde{\psi_{N_{3}}^{\dag}}(\kappa_3,x)
		d\kappa_2d\kappa_3\,,
		\]
		is a $\mathcal{B}_{\leq N}$-measurable function. We first apply the conditional Minkowski's inequality ($r\geq2$) and Lemma~\ref{lem:multi-C} at generation $N_1$ with $f_{n_1}$ and $k=1$ to see that 
		\begin{align*}
			\mathrm{l.h.s}\,\eqref{eq:CCC} 
			&\lesssim 
			\|\langle\kappa\rangle^{-b_1}
			\int_\R
			\sum_{n_{1}\approx 2N}
			\pi_{n_{1}}\widetilde{\psi_{N_{1}}^{\dag}}(\kappa_1,x)f_{n_1}(\kappa_1,n,\kappa,x)
			d\kappa_1 
			\|_{L_\omega^r\ell_n^2L_\kappa^2L_x^2}\\
			&\lesssim r^\frac{1}{2}RT^{-\gamma+\frac{1}{q'}}N^{-(\alpha-\frac{1}{2})}
			\|
			\langle\kappa\rangle^{-b_1}\langle\kappa_1\rangle^{-\gamma}f_{n_1} 
			\|_{L_\omega^r\ell_n^2L_\kappa^2L_x^2\ell_{n_1}^2L_{\kappa_1}^{q'}}\\
			&\lesssim r^{\frac{1}{2}}RT^{-\gamma+\frac{1}{q'}}N^{-(\alpha-\frac{1}{2})}
				\|
			\langle\kappa\rangle^{-b_1}\langle\kappa_1\rangle^{-\gamma}f_{n_1} 
			\|_{L_\omega^r
			L_\kappa^2L_x^2
			L_{\kappa_1}^{q'}
			\ell_{n,n_1}^2
			}
			\,,
		\end{align*}
	where we used Minkowski's inequality in the last step.

\medskip

Subsequently, for fixed \(n_1\), we write
		\[
		f_{n_1}(\kappa_1,\kappa,n,x) 
		:=\int_{\R} \sum_{n_2\approx N_2} \pi_{n_{2}}
		\overline{\widetilde{\psi_{N_{2}}^{\dag}}}(\kappa_2,x)f_{n_1,n_2}(\kappa_1,\kappa_2,n,\kappa,x)
		d\kappa_2\,,
		\]
		with
		\[
		f_{n_1,n_2}(\kappa_1,\kappa_2,n,\kappa,x) := \sum_{n_3\approx N_3}
		\int_\R \pi_{n_{3}}\widetilde{\psi_{N_{3}}^{\dag}}(\kappa_3,x)
		\widehat{\chi}(\widetilde{\kappa}-\Omega(\vec{n}))
		d\kappa_3\,.
		\]
		We apply Lemma~\ref{lem:multi-C} at generation $N_2$ with $f_{n_1,n_2}$ (which is a $\mathcal{B}_{\leq N_3}$-measurable function, with $N_{3}\leq \frac{N_{2}}{2}$) and $k=1$ together with the Minkowski's inequality to deduce that 
		\begin{multline*}
			\|
			\langle\kappa\rangle^{-b_{1}}\langle\kappa_1\rangle^{-\gamma}f_{n_1} 
			\|
			_{L_\omega^r(\Omega_N\cap\Xi';\ell_{n,n_1}^2L_\kappa^2L_x^2L_{\kappa_1}^{q'})}
			\\
			\lesssim 
			r^\frac{1}{2}RT^{-\gamma+\frac{1}{q'}} 
			N_{2}^{-(\alpha-\frac{1}{2})}
			\|
			\langle\kappa\rangle^{-b_{1}}
			 \prod_{i=1}^2\langle\kappa_i\rangle^{-\gamma} f_{n_1,n_2} 
			\|
			 _{L_{\omega}^{r}(\Omega_{N}\cap\Xi';L_x^2L_{\kappa_1,\kappa_2}^{q'}L_{\kappa}^2\ell_{n,n_1,n_2}^2)}
			\,.
		\end{multline*}
		Hence, 
\[
	\mathrm{l.h.s}\, \eqref{eq:CCC-moment}			
	\lesssim 
	r(RT^{-\gamma+\frac{1}{q'}})^2
	(NN_{2})^{-(\alpha-\frac{1}{2})}
	\|\langle\kappa\rangle^{-b_1}\prod_{i=1}^2\langle\kappa_{n_i}\rangle^{-\gamma}
	f_{n_1,n_2} 
	\|_{L_\omega^r(\Omega_N\cap\Xi' ; L_x^2L_{\kappa_1,\kappa_2}^{q'}L_{\kappa}^2\ell_{n,n_1,n_2}^2)}\,.
		\]
		Then, we apply Lemma~\ref{lem:multi-C} at generation $N_3$ with $f_{n_1,n_2,n_3}$ and $k=1$, where  
		\[
		f_{n_1,n_2,n_3}(\kappa_1,\kappa_2,\kappa_3,n,\kappa) := \widehat{\chi}(\widetilde{\kappa}-\Omega(\vec{n}))
		= \widehat{\chi}(\widetilde{\kappa}-(\lambda_{n}^{2}-\lambda_{n_{1}}^{2}+\lambda_{n_{2}}^{2}-\lambda_{n_{3}}^{2}))\,.
		\]
		We obtain
		\[
		\mathrm{l.h.s}\, \eqref{eq:CCC-moment}	
		\lesssim 
		r^\frac{3}{2}(RT^{\gamma-\frac{1}{q'}})^3
		(NN_{2}N_{3})^{-(\alpha-\frac{1}{2})}
		\|\langle\kappa\rangle^{-b_1} 
		\prod_{i=1}^3\langle\kappa_i\rangle^{-\gamma}
		\widehat{\chi}(\widetilde{\kappa}-\Omega(\vec{n})) \|_{L_\kappa^2L_{\kappa_1,\kappa_2,\kappa_3}^{q'}\ell_{n,n_1,n_2,n_3}^2 }\,.
		\]
		Then, we can use the divisor bound as in \eqref{eq:small-div} together with Remark \ref{rem:deg} and sum over $n_{1}$ and $n$ together for fixed $\vec{\kappa}$ and $n_{2}, n_{3}$: for all $\epsilon>0$,
	\begin{multline*}
		\mathrm{l.h.s}\, \eqref{eq:CCC-moment}
		\lesssim 
		(r^\frac{1}{2}RT^{\gamma-\frac{1}{q'}})^3(N_1N_{2}N_{3})^{-(\alpha-\frac{1}{2})}(N_{2}N_{3})^{\frac{1}{2}} \\
		\Big\|
		\langle \kappa\rangle^{-b_{1}}\prod_{i=1}^{3}\langle \kappa_{i}\rangle^{-\gamma}\kappa^{-b_{1}}
		\|
		\widehat{\chi}(\widetilde{\kappa}-\Omega(\vec{n}))
		\|_{L_{\kappa, \kappa_{1}, \kappa_{2}, \kappa_{3}}^{\infty}\ell_{n_{2},n_{3}}^{\infty}\ell_{n,n_{1}}^{2}}
		\Big\|_{L_{\kappa}^{2}L_{\kappa_{1},\kappa_{2},\kappa_{3}}^{q}}
		\\
		\lesssim_{\epsilon} r^{\frac{3}{2}}(RT^{\gamma-\frac{1}{q'}})^{3}N^{-(\alpha-\frac{1}{2})+\epsilon}(N_{2}N_{3})^{1-\alpha}\,,
	\end{multline*}
which is conclusive at least when $\alpha\geq1$. 

	\medskip

\noindent{\bf Case 3: Partial pairing with wick-ordering: $\mathcal{N}_{[0,1](23)}$.} 
We distinguish between two cases, whether $N_{1}\leq N_{2}=N_{3}=2N$, or $N_{2}=N_{3}\leq N$ and $N_{1}=2N$. 

	\medskip

	\noi
	$\bullet$ {\bf Case 3.1:} $N_{1}\leq N_2=N_{3}=2N$. In this case, it suffices to exploit the gain provided by the Wick ordering. For this reason, we do not need to use the Fourier transform in time:
	\begin{align*}
	\mathrm{l.h.s}\ \eqref{eq:CCC-moment}  &= 
		\|
		\chi \sum_{n}\sum_{\substack{n_{1}\approx N_{1}\\n_{1}\neq n}} \pi_{n}
		(\pi_{n_{1}}\psi_{N_{1}}^{\dag}
		(\psi_{2N}^{\dag}\hexagonal{=}\ \psi_{2N}^{\dag}))
		\|_{L_{\omega}^{r}(\Omega_{N}\cap\Xi';X^{0,-b_{1}})} \\
		&\lesssim
		\|
		\chi \sum_{n}\sum_{\substack{n_{1}\approx N_{1}\\n_{1}\neq n}} 
		\pi_{n}
		(\pi_{n_{1}}\psi_{N_{1}}^{\dag}
		(\psi_{2N}^{\dag}\hexagonal{=}\ \psi_{2N}^{\dag}))
		\|_{L_{\omega}^{r}(\Omega_{N}\cap\Xi';X^{0,0})}\,.
	\end{align*}
	Observe that 
		\begin{multline*}
		\|\chi \sum_{n}\sum_{\substack{n_{1}\approx N_{1}\\n_{1}\neq n}} \pi_{n}
		(\pi_{n_{1}}\psi_{N_{1}}^{\dag}
		(\psi_{2N}^{\dag}\hexagonal{=}\ \psi_{2N}^{\dag}))
		\|_{L_{\omega}^{r}(\Omega_{N}\cap\Xi';L_{t,x}^{2})} \\
		\leq 
		\|
		\chi \sum_{n}\sum_{n_{1}\approx N_{1}} \pi_{n}
		(\pi_{n_{1}}\psi_{N_{1}}^{\dag}
		(\psi_{2N}^{\dag}\hexagonal{=}\ \psi_{2N}^{\dag}))
		\|_{L_{\omega}^{r}(\Omega_{N}\cap\Xi';L_{t,x}^{2})}\\
		+ 
		\|
		\chi \sum_{n}\pi_{n}
		(\pi_{n}\psi_{N_{1}}^{\dag}
		(\psi_{2N}^{\dag}\hexagonal{=}\ \psi_{2N}^{\dag}))
		\|_{L_{\omega}^{r}(\Omega_{N}\cap\Xi';L_{t,x}^{2})}\,.
		\end{multline*}
	We have already controlled the second term in Case 1.2 (in the present case when $N_{2}=2N $), and we deduce from this bound that
		\[
		\mathrm{l.h.s}\ \eqref{eq:CCC-moment} 
		\lesssim
		\|
		\chi 
		\psi_{N_{1}}^{\dag}(\psi_{2N}^{\dag}\hexagonal{=}\ \psi_{2N}^{\dag})
		\|_{L_{\omega}^{r}(\Omega_{N}\cap\Xi';L_{t,x}^{2})}	
		+ 
		R^{3}N^{-2(\alpha-\frac{1}{2})+\frac{1}{2}+\frac{1}{q}+\delta}\,.
		\]
	The first term on the right-hand side can be estimated as follows:
	\[
		\|
		\chi 
		\psi_{N_{1}}^{\dag}(\psi_{2N}^{\dag}\hexagonal{=}\ \psi_{2N}^{\dag})
		\|_{L_{\omega}^{r}(\Omega_{N}\cap\Xi';L_{t,x}^{2})}	
		\leq
		\|\psi_{N_{1}}^{\dag}\|_{L_{t}^{\frac{2q}{q-2}}L_{x}^{2}}
		\|\psi_{2N}^{\dag}\hexagonal{=}\ \psi_{2N}^{\dag}\|_{L_{t}^{q}L_{x}^{\infty}}\,.
	\]
	We have from Minkowski's inequality $(r\geq\frac{2q}{q-2})$ and from \eqref{eq:psiN1} that 
	\[
	\|\psi_{N_{1}}^{\dag}\|_{L_{\omega}^r(L_{t}^{\frac{2q}{q-2}}L_{x}^{2})} 
	\lesssim r^{\frac{1}{2}}RN_{1}^{1-\alpha}\,,
	\]
	and, since \(\omega\in\Xi'\) given by Lemma~\ref{lem:wick}, we have
	\[
	\|\psi_{2N}^{\dag}\hexagonal{=}\ \psi_{2N}^{\dag}\|_{L_{t}^{q}L_{x}^{\infty}} 
	\lesssim R^{2}
	N^{-2(\alpha-\frac{1}{2})+\frac{1}{2}+\frac{1}{q}+\delta}\,.
	\]
	This proves that in Case 3.1, the same estimate as in Case 1.2 holds:
	\[
	\mathrm{l.h.s}\ \eqref{eq:CCC-moment}\lesssim R^{2}
	N^{-2(\alpha-\frac{1}{2})+\frac{1}{2}+\frac{1}{q}+\delta}\,, 
	\]
	for all $\alpha\geq1$. 
	
	\medskip

	\noi
	$\bullet$ {\bf Case 3.2: $N_2=N_{3}\leq N$, $N_{1}=2N$}. In this case, it suffices to gain over $N_{1}$ by using the time-modulation, and the non-pairing condition $n\neq n_{1}$. 
	\medskip
	
	We do not exploit the saving from the Wick-ordering (for instance, when $N_{2}\ll N_{1}$ this saving is of no use). In Case 3.2, we have
	\begin{multline*}
	\mathrm{l.h.s}\ \eqref{eq:CCC-moment} 
		\lesssim 
		\big\|
			\langle\kappa\rangle^{-b_{1}}\sum_{n_{1}\approx 2N}
			\int_{\R^{3}}
			\widehat{\chi}(\widetilde{\kappa}-(\lambda_{n}^{2}-\lambda_{n_{1}}^{2}))
			\pi_{n}
			\big(\pi_{n_{1}}\widetilde{\psi_{N_{1}}^{\dag}}(\kappa_{1},x)
			\\
			\overline{\widetilde{\psi_{N_{2}}^{\dag}}}(\kappa_{2},x)\hexagonal{=}\ 
			\widetilde{\psi_{N_{2}}^{\dag}}(\kappa_{3},x)
			\big)
			d\vec{\kappa}
		\big\|_{L_{\omega}^{r}\ell_{n}^{2}L_{\kappa}^{2}L_{x}^{2}}
		\,.
	\end{multline*}
	Since  $N_{2}\leq N$, the random function $\psi_{N_{2}}^{\dag}$ is $\mathcal{B}_{\leq N}$-measurable, we can readily apply Lemma \ref{lem:multi-C}  at generation $2N$ with $k=1$ and
	\[
	f_{n_{1}}(\kappa,\kappa_{1},x,n) 
		:=
		\int_{\R^{2}}
		\widehat{\chi}(\widetilde{\kappa}-(\lambda_{n}^{2}-\lambda_{n_{1}}^{2}))
		\overline{\widetilde{\psi_{N_{2}}^{\dag}}}(\kappa_{2},x)
		\hexagonal{=}\ 
		\widetilde{\psi_{N_{2}}^{\dag}}(\kappa_{3},x)
	d\kappa_{2}d\kappa_{3}\,.
	\]
	We obtain
	\[
	\mathrm{l.h.s}\ \eqref{eq:CCC-moment}
		\lesssim r^{\frac{1}{2}}RT^{-\gamma+\frac{1}{q'}}N^{-(\alpha-\frac{1}{2})}
		\big\|
		\langle \kappa\rangle^{-b_{1}}\langle \kappa_{1}\rangle^{-\gamma}f_{n_{1}}(\kappa,\kappa_{1},x,n)
		\big\|_{L_{\omega}^{r}L_{\kappa}^{2}L_{x}^{2}\ell_{n,n_1}^2L_{\kappa_{1}}^{q'}}\,.		
	\]
	It follows from the Minkowski's inequality ($q'<2$) and the Hölder's inequality that 
	\begin{multline*}
	\mathrm{l.h.s}\ \eqref{eq:CCC-moment}
		\lesssim r^{\frac{1}{2}}RT^{-\gamma+\frac{1}{q'}}N^{-(\alpha-\frac{1}{2})}\\
		\|
		\langle\kappa\rangle^{-b_{1}}(\langle\kappa_{1}\rangle\langle\kappa_{2}\rangle\langle\kappa_{3}\rangle)^{-\gamma}\widehat{\chi}(\widetilde{\kappa}-(\lambda_{n}^{2}-\lambda_{n_{1}}^{2}))
		\|_{L_{\kappa}^{2}L_{\kappa_{1},\kappa_{2},\kappa_{3}}^{q'}\ell_{n,n_{1}}^{2}}
		N_{2}^{1-\frac{2}{q}}\|\psi_{N_{2}}^{\dag}\|_{X_{q,q,\infty}^{0,\gamma}}^{2}\,.
	\end{multline*}
	Using the bound \eqref{eq:p-psi-} for $\psi_{N_{2}}^{\dag}$, which holds on $\Omega_N$, we get that for all $\omega\in \Omega_N\cap \Xi'$,
	\begin{multline*}
	\mathrm{l.h.s}\ \eqref{eq:CCC-moment}\ \lesssim 
		r^{\frac{1}{2}}(RT^{-\gamma+\frac{1}{q'}})^{3} 
		N^{-(\alpha-\frac{1}{2})}
		N_{2}^{1-2(\alpha-\frac{1}{2})+2\delta}
		\\
		\|\langle \kappa\rangle^{-b_{1}}(\langle\kappa_{1}\rangle\langle\kappa_{2}\rangle\langle\kappa_{3}\rangle)^{-\gamma}\widehat{\chi}(\widetilde{\kappa}-(\lambda_{n}^{2}-\lambda_{n_{1}}^{2}))\|
		_{L_{\kappa}^{2}L_{\kappa_{1},\kappa_{2},\kappa_{3}}^{q'}\ell_{n,n_{1}}^{2}}.
	\end{multline*}
	We conclude from the divisor bound and Remark \ref{rem:deg}, which allows to sum $n, n_{1}\lesssim N$ together. For all $\epsilon>0$, 
	\[
	\mathrm{l.h.s}\ \eqref{eq:CCC-moment}
	\lesssim_{\epsilon}
	r^{\frac{1}{2}}(RT^{-\gamma+\frac{1}{q'}})^{3}N^{-(\alpha-\frac{1}{2})+\epsilon+2\delta}
	N_{2}^{1-2(\alpha-\frac{1}{2})}\,,
	\]
	which is conclusive when $\alpha\geq1$.
	This completes the proof of Proposition \ref{prop:CCC}. 
\end{proof}


\section{Deterministic trilinear estimates}\label{sec:det}
In this section, we prove the trilinear estimates involving at least one term of type~($\mathrm{D}$), which are collected in Proposition~\ref{multilinear:det-2}. This will complete the proof of the main result. In the following, the functions $v_{N_j}^{\dag}$ denote either $\psi_{N_j}^{\dag}$ or $z_{N_j}^{\dag}$. We divide the various cases into the following groups.
\medskip

\begin{itemize}
	\item \textbf{Group I:} There is at least one term of type (D), say $z_{M}^{\dag}$, with $M \geq \frac{1}{100}N$.
	\medskip
	 
	\item \textbf{Group II:} This group includes interactions that are not in Group I and where the type (C) term $\psi_{2N}^{\dag}$ appears in the first or third position. These include interactions of the form:
	\[
	\mathcal{N}_{(0,1)}(\psi_{2N}^{\dag}, \psi_{2N}^{\dag}, z^{\dag}_{N_3}), \quad 
	\mathcal{N}_{(0,1)}(\psi_{2N}^{\dag}, z^{\dag}_{N_2}, \psi_{2N}^{\dag}), \quad 
	\mathcal{N}_{[0,1]}(\psi_{2N}^{\dag}, v_{N_2}^{\dag}, v_{N_3}^{\dag}),
	\]
	where at least one of $v_{N_2}^{\dag}, v_{N_3}^{\dag}$ is of type (D). Note that we exclude the singular interactions $\mathcal{N}_{(0,1)}(\psi_{2N}^{\dag}, v_{N_2}^{\dag}, v_{N_3}^{\dag})$ with $N_2, N_3 \leq N$, thanks to the ansatz.
	\medskip
	
	\item \textbf{Group III:} This group includes interactions not belonging to Group I or II, and where the term $\psi_{2N}^{\dag}$ appears in the second position:
	\[
	\mathcal{N}(v_{N_1}^{\dag}, \psi_{2N}^{\dag}, v_{N_3}^{\dag})\,,
	\] 
	with $N_1, N_3 \leq N$, and where each $v_{N_i}^{\dag}$ is of type (C) or (D), with at least one of type (D).
\end{itemize}
\medskip

Before estimating the interactions in each group, we prove an elementary lemma useful for controlling terms of type (D). This lemma exploits the fact that type (D) terms are not arbitrary elements of $X^{s,b}$. In fact, they satisfy a semi-classical energy estimate and are essentially oscillating at frequency $N$.  

\begin{lemme}\label{lem:locN-c}
Suppose that $\mathrm{Loc}(N)$ holds. Then for all $M \leq N$ and for any $z_{M}^{\dag}$ of type $\mathrm{(D)}$ and $\epsilon_0 > 0$, we have
\[
\|z_{M}^{\dag}\|_{X^{s - \epsilon_0, b}} \lesssim R^{-1} M^{-\epsilon_0}\,.
\]
\end{lemme}

\begin{proof}
The bound~\eqref{eq:bo-r} on the high-frequencies yields
\begin{align*}
\|z_{M}^{\dag}\|_{X^{s - \epsilon_0, b}} 
&\leq \sum_{K} K^{s - \epsilon_0} \|\mathbf{P}_{K} z_{M}^{\dag}\|_{X^{0,b}} \\
&\lesssim \sum_{K \leq M} K^{s - \epsilon_0} \|z_{M}^{\dag}\|_{X^{0,b}} 
+ R^{-1} \sum_{K \geq 2M} \left( \frac{M}{K} \right)^{10 - s} K^{-\epsilon_0} \\
&\lesssim R^{-1} M^{-\epsilon_0}\,,
\end{align*}
which proves Lemma~\ref{lem:locN-c}.
\end{proof}

We now turn to estimating the terms in Group~I, which contains terms of type (D) constructed at a generation $M$ with $M \geq \frac{N}{100}$. The goal is to place this high-frequency term in $X^{0,b}$ to avoid losing derivatives. The argument relies solely on the semi-classical Strichartz estimate stated in Proposition~\ref{prop:str}.
\begin{lemme}\label{lem:Dmax} Suppose that  $0<\epsilon_{0}<\frac{1}{2}(s-\frac{1}{2})$. There exists $C>0$ such that for all function $v_{N_{2}}^{\dag}, v_{N_{3}}^{\dag}$ of type $\mathrm{(C)}$ or $\mathrm{(D)}$, and $b_{1}>\frac{1}{2}$, 
	\[
	\|\chi\mathcal{N}( z_{N_{1}}^{\dag},v_{N_{2}}^{\dag},v_{N_{3}}^{\dag})\|_{X^{0,-b_{1}}} \leq C\|z_{N_{1}}^{\dag}\|_{X^{0,b}}
	\prod_{i\in\{2,3\}}
	\min(\|v_{N_{i}}^{\dag}\|_{X^{s-2\epsilon_{0},b}},\|v_{N_{i}}^{\dag}\|_{L_{t}^{q}L_{x}^{\infty}})\,.
	\] 
Similarly, for all functions $v_{N_1}^{\dag},v_{N_3}^{\dag}$ of type $\mathrm{(C)}$ or $\mathrm{(D)}$,
	\[
\|\chi\mathcal{N}( v_{N_{1}}^{\dag},z_{N_{2}}^{\dag},v_{N_{3}}^{\dag})\|_{X^{0,-b_{1}}} \leq C\|z_{N_{2}}^{\dag}\|_{X^{0,b}}
\prod_{i\in\{1,3\}}
\min(\|v_{N_{i}}^{\dag}\|_{X^{s-2\epsilon_{0},b}},\|v_{N_{i}}^{\dag}\|_{L_{t}^{q}L_{x}^{\infty}})\,.
\] 
\end{lemme}
\begin{proof} By duality it suffices to show that for all $v\in X^{0,b_{1}}$ such that $\|v\|_{X^{0,b_{1}}}\leq1$, there holds
	\begin{align*}
	\Big|
	\int_{\R}\int_{\mathbb{S}^{2}}
	\chi(t)\mathcal{N}( z_{N_{1}}^{\dag},&v_{N_{2}}^{\dag},v_{N_{3}}^{\dag})(t,x)\overline{v(t,x)}\mathrm{d}t\mathrm{d}x
	\Big| \\
	\leq
	&C
	\|z_{N_{1}}^{\dag}\|_{X^{0,b}}\|v\|_{X^{0,b_1}}\prod_{i\in\{2,3\}}\min\big(\|v_{N_i}^{\dag}\|_{X^{s-2\epsilon_0,b}},\|v_{N_i}^{\dag}\|_{L_t^qL_x^{\infty}} \big).
  \end{align*}
Since the Wick renormalization plays no role in the analysis below, without loss of generality, we will replace the nonlinearity $\mathcal{N}(f_1,f_2,f_3)$ by the usual multiplication $f_1\ov{f}_2f_3$. 
	By Hölder's inequality, 
	\begin{multline*}
		\Big|\int_{\R}\int_{\mathbb{S}^{2}}
		\chi(t)\mathcal{N}( z_{N_{1}}^{\dag},v_{N_{2}}^{\dag},v_{N_{3}}^{\dag})(t,x)\overline{v(t,x)}\mathrm{d}t\mathrm{d}x
		\Big|
		\\
		\leq	
		\|z_{N_{1}}^{\dag}\|_{L_{t}^{\infty}L_{x}^{2}}
		\|v\|_{L_{t}^{\infty}L_{x}^{2}}
		\|v_{N_{2}}^{\dag}\|_{L_{t}^{p_{2}}L_{x}^{\infty}}
		\|v_{N_{3}}^{\dag}\|_{L_{t}^{p_{3}}L_{x}^{\infty}}
		\|\chi\|_{L_{t}^{p_{0}}}\,,
	\end{multline*}
	where $p_{2},p_{3}\in(2,+\infty)$, will be chosen later, and $p_{0}\in(1,+\infty)$ such that~$1~=~\frac{1}{p_{0}}~+~\frac{1}{p_{2}}~+~\frac{1}{p_{3}}$. 
	By Lemma \ref{Embedding} we obtain 
	\[
	\|z_{N_{1}}^{\dag}\|_{L_{t}^{\infty}L_{x}^{2}}
	\lesssim \|z_{N_{1}}^{\dag}\|_{X^{0,b}}\,,\quad 
	\|v\|_{L_{t}^{\infty}L_{x}^{2}}
	\lesssim \|v\|_{X^{0,b_{1}}}\,.
	\]
	As for the terms $v_{N_{2}}$ and $v_{N_{3}}$, of type $\mathrm{(C)}$ or $\mathrm{(D)}$, we proceed as follows. When $v_{N_{i}}$ is of type $\mathrm{(C)}$, we take $p_{i}=q_{\sigma}$, where the parameter $q_{\sigma}$ is given in~\eqref{smallparameters}, and we use the pointwise bound~\eqref{eq:wick}. When $v_{N_{i}}$ is of type $\mathrm{(D)}$, we let $(p_{i},q_{i})$ be an admissible pair and we apply the Sobolev embedding and the semi-classical Strichartz estimate from \cite{BGT} recalled in Proposition \ref{prop:str}:
	\[
	\|v_{N_{i}}^{\dag}\|_{L_{t}^{p_{i}}L_{x}^{\infty}}
	\lesssim
	\|v_{N_{i}}^{\dag}\|_{X^{\frac{2}{q_{i}}+\frac{1}{p_{i}},b}} 
	\lesssim
	\|v_{N_{i}}^{\dag}\|_{X^{1-\frac{1}{p_{i}},b}}\,,
	\]
	where we used the admissibility condition $\frac{2}{p_{i}}+\frac{2}{q_{i}}=1$. The desired bound for terms of type $\mathrm{(D)}$ follows by choosing $p_{i}>2$ sufficiently close to $2$, such that $1-\frac{1}{p_{i}}<s-2\epsilon_{0}$. This is possible thanks to the assumption $\epsilon_{0}<\frac{1}{2}(s-\frac{1}{2})$.
\end{proof}
To handle terms in Group II, we first isolate a special contribution where our analysis relies on the Wick ordering: 
\begin{lemme}[Terms in Group II requiring the Wick-ordering]
\label{lem:wick-b}For all $z^{\dag}_{N_{3}}$ of type~$\mathrm{(D)}$, we have 
	\[
	\| \chi\mathcal{N}_{(1,2)}(\psi_{2N}^{\dag},\psi_{2N}^{\dag},z^{\dag}_{N_3}) \|_{X^{0,-b_1}}\lesssim R^{2}N^{-2(\alpha-\frac{1}{2})+\frac{1}{2}+\frac{1}{q}+2\delta}
	N_{3}^{-s}
	\,.
	\]
\end{lemme}

\begin{proof}
The proof is carried out in the physical space, following the approach used in Lemma~\ref{lem:Dmax}. Specifically, we place \(\psi_{2N}^{\dag} = \overline{\psi_{2N}^{\dag}}\) in \(L_{t}^{q}L_{x}^{\infty}\) for some \(q \in (2, +\infty)\), and apply Lemma~\ref{lem:wick}. The term \(z_{N_3}^{\dag}\) is placed in \(X^{0,b}\), and we exploit the estimate~\eqref{eq:tD}---which follows from the assumption that \(z_{N_3}^{\dag}\) is of type~\((\mathrm{D})\)---to gain the decaying factor~\(R^{-1}N_3^{-s}\).
\end{proof}
It remains to handle the other interactions in Group II. In the next Lemma we control the remaining interactions not covered by Lemma \ref{lem:Dmax} and Lemma \ref{lem:wick-b} (even if there might be some overlap in some cases).

\begin{lemme}[Terms in Group II: remaining cases]\label{lem:no-pairing} For $i\in\{2,3\}$,  we suppose that $N_{i}\leq 2N$ and $v_{N_{i}}$ is of type $\mathrm{(C)}$ or $\mathrm{(D)}$ with at least one of type $\mathrm{(D)}$. Then, for all~$0<\epsilon<\min(\frac{1}{2}(s-\frac{1}{2}),\alpha-1-\frac{2}{q}-\delta,100\sigma)$~\footnote{Our assumptions~\eqref{smallparameters3} on the small parameter \(\sigma\) ensures that the quantity on the right-hand-side is positive.}, we have
	\begin{align}
		\label{eq:38}
		\|\chi\mathcal{N}_{[0,1][1,2]}(\psi_{2N}^{\dag},v_{N_{2}},v_{N_{3}})\|_{X^{0,-b_{1}}}  &\lesssim_{\epsilon} R
		N^{-s}(N_{2}N_{3})^{-\epsilon}\,.
		\intertext{Moreover, if $N_{3}\leq2N,$}
		\label{eq:39}
		\|\chi\mathcal{N}_{(0,1)}(\psi_{2N}^{\dag},\psi_{2N}^{\dag},z_{N_{3}}^{\dag})\|_{X^{0,-b_{1}}} &\lesssim_{\epsilon} RN^{-(\alpha-\frac{1}{2})-(\alpha-1)}N_{3}^{-(s-\frac{1}{2})+\epsilon}\,.
			\intertext{If $N_{2}\leq2N$,}
		\label{eq:39'}
		\|\chi\mathcal{N}(\psi_{2N}^{\dag},z_{N_{2}}^{\dag},\psi_{2N}^{\dag})\|_{X^{0,-b_{1}}} &\lesssim_{\epsilon} RN^{-(\alpha-\frac{1}{2})-(\alpha-1)}N_{2}^{-(s-\frac{1}{2})+\epsilon}\,.
	\end{align}
	\begin{proof}
		Thanks to the high frequency decay assumption~\eqref{eq:bo-r} in $\mathrm{Loc(N)}$ we can suppose that the terms of type (D) are truncated at frequency $4N$. Consequently the output frequency is less than $100N$. As in the proof of Lemma \ref{lem:Dmax}, it suffices to replace the nonlinearity $\mathcal{N}_{\cdots}$ by the usual multiplication, subjecting to the specific constraint.
		\medskip
		
		We first prove \eqref{eq:38}.
		\medskip
		
		\noindent{$\bullet$ \bf Case 1:} $\mathrm{(C)(C)(D)}$. By duality and Cauchy--Schwarz in $x\in\mathbb{S}^{2}$,
		\begin{multline}
			\label{eq:93}
			\|\chi\mathcal{N}_{[0,1][1,2]}(\psi_{2N}^{\dag},\psi_{N_{2}}^{\dag},\Pi_{2N_{3}}z_{N_{3}}^{\dag})\|_{X^{0,-b_{1}}} 
			\\
			\lesssim 
			\sum_{n_{0}\lesssim N\,,\,n_{2}\sim N_{2}}\sum_{\substack{n_{1}\sim N\\ n_{1}\neq n_{0}, n_{2}}}
			(\lambda_{n_{1}}\lambda_{n_{2}})^{-(\alpha-\frac{1}{2})}
			\int_{\R^{4}}
			\frac{|\widehat{\chi}(\widetilde{\kappa}-\Omega(\vec{n}))|}
			{
				\langle \kappa_{0}\rangle^{b_{1}}
				\langle \kappa_{1}\rangle^{\gamma}
				\langle\kappa_{2}\rangle^{\gamma}
				\langle\kappa_{3}\rangle^{b}
			}
			\prod_{i=0}^{3}\mathbf{a}_{n_{i}}^{(i)}(\kappa_{i})
			\mathrm{d}\vec{\kappa}\,,
		\end{multline}
where given a function $v\in X^{0,b_{1}}$ with $\|v\|_{X^{0,b_{1}}}\leq1$ we denoted
		\begin{align*}
			\mathbf{a}_{n_{0}}^{(0)}(\kappa_{0})
			&:= \langle\kappa_{0}\rangle^{b_{1}}
			\|\pi_{n_{0}}\widehat{v}_{n_{0}}(\kappa_{0}-\lambda_{n_{0}}^{2})\|_{L_{x}^{2}}
			\,,
			\quad
			\mathbf{a}_{n_{1}}^{(1)}(\kappa_{1})
			:=\langle\kappa_{1}\rangle^{\gamma}
			\|\widehat{e_{n_{1}}^{2N,\dag}}(\kappa_{1}-\lambda_{n_{1}}^{2})\|_{L_{x}^{\infty}}\,,
			\\
			\mathbf{a}_{n_{2}}^{(2)}(\kappa_{2})
			&:=\langle\kappa_{2}\rangle^{\gamma}\|
			\widehat{e_{n_{2}}^{N_{2},\dag}}(\kappa_{2}-\lambda_{n_{2}}^{2})\|_{L_{x}^{\infty}}\,,
			\quad
			\mathbf{a}_{n_{3}}^{(3)}(\kappa_{3})
			:=
			\langle \kappa_{3}\rangle^{b}
			\|\pi_{n_{3}}
			\widehat{z_{N_{3}}^{\dag}}(\kappa_{3}-\lambda_{n_{3}}^{2})\|_{L_{x}^{2}}\,. 
		\end{align*}
		Under our assumptions, together with Hölder's and Minkowski's inequalities to recover the estimate \eqref{eq:p-psi-} for type~(C) terms, we obtain
		\[
		\|\mathbf{a}_{n_{0}}^{(0)}(\kappa_{0})\|_{\ell_{n_{0}}^{2}L_{\kappa_{0}}^{2}} 
		= \|v\|_{X^{0,b_{1}}}\leq1\,,
		\quad 
		\|\mathbf{a}_{n_{i}}^{(i)}(\kappa_{i})\|
		_{L_{\kappa_{i}}^{q}\ell_{n_{i}}^{\infty}}
		\lesssim RN_{i}^{\frac{2}{q}+\delta}\quad \text{for}\ i\in\{1,2\}\,,
		\]
		and for all $\epsilon>0$, according to Lemma \ref{lem:locN-c},
		\begin{equation}
			\label{eq:n3}
			\|\mathbf{a}_{n_{3}}^{(3)}(\kappa_{3})\|_{L_{\kappa_{3}}^{2}\ell_{n_{3}}^{1}} 
			\lesssim_{\epsilon} \|\lambda_{n_{3}}^{\frac{1}{2}+\epsilon}\mathbf{a}_{n_{3}}^{(3)}(\kappa_{3})\|_{L_{\kappa_{3}}^{2}\ell_{n_{3}}^{2}} 
			\lesssim_{\epsilon} \|z_{N_{3}}^{\dag}\|_{X^{\frac{1}{2}+\epsilon,b}}
			\lesssim_{\epsilon} N_{3}^{-(s-\frac{1}{2})+\epsilon}R^{-1}\,.
		\end{equation}
		For fixed $n_2,n_3$ and $\vec{\kappa}$,  
		\[
		\sum_{n_{0}\lesssim N}
		\sum_{\substack{n_{1}\sim N\\ n_{1}\neq n_{0},n_{2}}}
		|\widehat{\chi}(\widetilde{\kappa}-\Omega(\vec{n}))|\mathbf{a}_{n_{1}}^{(1)}(\kappa_{1})\lambda_{n_{1}}^{-(\alpha-\frac{1}{2})}  
		\lesssim_{\epsilon} N^{-(\alpha-\frac{1}{2})+\epsilon}\|\mathbf{a}_{n_{1}}^{(1)}(\kappa_{1})\|_{\ell_{n_{1}}^{\infty}}\,.
		\]
		Indeed, for fixed $n_{2}$, $n_{3}$ and $\vec{\kappa}$ we have from the small divisor bound that
		\begin{align*}
		\sum_{n_{0}\lesssim N}
		\sum_{\substack{n_{1}\sim N\\ n_{1}\neq n_{0},n_{2}}}
		&|\widehat{\chi}(\widetilde{\kappa}-\Omega(\vec{n}))| \mathbf{a}_{n_{1}}^{(1)}(\kappa_{1})\lambda_{n_{1}}^{-(\alpha-\frac{1}{2})}  
			\\
			&\lesssim  
			N^{-(\alpha-\frac{1}{2})}\|\mathbf{a}_{n_{1}}^{(1)}(\kappa_{1})\|_{\ell_{n_{1}}^{\infty}}\sum_{m}
		\sum_{n_{0}\lesssim N}
		\sum_{\substack{n_{1}\sim N\\ n_{1}\neq n_{0},n_{2}}}
			\mathbf{1}_{\lambda_{n_{1}}^{2}-\lambda_{n_{0}}^{2}=m}
			|
			\widehat{\chi}(\widetilde{\kappa}-\lambda_{n_{3}}^{2}+\lambda_{n_{2}}^{2}-m)
			| 
			\\
			&\lesssim_{\chi} N^{-(\alpha-\frac{1}{2})}
			\|\mathbf{a}_{n_{1}}^{(1)}(\kappa_{1})\|_{\ell_{n_{1}}^{\infty}}
			\sup_{m}\#\{n_{0}\neq n_{1}\lesssim N\,,\ \lambda_{n_{1}}^{2}-\lambda_{n_{0}}^{2}=m \}
			\\
			&\lesssim_{\epsilon,\chi}N^{-(\alpha-\frac{1}{2})+\epsilon}\|\mathbf{a}_{n_{1}}^{(1)}(\kappa_{1})\|_{\ell_{n_{1}}^{\infty}}\,.
		\end{align*}
	Similarly, for fixed \(n_{0}, n_{3}\) and \(\vec{\kappa}\), we prove that
	\[
	\sum_{n_{2} \sim N_{2}}
	\sum_{\substack{n_{1} \sim N \\ n_{1} \neq n_{0}, n_{2}}}
	|\widehat{\chi}(\widetilde{\kappa} - \Omega(\vec{n}))| \,
	\mathbf{a}_{n_{1}}^{(1)}(\kappa_{1}) \,
	\lambda_{n_{1}}^{-(\alpha - \frac{1}{2})}
	\lesssim_{\epsilon, \chi} N^{-(\alpha - \frac{1}{2}) + \epsilon}
	\left\|
	\mathbf{a}_{n_{1}}^{(1)}(\kappa_{1})
	\right\|_{\ell_{n_{1}}^{\infty}}.
	\]
	We deduce from the Schur's test in the sum over $(n_{0},n_{2})$ that for fixed $n_{3}$ and $\vec{\kappa}$,
		\begin{align*}
			\sum_{n_{0}\lesssim N\,,\,n_{2}\sim N_{2}}
			\sum_{\substack{n_{1}\sim N\\ n_{1}\neq n_{0},n_{2}}}
			(\lambda_{n_{1}}\lambda_{n_{2}})^{-(\alpha-\frac{1}{2})}
			&|\widehat{\chi}(\widetilde{\kappa}-\Omega(\vec{n}))|
			\mathbf{a}_{n_{0}}^{(0)}(\kappa_{0})\mathbf{a}_{n_{1}}^{(1)}(\kappa_{1})
			\mathbf{a}_{n_{2}}^{(2)}(\kappa_{2})\\	
			&\lesssim_{\epsilon}
			N^{-(\alpha-\frac{1}{2})+\epsilon} 
			\|\mathbf{a}_{n_{1}}^{(1)}(\kappa_{1})\|_{\ell_{n_{1}}^{\infty}}
			\|\mathbf{a}_{n_{0}}^{(0)}(\kappa_{0})\|_{\ell_{n_{0}}^{2}} \|\lambda_{n_{2}}^{-(\alpha-\frac{1}{2})}\mathbf{a}_{n_{2}}^{(2)}(\kappa_{2})\|_{\ell_{n_{2}}^{2}}
			\\	
			&\lesssim_{\epsilon}N^{-(\alpha-\frac{1}{2})+\epsilon}N_{2}^{-(\alpha-1)}\|\mathbf{a}_{n_{1}}^{(1)}(\kappa_{1})\|_{\ell_{n_{1}}^{\infty}}
			\|\mathbf{a}_{n_{0}}^{(0)}(\kappa_{0})\|_{\ell_{n_{0}}^{2}}
			\|\mathbf{a}_{n_{2}}^{(2)}(\kappa_{2})\|_{\ell_{n_{2}}^{\infty}}\,.
		\end{align*}
		Summing over $n_{3}$ in \eqref{eq:93}, using the bound~\eqref{eq:n3}, and integrating over $\vec{\kappa}$, we conclude from Hölder's inequality that 
		\begin{align*}
			|\eqref{eq:93}| 
			&\lesssim_{\epsilon}N^{-(\alpha-\frac{1}{2})-\epsilon}N_{2}^{-(\alpha-1)}\|\mathbf{a}_{n_{0}}^{(0)}(\kappa_{0})\|_{L_{\kappa_{0}}^{2}\ell_{n_{0}}^{2}}
			\|\mathbf{a}_{n_{1}}^{(1)}(\kappa_{1})\|_{L_{\kappa_{1}}^{q}\ell_{n_{1}}^{\infty}}
			\|\mathbf{a}_{n_{2}}^{(2)}(\kappa_{2})\|_{L_{\kappa_{2}}^{q}\ell_{n_{2}}^{\infty}}
			\|\mathbf{a}_{n_{3}}^{(3)}(\kappa_{3})\|_{L_{\kappa_{3}}^{2}\ell_{n_{3}}^{1}} 
			\\
			&\lesssim_{\epsilon}
			N^{-(\alpha-\frac{1}{2})+\frac{2}{q}+\delta+\epsilon}N_{2}^{-(\alpha-1)+\frac{2}{q}+\delta}N_{3}^{-(s-\frac{1}{2})+\epsilon}R
			\\
			&\lesssim_{\epsilon}
			N^{-s}(N_{2}N_{3})^{-\epsilon}R\,,
		\end{align*}
		which gives the estimate~\eqref{eq:38}. 
		\medskip
		
		\noindent $\bullet$ {\bf Case 2:} (C)(D)(C). This case is analog to the previous one but inverting the role of $n_{2}$ and $n_{3}$, by first fixing $n_{2}$ while summing over $n_{0},n_{1},n_{3}$. When we apply the Schur's test, we apply the divisor bound
		\[
		\sup_{m}\ \#\{n_{1},n_{3}\lesssim N\ |\ \lambda_{n_{1}}^{2}+\lambda_{n_{3}}^{3}=m\} \lesssim_{\epsilon}N^{\epsilon}\,.	
		\]

		\noindent $\bullet$ {\bf Case 3:} (C)(D)(D). We do the same analysis as in Case 1, but replacing $\|\lambda_{n_{2}}^{-(\alpha-\frac{1}{2})}\mathbf{a}_{n_{2}}^{(0)}(\kappa_{2})\|_{\ell_{n_{2}}^{2}}$ by 
		\[
		\|\langle\kappa_{2}\rangle^{b}\pi_{n_{2}}\widehat{z_{N_{2}}^{\dag}}(\kappa_{2}-\lambda_{n_{2}}^{2})\|_{L_{\kappa}^{2}\ell^{2}L_{x}^{\infty}} \lesssim_{\epsilon} \|z_{N_{2}}^{\dag}\|_{X^{\frac{1}{2}+\epsilon,b}} \lesssim N_{2}^{-(s-\frac{1}{2})+\epsilon}R^{-1}\,,
		\] 
		where we used the Sobolev embedding and applied Lemma \ref{lem:locN-c} to obtain the last bound.
		\medskip
		
		We now turn to the proof of \eqref{eq:39}. In this case, we have the analog of \eqref{eq:93}:
		\begin{multline}
			\label{eq:94}
			\|\chi\mathcal{N}_{(0,1)}(\psi_{2N}^{\dag},\psi_{2N}^{\dag},\Pi_{2N_{3}}z_{N_{3}}^{\dag})\|_{X^{0,-b_{1}}}\\
			\lesssim
			\sum_{\substack{n,n_{2},n_{3}\\ n,n_{2}\approx N}}
			(\lambda_{n}\lambda_{n_{2}})^{-(\alpha-\frac{1}{2})}
			\int_{\R^{4}}
			\frac{|\widehat{\chi}(\widetilde{\kappa}-\Omega(\vec{n}))|}
			{\langle\kappa_{0}\rangle^{b_{1}}\langle\kappa_{1}\rangle^{\gamma}\langle\kappa_{2}\rangle^{\gamma}\langle\kappa_{3}\rangle^{b}}
			\mathbf{a}_{n}^{(0)}(\kappa_{0})\mathbf{a}_{n}^{(1)}(\kappa_{1})\mathbf{a}_{n_{2}}^{(2)}(\kappa_{2})\mathbf{a}_{n_{3}}^{(3)}(\kappa_{3})
			d\vec{\kappa}\,,
		\end{multline}
		where the $\mathbf{a}^{(i)}$ for $i\in\{0,1,2,3\}$ are as in \eqref{eq:93}. First, we fix $n,n_{3}$ and $\vec{\kappa}$ to sum over~$n_{2}$:
		\[
		\sum_{n_{2}\approx N}\lambda_{n_{2}}^{-(\alpha-\frac{1}{2})}\mathbf{a}_{n_{2}}^{(2)}(\kappa_{2})|\widehat{\chi}(\widetilde{\kappa}-(\lambda_{n_{1}}^{2}-\lambda_{n_{2}}^{2}+\lambda_{n_{3}}^{2}-\lambda_{n_{0}}^{2}))|
		\lesssim
		N^{-(\alpha-\frac{1}{2})}\|\mathbf{a}_{n_{2}}^{(2)}(\kappa_{2})\|_{\ell_{n_{2}}^{\infty}}\,.
		\]
		By Cauchy-Schwarz, summing over $n$ gives 
		\[
		\sum_{n\approx N}\lambda_{n}^{-(\alpha-\frac{1}{2})}\mathbf{a}_{n}^{(0)}(\kappa_{0})\mathbf{a}_{n}^{(1)}(\kappa_{1})
		\lesssim N^{-(\alpha-1)}\|\mathbf{a}_{n}^{(0)}\|_{\ell_{n}^{2}}\|\mathbf{a}_{n}^{(1)}\|_{\ell_{n}^{\infty}}\,.
		\]
		We conclude by integrating over $\vec{\kappa}$ using Hölder, and summing over $n_{3}$ using \eqref{eq:n3} to finally obtain \eqref{eq:39}. Since the proof of \eqref{eq:39'} is similar, we omit it. This finishes the proof of Lemma~\ref{lem:no-pairing}
	\end{proof}
	
\end{lemme}

Finally, we handle the interactions in Group III where $\psi_{2N}^{\dag}$ is at the second position and the other terms are of generation $\leq N$. 

\begin{lemme}[Remaining interactions with $\psi_{2N}^{\dag}$ in second position]\label{lem:psi2} 
	Suppose that for $i\in\{1,3\}$,  $N_{i}\leq N$ and $v_{N_{i}}$ is of type $\mathrm{(C)}$ or $\mathrm{(D)}$ with at least one of type $\mathrm{(D)}$. Moreover, for type $\mathrm{(D)}$ terms we suppose that $N_{i}\leq\frac{1}{50}N$. For all $b_{1}>\frac{1}{2}$ and $0<\epsilon<\frac{1}{2}\min(s-\frac{1}{2},\alpha-1)$, we have
	\begin{equation}
		\label{eq:40}
		\|\chi\mathcal{N}(v_{N_{1}},\psi_{2N}^{\dag},v_{N_{3}})\|_{X^{0,-b_{1}}} \lesssim_{\epsilon} N^{-(\alpha-\frac{1}{2})+\epsilon}(N_{1}N_{3})^{-\epsilon}\,.
	\end{equation}
\end{lemme}

\begin{proof} The proof is similar to the proof of Lemma \ref{lem:no-pairing}. We reduce to the case when the terms of type are truncated at frequency $4N$ and the output frequency is less than $100N$. As in the proof of Lemma \ref{lem:Dmax}, it suffices to replace the nonlinearity $\mathcal{N}_{\cdots}$ by the usual multiplication, subjecting to the specific constraint.
	\medskip
	
	\noindent{$\bullet$ \bf Case 1:} (C)(C)(D). By duality and Cauchy--Schwarz in $x\in\S^{2}$, 
	\begin{multline}
		\|\chi\mathcal{N}(\psi_{N_{1}}^{\dag},\psi_{2N}^{\dag},\Pi_{2N_{3}}z_{N_{3}}^{\dag})\|_{X^{0,-b_{1}}}
		\lesssim
		\sum_{n_{0}\lesssim N, n_{1}\approx N_{1}, n_{2}\approx 2N, n_{3}\lesssim N_{3}}
		\int_{\R^{4}}d\vec{\kappa}
		\\
		(\lambda_{n_{1}}\lambda_{n_{2}})^{-(\alpha-\frac{1}{2})}
		\frac{|\widehat{\chi}(\widetilde{\kappa}-\Omega(\vec{n}))|}
		{\langle\kappa_{0}\rangle^{b_{1}}\langle \kappa_{1}\rangle^{\gamma}\langle\kappa_{2}\rangle^{\gamma}\langle\kappa_{3}\rangle^{b}}
		\prod_{i=0}^{3}\mathbf{a}_{n_{i}}^{(i)}(\kappa_{i})\,,
	\end{multline}
	where, given a function $v\in X^{0,b_{1}}$ with $\|v\|_{X^{0,b_{1}}}\leq1$, we set
	\begin{align*}
		\mathbf{a}_{n_{0}}^{(0)}(\kappa_{0})
		&:= \langle\kappa_{0}\rangle^{b_{1}}
		\|\pi_{n_{0}}\widehat{v}_{n_{0}}(\kappa_{0}-\lambda_{n_{0}}^{2})\|_{L_{x}^{2}}
		\,,
		\quad
		\mathbf{a}_{n_{1}}^{(1)}(\kappa_{1})
		:=\langle\kappa_{1}\rangle^{\gamma}
		\|\widehat{e_{n_{1}}^{N_{1},\dag}}(\kappa_{1}-\lambda_{n_{1}}^{2})\|_{L_{x}^{\infty}}\,,
		\\
		\mathbf{a}_{n_{2}}^{(2)}(\kappa_{2})
		&:=\langle\kappa_{2}\rangle^{\gamma}\|
		\widehat{e_{n_{2}}^{2N,\dag}}(\kappa_{2}-\lambda_{n_{2}}^{2})\|_{L_{x}^{\infty}}\,,
		\quad
		\mathbf{a}_{n_{3}}^{(3)}(\kappa_{3})
		:=
		\langle \kappa_{3}\rangle^{b}
		\|\pi_{n_{3}}
		\widehat{z_{N_{3}}^{\dag}}(\kappa_{3}-\lambda_{n_{3}}^{2})\|_{L_{x}^{2}}\,.
	\end{align*}
	 Under our assumptions, we have that for all $\epsilon>0$,
	\begin{multline*}
		\|\mathbf{a}_{n_{0}}^{(0)}(\kappa_{0})\|_{\ell_{n_{0}}^{2}L_{\kappa_{0}}^{2}} 
		= \|v\|_{X^{0,b_{1}}}\leq1\,,
		\quad 
		\|\mathbf{a}_{n_{i}}^{(i)}(\kappa_{i})\|_{L_{\kappa_{i}}^{q}\ell_{n_{1}}^{\infty}}
		\lesssim R\quad \text{for}\ i\in\{1,2\}\,,
		\\ 
		\|\mathbf{a}_{n_{3}}^{(3)}(\kappa_{3})\|_{L_{\kappa_{3}}^{2}\ell_{n_{3}}^{1}} 
		\lesssim_{\epsilon} N_{3}^{-(s-\frac{1}{2})+\epsilon}R^{-1}\,,
	\end{multline*}
	where the last estimate was proved in \eqref{eq:n3}. We now fix $n_{3}$ and $\vec{\kappa}$ and we sum over $n_{0},n_{1},n_{2}$ in the same spirit as in the proof of Lemma \ref{lem:no-pairing}, but  in this case we will sum over $n_{0}$,$n_{1}$ using Schur's test and avoid to sum over $n_{2}$ by using the modulation.  The assumption $N_{1}\leq 2N$ ensures the non-pairing condition $n_{1}\neq n_{2}$. Moreover, we use the divisor bound 
	\[
	\sup_{m}\#\{n_{0},n_{2}\lesssim N\ |\ \lambda_{n_{0}}^{2}+\lambda_{n_{2}}^{2}=m\} \lesssim_{\epsilon}N^{\epsilon}\,.
	\]

	\noindent$\bullet$ {\bf Case 2:} (D)(C)(D). In this case we assume that $N_{1}\leq \frac{1}{50}N$ to ensure the non-pairing condition $n_{1}\neq n_{2}$ (otherwise this case falls into the scope of Group I and can be bounded by Lemma \ref{lem:Dmax}). The rest of the analysis is as in the previous cases and we do not give further details. 
\end{proof}

We are now ready to prove the trilinear deterministic estimates claimed in Proposition \ref{multilinear:det-2}. 

\begin{proof}[Proof of Proposition \ref{multilinear:det-2}] 
		Cases when there is a term  $z_{2N}^{\dag}$ follow from Lemma \ref{lem:Dmax}, Lemma \ref{lem:locN-c} and an interpolation argument (using \cite[Proposition 1.2.1]{DeSz} for instance, in order to have a crude estimate for the $X^{0,0}$-norm but loosing arbitrarily small powers of $N_{2}$ only). 
		\medskip
		
		The other cases are source terms, so we can afford to lose a small power of $N$ in the interpolation. These terms are covered by Lemma \ref{lem:wick-b}, \ref{lem:no-pairing} and \ref{lem:psi2} (always in the case $\alpha>1$ under the assumption~\eqref{smallparameters3} on $\sigma$), where we collected before the statement of Proposition \ref{multilinear:det-1} the different configurations that contribute to the equation \eqref{eq:wNdag} satisfied by $w_{2N}^{\dag}$.
\end{proof}

	\appendix

	\section{Banach-valued Fourier-Lebesgue function spaces} 
	Let $\mathcal{X}$ be a Banach space and $\mathbf{F}:\R\rightarrow \mathcal{X}$ a Schwartz function such that
	\begin{align}\label{def:FLabstract} \|\mathbf{F}\|_{\mathcal{F}L_q^{\gamma}(\mathcal{X})}:=\|\lg\tau\rg^{\gamma}\widehat{\mathbf{F}}(\tau)\|_{L^q(\R_{\tau};\mathcal{X})}<\infty.
	\end{align}
	The space $\mathcal{F}L_q^{\gamma}(\mathcal{X})$ is the closure of $\mathcal{X}$-valued Schwartz functions $\mathcal{S}(\R;\mathcal{X})$ with respect to the topology defined via the above norm. We will state and prove estimates for a general Banach space $\mathcal{X}$ in this appendix, but one should keep in mind that in our application, $\mathcal{X}=L^p(E_n)$ for some $1\leq p\leq \infty$ which is a finite (but large as $n\rightarrow\infty$) dimensional space.

	\subsection{Time localization properties}
	\begin{proposition}\label{timelocalization}
		Let $\mathcal{X}$ be a Banach space, $1<q<\infty, \gamma\in\R$, and $\varphi\in \mathcal{S}(\R)$.
		\begin{enumerate}
		\item   For any $\mathbf{F}\in\mathcal{F}L_q^{\gamma}(\mathcal{X})$,
		\[ \|\varphi(t)\mathbf{F}\|_{\mathcal{F}L_q^{\gamma}(\mathcal{X})}\lesssim \|\mathbf{F}\|_{\mathcal{F}L_q^{\gamma}(\mathcal{X})}.
		\]
		\item For $0<T\leq 1$, denote $\varphi_T(t)=\varphi(T^{-1}t)$ . For $0\leq \gamma\leq \gamma_1<1$, we have
		\begin{align*} 
		\|\varphi_T(t)\mathbf{F}\|_{\mathcal{F}L_q^{\gamma}(\mathcal{X})}
		&\lesssim T^{\gamma_1-\gamma}\|\mathbf{F}\|_{\mathcal{F}L_{q}^{\gamma_1}(\mathcal{X})}
		\quad \text{ if }\gamma_1<\frac{1}{q'}\,,
		\intertext{and}
		 \|\varphi_T(t)\mathbf{F}\|_{\mathcal{F}L_q^{\gamma}(\mathcal{X})}
		 &\lesssim T^{\gamma_1-\gamma}\log\big(1+\frac{1}{T}\big)\|\mathbf{F}\|_{\mathcal{F}L_{q}^{\gamma_1}(\mathcal{X})}
		 \quad \text{ if }\gamma_1=\frac{1}{q'},
		\intertext{and}
		\|\varphi_T(t)\mathbf{F}\|_{\mathcal{F}L_q^{\gamma}(\mathcal{X})}
		&\lesssim T^{\frac{1}{q'}-\gamma}\|\mathbf{F}\|_{\mathcal{F}L_{q}^{\gamma_1}(\mathcal{X})}
		\quad \text{ if }\gamma_1>\frac{1}{q'}\,.
		\end{align*}
		\item If in addition $\mathbf{F}(0)=0$, then for any $0<\gamma\leq \gamma_1<1$ (with possibly $\gamma_1>\frac{1}{q'}$), 
		\[ \|\varphi_T(t)\mathbf{F}\|_{\mathcal{F}L_q^{\gamma}(\mathcal{X})}\lesssim T^{\gamma_1-\gamma}\|\mathbf{F}\|_{\mathcal{F}L_{q}^{\gamma_1}(\mathcal{X})}.
		\]
		\end{enumerate}
	\end{proposition}
\begin{remarque}\label{remtimelocalization}
	It will be clear from the proof that all implicit bounds in Proposition~\ref{timelocalization} can be chosen to depend on
	\[
	\|\varphi\|_{W_t^{10,1}(\R)}:= \sum_{k=0}^{10}\|\varphi^{(k)}\|_{L^1(\R)}\,.
	\]
\end{remarque}	
	\begin{proof}
		(1)	By the Fourier inversion formula,
		\[ \varphi(t)\mathbf{F}(t)=\frac{1}{2\pi}\int_{\R}\widehat{\varphi}(\eta)\mathbf{F}(t)\e^{it\eta}d\eta\,.
		\]
		Since $\lg\tau\rg^{\gamma}\lesssim \lg\eta\rg^{\gamma}\lg\tau-\eta\rg^{\gamma}$, we deduce that
		\[
		\|\varphi(t)\mathbf{F}(t)\|_{\mathcal{F}L_q^{\gamma}(\mathcal{X})}\lesssim \int_{\R}|\widehat{\varphi}(\eta)|
		\lg\eta\rg^{\gamma}\|
		\lg\tau-\eta\rg^{\gamma}\widehat{\mathbf{F}}(\tau-\eta)\|_{L_{\tau}^q(\mathcal{X})}d\eta\leq \|\varphi\|_{\mathcal{F}L_{1}^{\gamma}(\R)} \|\mathbf{F}\|_{\mathcal{F}L_q^{\gamma}(\mathcal{X})}.
		\]
		Inequality (1) then follows from the fact that $\varphi\in\mathcal{S}(\R)$.
\medskip

		(2) Let us prove the second inequality with time cutoff $\varphi_T(t)$. 
%
		We split $\mathbf{F}=\mathbf{F}_1+\mathbf{F}_2$ with
		\[ \widehat{\mathbf{F}}_1(\tau)=\mathbf{1}_{|\tau|\geq\frac{1}{T}}\widehat{\mathbf{F}}(\tau),\quad \widehat{\mathbf{F}}_2(\tau)=\mathbf{1}_{|\tau|<\frac{1}{T}}\widehat{\mathbf{F}}(\tau),
		\] 
		where $\widehat{\mathbf{F}}(\tau)$ is the time-Fourier transform of $\mathbf{F}$.
		It suffices to show that
		\begin{align}\label{appendix2-2}
			\tag{a}
			\|\lg\tau\rg^{\gamma}(\widehat{\varphi_T}\ast\widehat{\mathbf{F}}_1)(\tau)\|_{L_{\tau}^q(\mathcal{X})}
			&\lesssim T^{\gamma_1-\gamma}\|\lg\eta\rg^{\gamma_1}\widehat{\mathbf{F}}(\eta) \|_{L_{\eta}^q(\mathcal{X})}\,,\\
			\label{appendix2-22}
			\tag{b}
			\|\lg\tau\rg^{\gamma}(\widehat{\varphi_T}*\widehat{\mathbf{F}}_2)(\tau)\|_{L_{\tau}^q(\mathcal{X})}
			&\lesssim T^{\gamma_1-\gamma}\|\lg\eta\rg^{\gamma_1}\widehat{\mathbf{F}}(\eta) \|_{L_{\eta}^q(\mathcal{X})}.
		\end{align}
		To prove \eqref{appendix2-2}, we note that for fixed $\tau\in\R$,
		\[
			\|\lg\tau\rg^{\gamma}(\widehat{\varphi_T}*\widehat{\mathbf{F}}_1)(\tau)
			\|_{\mathcal{X}}
			=\big
			\|
			\int_{\R}K_T^+(\tau,\eta)\lg \eta\rg^{\gamma_1}\widehat{\mathbf{F}}(\eta)d \eta
			\big
			\|_{\mathcal{X}} 
			\leq 
			\int_{\R}|K_T^+(\tau,\eta)|\lg \eta\rg^{\gamma_1}\|\widehat{\mathbf{F}}(\eta)
			\|_{\mathcal{X}}d \eta\,,
		\]
		where
		\[
		K_T^+(\lambda,\eta)
		=\widehat{\varphi}(T(\lambda-\eta))T\frac{\lg\lambda\rg^{\gamma}}{\lg \eta\rg^{\gamma_1}}\mathbf{1}_{|\eta|\geq \frac{1}{T}}\,.
		\]
		By Young's inequality, it suffices to show that there exists $C>0$, not depending on~$T$, such that 
		\[
			\sup_{\lambda\in\R}\int_{\R}|K_T^+(\lambda,\eta)|d \eta
			\ +\ 
			\sup_{\eta\in\R}\int_{\R}|K_T^+(\lambda,\eta)|d\lambda
			\leq CT^{\gamma_1-\gamma }\,.
		\]
		One can verify these two inequalities by direct computation. Here, we provide an informal explaination. Since $\widehat{\varphi}$ is a Schwartz function,  $|\lambda-\eta|$ is essentially bounded by $O(\frac{1}{T})$. Due to the fact that $|\eta|\geq \frac{1}{T}$, $\lambda$ is also essentially constrained in the region $ |\lambda| \sim |\tau|\gtrsim \frac{1}{T}$. Hence, when freezing one of the two variables $\lambda$ or $\eta$ the other one is restricted to an interval of length $\lesssim\frac{1}{T}$. We deduce that two integrations are bounded by $O(1)T^{\gamma_1-\gamma}$. We note that there is no constraint about $0\leq \gamma\leq \gamma_1<1$ for this contribution. 
		\medskip
		
		Next we prove \eqref{appendix2-22}. 
		By splitting $\widehat{\varphi_T}(\lambda-\eta)=[\widehat{\varphi_T}(\lambda-\eta)-\widehat{\varphi_T}(\lambda)]+\widehat{\varphi_T}(\lambda)$, we have
		\begin{align}\label{splitingF2}
			&\|(\widehat{\varphi_T}\ast\widehat{\mathbf{F}}_2)(\lambda)\|_{\mathcal{X}}=\big\|\int_{|\eta|<\frac{1}{T}} \wh{\varphi_T}(\lambda-\eta)\widehat{\mathbf{F}}(\eta)d \eta
			\big\|_{\mathcal{X}}\notag \\
			\leq &
			\big\|
			\int_{|\eta|<\frac{1}{T}}[\widehat{\varphi_T}(\lambda-\eta)-\widehat{\varphi_T}(\lambda)]\widehat{\mathbf{F}}(\eta)d\eta
			\big\|_{\mathcal{X}}+|\widehat{\varphi_T}(\lambda)|\big\|\int_{|\eta|<\frac{1}{T}}\widehat{\mathbf{F}}(\eta)d\eta  
			\big\|_{\mathcal{X}}.
		\end{align}
		Since $\widehat{\varphi_T}(\lambda)=T\widehat{\varphi}(T\lambda)$, we have 
		\[
		[\widehat{\varphi_T}(\lambda-\eta)-\widehat{\varphi_T}(\lambda)]\mathbf{1}_{T|\eta|<1}=T\cdot \mathcal{O}(|T\eta|\lg T\lambda\rg^{-100}).
		\]
		By H\"older, we deduce that
		\begin{multline*}  
		\big|
		\int_{|\eta|<\frac{1}{T}}\|\widehat{\mathbf{F}}(\eta)\|_{\mathcal{X}}[\widehat{\varphi_T}(\lambda-\eta)-\widehat{\varphi_T}(\lambda)]d\eta 
		\big| 
			\leq T^2\lg T\lambda\rg^{-100}
			\|\lg\eta\rg^{\gamma_1}\widehat{\mathbf{F}}(\eta)
			\|_{L_{\eta}^q(\mathcal{X})} 
			\|\frac{|\eta|}{\lg\eta\rg^{\gamma_1}}
			\|_{L^{q'}(|\eta|\leq \frac{1}{T})} 
			\\
			\lesssim T^{\gamma_1+1-\frac{1}{q'}}\lg T\lambda\rg^{-100}\|\mathbf{F}\|_{\mathcal{F}L_q^{\gamma_1}(\mathcal{X})}.
		\end{multline*}
		Multiplying by $\lg\lambda\rg^{\gamma}$ and taking the $L_{\lambda}^q$ norm of the above quantity, we obtain the bound
		\[ 
		\big\|
		\lg\lambda\rg^{\gamma}
		\int_{|\eta|<\frac{1}{T}}[\widehat{\varphi_T}(\lambda-\eta)-\widehat{\varphi_T}(\lambda)]\widehat{\mathbf{F}}(\eta)d\eta
		\big\|_{L_{\lambda}^q(\mathcal{X})}
		\lesssim
		T^{\gamma_1-\gamma}\|\mathbf{F}\|_{\mathcal{F}L_q^{\gamma_1}(\mathcal{X})}.
		\]
		Once again, we note that there is no constraint about $0\leq \gamma\leq \gamma_1<1$ for this contribution.
		
		Finally, 
		\begin{multline}
		\label{crudebound} 
			\big
			\|
			\lg
			\lambda\rg^{\gamma}\widehat{\varphi_T}(\lambda)
			\|
			\int_{|\eta|<\frac{1}{T}}\widehat{\mathbf{F}}(\eta)
			d\eta
			\|_{\mathcal{X}} 
			\big\|_{L_{\lambda}^q}
			\\
			\leq 
			\|
			\lg\lambda\rg^{\gamma}\widehat{\varphi_T}(\lambda)\|_{L_{\lambda}^q}\|\mathbf{F}\|_{\mathcal{F}L_q^{\gamma_1}(\mathcal{X})}
			\|
			\frac{1}{\lg\eta\rg^{\gamma_1}}\mathbf{1}_{|\eta|<\frac{1}{T}}
			\|_{L_{\eta}^{q'}}.
		\end{multline}
		Since for $0<T\leq 1$, $\lg\lambda\rg^{\gamma}\leq T^{-\gamma}\lg T\lambda\rg^{\gamma}$, we have $\|\lg\lambda\rg^{\gamma}\widehat{\varphi_T}(\lambda)\|_{L_{\lambda}^q}\|\lesssim T^{\frac{1}{q'}-\gamma}$.
		If $0\leq \gamma\leq \gamma_1< \frac{1}{q'}$, $\|\lg\eta\rg^{-\gamma_1}\mathbf{1}_{|\eta|<T^{-1}}\|_{L_{\eta}^{q'}}\lesssim T^{\gamma_1-\frac{1}{q'}}$, we obtain that
		\[ 
		\big\|
		\lg\lambda\rg^{\gamma}\widehat{\varphi_T}(\lambda)
		\big\|
		\int_{|\eta|<\frac{1}{T}}\widehat{\mathbf{F}}(\eta)d\eta
		\big\|_{\mathcal{X}} 
		\big\|_{L_{\lambda}^q}\lesssim T^{\gamma_1-\gamma}
		\|\mathbf{F}\|_{\mathcal{F}L_q^{\gamma_1}(\mathcal{X})}.
		\]
		When $\gamma_1=\frac{1}{q'}$, we get an extra factor $\log\big(1+\frac{1}{T}\big)$. When $\gamma_1>\frac{1}{q'}$ and $0\leq \gamma\leq \gamma_1$, we only get the bound 
		$ T^{\frac{1}{q'}-\gamma}\|\mathbf{F}\|_{\mathcal{F}L_q^{\gamma_1}(\mathcal{X})}.
		$
\medskip

		(3) If in addition $\mathbf{F}(0)=0$, we are able to improve the estimate~\eqref{appendix2-22}. This is an analogue of Proposition 2.7 of \cite{DNY1}.
		
		From the argument in (2), we realize that it suffices to improve the second term on the right hand side of \eqref{splitingF2}:
		\begin{align}\label{improvebound} 
			\Big\|\lg\lambda\rg^{\gamma}\widehat{\varphi_T}(\lambda)\Big\|\int_{|\eta|<\frac{1}{T}}\widehat{\mathbf{F}}(\eta)d\eta\Big\|_{\mathcal{X}} \Big\|_{L_{\lambda}^q}\lesssim T^{\gamma_1-\gamma}\|\mathbf{F}\|_{\mathcal{F}L_q^{\gamma_1}(\mathcal{X})}
		\end{align}
		in the range $\gamma_1> \frac{1}{q'}$.   
		The point is to exploit the cancellation from the condition $\mathbf{F}(0)=0$. By the Fourier inversion formula, we have
		\begin{align}\label{vanishing}  \int_{\R}\widehat{\mathbf{F}}(\eta)d \eta=0=\int_{|\eta|<\frac{1}{T}}\widehat{\mathbf{F}}(\eta)d \eta+\int_{|\eta|\geq \frac{1}{T}}\widehat{\mathbf{F}}(\eta)d \eta.
		\end{align}

		Therefore, the left hand side of \eqref{improvebound} is equal to
		\begin{align*}
			\big
			\|\lg\lambda\rg^{\gamma}T\widehat{\varphi}(T\lambda)
			\big
			\|\int_{|\eta|\geq \frac{1}{T}}\widehat{\mathbf{F}}(\eta)d\eta
			\big\|_{\mathcal{X}} 
			\big\|_{L_{\lambda}^q}.
		\end{align*}
		By H\"older's inequality and the fact that $\gamma_1>\frac{1}{q'}$,
		\begin{align*}
		\int_{|\eta|\geq \frac{1}{T}}\|\widehat{\mathbf{F}}(\eta)\|_{\mathcal{X}}d\eta
		&\leq \|\lg\eta\rg^{\gamma_1}\widehat{\mathbf{F}}(\eta)\|_{L_{\eta}^q(\mathcal{X})}\cdot T^{\gamma_1-\frac{1}{q'}},
		\intertext{hence (using $\lg\lambda\rg\leq T^{-1}\lg T\lambda\rg$ if $0<T\leq 1$)
		}
		\big
		\|\lg\lambda\rg^{\gamma}T\wh{\varphi}(T\lambda)\int_{|\eta|\geq \frac{1}{T}}\|\widehat{\mathbf{F}}(\eta)\|_{\mathcal{X}}d\eta
		\big\|_{L_{\lambda}^q}
		&\lesssim T^{\gamma_1-\gamma}
		\|\lg\eta\rg^{\gamma_1}\widehat{\mathbf{F}}(\eta)
		\|_{L_{\eta}^q(\mathcal{X})}.
		\end{align*}
		This completes the proof of Proposition \ref{timelocalization}.
	\end{proof}
	\begin{remarque}
		Note that on the right hand side of \eqref{crudebound}, when $\gamma_1>\frac{1}{q'}$, the function $\lg\eta\rg^{-\gamma_1}$belongs to $L^{q'}$, and we do not gain positive power in $T$ when integrating the region $|\eta|\leq \frac{1}{T}$ to compensate the power $T^{\frac{1}{q'}-\gamma}$ from $\|\lg\lambda\rg^{\gamma}\widehat{\varphi_T}(\lambda) \|_{L_{\lambda}^{q}}$. By applying vanishing property, finally we are allowed to integrate $\lg\eta\rg^{-\gamma_1}$ from the region $|\eta|>\frac{1}{T}$, and this explains the gain.
		
	\end{remarque}

	\begin{corollaire}\label{timecutoffappendix} 
		Let $1\leq p,r\leq \infty,$ and $1\leq q<\infty$, $\frac{1}{q'}<\gamma\leq \gamma_1<1$. For any $u\in X_{p,q,r}^{\gamma_1}$ such that $u(t=0,\cdot)=0$, we have
		\[
		\|\chi_T(t)u\|_{X_{p,q,r}^{s,\gamma}}\lesssim T^{\gamma_1-\gamma}\|u\|_{X_{p,q,r}^{s,\gamma_1}}
		\]
		for all $0<T<1$.
	\end{corollaire}
	
	\begin{proof} 
		
		It suffices to prove that for $u=\pi_nu\in E_n$ such that $u|_{t=0}=0$, the following inequality:
		\begin{align}\label{appendix2-1} \|\chi_T(t)u\|_{X_{q,r}^{0,\gamma}(E_n)}\lesssim T^{\gamma_1-\gamma}\|u\|_{X_{q,r}^{0,\gamma_1}(E_n) }.
		\end{align}	
		Note that with $v_n(t)=\e^{i\lambda_n^2t}u(t)$, this is equivalent to 
		\[ \|\chi_T(t)v_n(t)\|_{\mathcal{F}L_q^{\gamma}(L^r(E_n))}\lesssim T^{\gamma_1-\gamma}\|v_n\|_{\mathcal{F}L_q^{\gamma_1}(L^r(E_n))},
		\]
		the consequence of (3) of Proposition \ref{timelocalization}.
	\end{proof}

	\subsection{Embedding properties}

	\begin{lemme}\label{FLtoHolder}
		Let $\mathcal{X}$ be a Banach space, $1<q<\infty$ and $\gamma\in\big(\frac{1}{q'},1\big)$. For all $0<\alpha\leq \gamma-\frac{1}{q'}$, there exists $C>0$ such that for all $\mathbf{F}\in C_{c}^{\infty}(\R;\mathcal{X})$,
		\[
		\|\mathbf{F}\|_{C^{\alpha}(\R;\mathcal{X})}\leq C \|\lg\tau\rg^{\gamma}\widehat{\mathbf{F}}(\tau)\|_{L_q^{\gamma}(\mathcal{X})}\,.
		\]
	\end{lemme}
	
	\begin{proof}
		Pick $t,t'\in\R$ such that $0<|t-t'|<1$. By the inverse Fourier transform for vector-valued function $f\in C_c^{\infty}(\R;\mathcal{X})$, 
		\begin{align*}  \mathbf{F}(t)-\mathbf{F}(t')=&\frac{1}{2\pi}\int_{\R}\widehat{\mathbf{F}}(\tau)\e^{it'\tau}(\e^{i(t-t')\tau}-1)d\tau\\
			=&\frac{1}{2\pi}\Big(\int_{|\tau|\leq \frac{1}{|t-t'|}}+\int_{|\tau|>\frac{1}{|t-t'|}}\Big)\lg\tau\rg^{-\gamma}\cdot \lg\tau\rg^{\gamma}\widehat{\mathbf{F}}(\tau)\e^{it'\tau}(\e^{i(t-t')\tau}-1)d\tau.
		\end{align*}
		For the contribution $|\tau|> \frac{1}{|t-t'|}$, by H\"older,
		\begin{multline*}  \Big\|\int_{|\tau|>\frac{1}{|t-t'|}}\lg\tau\rg^{-\gamma}\cdot \lg\tau\rg^{\gamma}\widehat{\mathbf{F}}(\tau)\e^{it'\tau}(\e^{i(t-t')\tau}-1)d\tau\Big\|_{\mathcal{X}}
		\\
		\leq 2\|\langle\tau\rangle^{-\gamma}\mathbf{1}_{|\tau||t-t'|>1} \|_{L_{\tau}^{q'}}\|\lg\tau\rg^{\gamma}\widehat{\mathbf{F}}(\tau)\|_{L_{\tau}^q\mathcal{X}}
			\lesssim  |t-t'|^{\gamma-\frac{1}{q'}}\|\lg\tau\rg^{\gamma}\widehat{\mathbf{F}}(\tau)\|_{L^q(\R;\mathcal{X})}.
		\end{multline*}
		For the contribution $|\tau|\leq \frac{1}{|t-t'|}$, from $|\e^{i(t-t')\tau}-1|\leq |t-t'||\tau|$, we have
		\begin{multline*}  \Big\|\int_{|\tau|\leq \frac{1}{|t-t'|}}\lg\tau\rg^{-\gamma}\cdot \lg\tau\rg^{\gamma}\widehat{\mathbf{F}}(\tau)\e^{it'\tau}(\e^{i(t-t')\tau}-1)d\tau\Big\|_{\mathcal{X}}
		\\
		\leq |t-t'|\|\langle\tau\rangle^{1-\gamma}\mathbf{1}_{|\tau||t-t'|\leq 1} \|_{L_{\tau}^{q'}}\|\lg\tau\rg^{\gamma}\widehat{\mathbf{F}}(\tau)\|_{L_{\tau}^q\mathcal{X}}
			\lesssim  |t-t'|^{\gamma-\frac{1}{q'} }\|\lg\tau\rg^{\gamma}\widehat{\mathbf{F}}(\tau)\|_{L^q(\R;\mathcal{X})}.
		\end{multline*}
		The boundedness of $\|\mathbf{F}(t)\|_{\mathcal{X}}$ follows similarly, and we omit the details. The proof of Lemma \ref{FLtoHolder} is complete. 
	\end{proof}

	\begin{lemme}[Vector-valued inequality: Theorem 5.5.1 of \cite{Graf}]\label{vectorvaluedlemma}
		Let $\mathcal{H}$ be a separable Hilbert space and $T$ be a linear operator that is bounded from $L^p(\R;\C)$ to $L^r(\R;\C)$ for some $1<p,r<\infty$. Then $T$ can be extended as a bounded linear operator from $L^p(\mathbb{R};\mathcal{H})$ to $L^{r}(\mathbb{R};\mathcal{H})$, denoted by $\mathbf{T}$, such that
		\[
		\|\mathbf{T}(\mathbf{F})\|_{L^r(\R;\mathcal{H})}\lesssim_{p,r}
		\|T\|_{L^p(\R)\rightarrow L^r(\R)}\|\mathbf{F}\|_{L^p(\R;\mathcal{H})}.
		\]
		\end{lemme}
	In particular, the Fourier transform extends to the space $L^2(\mathbb{R};\mathcal{H})$.
	\begin{proof} 
		The proof can be found in the book \cite{Graf}. Here we provide the proof by using the probabilistic method. Let~$(\mathbf{h}_j)_{j\in\N}$ be an orthonormal basis of $\mathcal{H}$. Then any $\mathbf{F}\in \mathcal{S}(\R;\mathcal{H})$ can be represented as 
		$ \mathbf{F}=\sum_{j\geq 1}F_j\mathbf{h}_j,  
		$ where $F_j\in \mathcal{S}(\R;\C)$. It suffices to show that there exists a constant $C>0$ such that  for any finite number~$m$,
		\begin{equation}
		\label{vectorvalued}
		\big\|
		\big(
		\sum_{j=1}^m|TF_j|^2
		\big)^{\frac{1}{2}}
		\big\|_{L^r(\R)}
		\leq 
		C\|T\|_{L^p\rightarrow L^r}
		\big\|\big(\sum_{j=1}^m|F_j|^2
		\big)^{\frac{1}{2}}\big\|_{L^p(\R)}\,.
		\end{equation}
		Once this inequality is proved, a density argument allows us to conclude. Consider a sequence of i.i.d. random variables $(\epsilon_j^{\omega})_{j\in\N}$, such that
		\[
		 \mathbb{P}[\epsilon_j^{\omega}=+1]=\mathbb{P}[\epsilon_j^{\omega}=-1]=\frac{1}{2}\,,
		\]
		and the corresponding randomized function
		$ F^{\omega}:=\sum_{j=1}^m\epsilon_j^{\omega}F_j.
		$
		For each fixed event~$\omega$, from the boundedness of $T$, we have
		\[ \|TF^{\omega}\|_{L^r(\R)}\leq \|T\|_{L^p\rightarrow L^r}\|F^{\omega}\|_{L^p(\R)}.
		\]
		In the case when $r\geq p$, by taking the $L_{\omega}^r$ norm on both sides and using Minkowski, we have
		\begin{equation}
		\label{randombound}  
		\|TF^{\omega}\|_{L^r(\R;L_{\omega}^r)}
			\leq 
			\|T\|_{L^p\rightarrow L^r}\|F^{\omega}\|_{L^p(\R;L_{\omega}^r)}.
		\end{equation}
		By the Khinchin's inequality, for fixed $x\in\R^d$,
		\[
		\big(\sum_{j=1}^m|TF_j(x)|^2 
		\big)^{\frac{1}{2}}
		\lesssim_{r}
		 \big\|
		 \sum_{j=1}^m\epsilon_j^{\omega}TF_j(x)
		\big\|_{L_{\omega}^r}\,,
		\quad \big\|\sum_{j=1}^m\epsilon_j^{\omega}F_j(x)
		\big\|_{L_{\omega}^r}
		\lesssim_{r}
		\big(\sum_{j=1}^m|F_j(x)|^2 \big)^{\frac{1}{2}}. 
		\]
		Plugging into \eqref{randombound}, we obtain \eqref{vectorvalued}.
		Similarly, when \(r< p\) we obtain
		\begin{align*}
		\|TF^{\omega}\|_{L^r(\R;L_{\omega}^r)}
			&\leq 
			\|T\|_{L^p\rightarrow L^r}\|F^{\omega}\|_{L_{\omega}^{r}L^{p}(\R)}
			\\
			&\leq
			\|T\|_{L^p\rightarrow L^r}\|F^{\omega}\|_{L_{\omega}^{p}L^{p}(\R)}
			\\
			&\lesssim_{p} 
			\|T\|_{L^p\rightarrow L^r}
			\big\|\big(\sum_{j=1}^m|F_j(x)|^2 \big)^{\frac{1}{2}}\big\|_{L^{p}(\R)}\,,
		\end{align*}
		and we bound from below the left-hand-side as in the previous case (\(r\geq p\)). This completes the proof of Lemma~\ref{vectorvaluedlemma}.
	\end{proof}
	
	\begin{corollaire}[Hausdorff-Young]\label{Hausdorff-Young}
		Let $\mathcal{H}$ be a Hilbert space and $\mathbf{F}\in L^{p'}(\R;\mathcal{H})$ for some $p \geq  2$. Then $\mathbf{F}\in \mathcal{F}L_{p}^{0}(\mathcal{H})$ and
		\[
		\|\mathbf{F}\|_{\mathcal{F}L_{p}^{0}(\mathcal{H})}\lesssim_p \|\mathbf{F}\|_{L^{p'}(\R;\mathcal{H})}.
		\]
	\end{corollaire}
	\begin{proof}
		First, if $\mathbf{F}\in L^1(\mathbb{R};\mathcal{H})$, by the triangle inequality, we get
		\[ \|\widehat{\mathbf{F}}(\tau)\|_{\mathcal{H}}\leq \int_{\R}\|\mathbf{F}(t)\|_{\mathcal{H}}dt=\|\mathbf{F}\|_{L^1(\mathbb{R};\mathcal{H})}.
		\]
		If $2\leq p<\infty$, by Hausdorff-Young and Lemma \ref{vectorvaluedlemma}, the Fourier transform $\mathcal{F}$ extends to $\mathbf{F}\in L^{p'}(\R;\mathcal{H})$, and the desired inequality holds.
	\end{proof}

	\subsection{Inhomogeneous linear estimates}
	
	\begin{lemme}\label{inhomogenousabstract} 
		Let $\mathcal{X}$ be a Banach space and $\mathbf{F}:\R\rightarrow \mathcal{X}$ a Schwartz function such that for some $1<q<+\infty$,
		\[ \|\mathbf{F}\|_{\mathcal{F}L_q^{\gamma}(\mathcal{X})}:=\|\lg\tau\rg^{\gamma}\widehat{\mathbf{F}}(\tau)\|_{L^q(\R_{\tau};\mathcal{X})}<\infty.
		\]
		Then for any $\chi\in C_c^{\infty}(\R)$, $\chi_T(t):=\chi(T^{-1}t)$ for $ 0<T\leq 1$ and $\gamma$ with $\frac{1}{q'}<\gamma \leq 1$,
		\[
		\|\chi_T(t)\int_{0}^{t}\mathbf{F}(t')dt'\|_{\mathcal{F}L_q^{\gamma}(\mathcal{X})}\lesssim \|\mathbf{F}\|_{\mathcal{F}L_q^{\gamma-1}(\mathcal{X})},
		\]
		where the implicit constant does not depend on $T>0$. 
	\end{lemme}
	\begin{proof} 
		We follow the argument in \cite{Gin}. Write
			\[
			\chi_T(t)\int_{0}^{t}\mathbf{F}(t')dt'
			=\frac{\chi_T(t)}{2\pi}\int_0^tdt'\int_{\R}\widehat{\mathbf{F}}(\tau)\e^{it'\tau}d\tau
			=\frac{\chi_T(t)}{2\pi}\int_{\R}\widehat{\mathbf{F}}(\tau)\frac{\e^{it\tau}-1}{i\tau}d\tau
			\]
			We consider three different contributions:
			\begin{multline*}
			\frac{\chi_T(t)}{2\pi}\int_{\R}\widehat{\mathbf{F}}(\tau)\frac{\e^{it\tau}-1}{i\tau}d\tau
			=\underbrace{\frac{\chi_T(t)}{2\pi}\sum_{k\geq 0}\frac{t^k}{k!}\int_{T|\tau|\leq 1} (i\tau)^{k-1}\widehat{\mathbf{F}}(\tau)d\tau}_{\mathrm{I}}
			\\-
			\underbrace{\frac{\chi_T(t)}{2\pi}\int_{T|\tau|> 1} (i\tau)^{-1}\widehat{\mathbf{F}}(\tau)d\tau}_{\mathrm{II}}
			+ \underbrace{\frac{\chi_T(t)}{2\pi}\int_{T|\tau|> 1} \e^{it\tau}(i\tau)^{-1}\widehat{\mathbf{F}}(\tau)d\tau}_{\mathrm{III}}\,.
		\end{multline*}
		Provided that $\gamma\leq 1$,
		\begin{multline*}
		\int_{T|\tau|\leq 1}|(i\tau)^{k-1}|\|\widehat{\mathbf{F}}(\tau)\|_{\mathcal{X}}d\tau 
		\leq T^{-(k-1)}\|\mathbf{F}\|_{\mathcal{F}L_q^{\gamma-1}(\mathcal{X})}\Big(\int_{|\tau|T\leq 1}\lg\tau\rg^{(1-\gamma)q'}d\tau\Big)^{\frac{1}{q'}}
		\\
		\leq CT^{-(k-\gamma+\frac{1}{q'}) }\|\mathbf{F}\|_{\mathcal{F}L_q^{\gamma-1}(\mathcal{X})}\,.
		\end{multline*}
		Similarly, provided that $\gamma>\frac{1}{q'}$,
		\[
		\int_{T|\tau|>1}|\tau|^{-1}\|\widehat{\mathbf{F}}(\tau)\|_{\mathcal{X}}d\tau
		\leq
		\|\mathbf{F}\|_{\mathcal{F}L_q^{\gamma-1}(\mathcal{X})}\Big(\int_{T|\tau|>1}\lg\tau\rg^{-q'\gamma}d\tau\Big)^{\frac{1}{q'}}
			\leq CT^{\gamma-\frac{1}{q'}}
			\|\mathbf{F}\|_{\mathcal{F}L_q^{\gamma-1}(\mathcal{X})}\,.
		\]
		Then
		\begin{multline*}
			\|\mathrm{I}\|_{\mathcal{F}L_q^{\gamma}(\mathcal{X})}
			\leq C\sum_{k\geq 0}\frac{
			\|\chi_T(t)t^k\|_{\mathcal{F}L_q^{\gamma}(\R)}
		}{k!}
			\big\|
			\int_{T|\tau|\leq 1}(i\tau)^{k-1}\widehat{\mathbf{F}}(\tau)d\tau 
			\big\|
			_{\mathcal{X}}
			\\
			\leq C\sum_{k\geq 0}\frac{k^{2}T^{k-\gamma+\frac{1}{q'}}}{k!}\cdot T^{-(k-\gamma+\frac{1}{q'})}\|\mathbf{F}\|_{\mathcal{F}L_q^{\gamma-1}(\mathcal{X})}
			\leq C\|\mathbf{F}\|_{\mathcal{F}L_q^{\gamma-1}(\mathcal{X})}\,,
		\end{multline*} 
	where we use the fact that
	\[
	\|\chi_T(t)(t/T)^k \|_{\mathcal{F}L_q^{\gamma}(\R)}\lesssim T^{\frac{1}{q'}-\gamma} \|\chi(t)t^{k}\|_{\mathcal{F}L_q^{\gamma}(\R)}\lesssim T^{\frac{1}{q'}-\gamma}\|\chi(t)t^k\|_{W_t^{2,q'}}\lesssim k^2T^{\frac{1}{q'}-\gamma}\,.
	\]
		This also prove that, provided $\gamma>\frac{1}{q'}$,
		\[
		\|\mathrm{II}\|_{\mathcal{F}L_q^{\gamma}(\mathcal{X})}\leq C\|\chi_T(t)\|_{\mathcal{F}L_q^{\gamma}(\R)}T^{\gamma-\frac{1}{q'}}\|\mathbf{F}\|_{\mathcal{F}L_q^{\gamma-1}(\mathcal{X})}.
		\]
		For III, set
		\[ J_T(t):=\frac{1}{2\pi}\int_{T|\tau|>1} \frac{\widehat{\mathbf{F}}(\tau)}{i\tau}\e^{it\tau}d\tau,
		\]
		then by H\"older, we have
		\[ \|J_T\|_{\mathcal{F}L_q^{\gamma}(\mathcal{X})}\leq C\|\mathbf{F}\|_{\mathcal{F}L_q^{\gamma-1}(\mathcal{X})},\quad \|J_T\|_{\mathcal{F}L_q^{0}(\mathcal{X})}\leq CT^{\gamma}\|\mathbf{F}\|_{\mathcal{F}L_q^{\gamma-1}(\mathcal{X})}.
		\]
		From $\lg\eta\rg^{\gamma}\leq C\lg\eta-\tau\rg^{\gamma}+C\lg\tau\rg^{\gamma}$ for any $\tau\in\R$, we deduce that
		\begin{align*}  \|\mathrm{III}\|_{\mathcal{F}L_q^{\gamma}(\mathcal{X})}\leq & C\|(\lg\cdot\rg^{\gamma}\widehat{\chi_T})\ast \widehat{J}_T(\eta) \|_{L_{\eta}^q\mathcal{X}}+ C\|\widehat{\chi_T}\ast (\lg\cdot\rg^{\gamma}\widehat{J}_T)(\eta) \|_{L_{\eta}^q\mathcal{X}}\\
			\leq & C\|\lg\cdot\rg^{\gamma}\widehat{\chi_T}\|_{L^1}\|J_T\|_{\mathcal{F}L_q^{0}(\mathcal{X})}+C\|\widehat{\chi_T}\|_{L^1}\|J_T\|_{\mathcal{F}L_q^{\gamma}(\mathcal{X})}\\
			\leq &C\|\mathbf{F}\|_{\mathcal{F}L_q^{\gamma-1}(\mathcal{X})}.
		\end{align*}
		where to the second inequality, we used Young's convolution inequality, and 
		\[ \|\lg\cdot\rg^{\gamma}\widehat{\chi_T}\|_{L^1}\lesssim T^{-\gamma}.
		\] 
		The proof of Lemma \ref{inhomogenousabstract} is now complete.
	\end{proof}
	
	\begin{lemme}\label{inhomogeneousabstract2}
		Let $\chi\in C_c^{\infty}(\R), 0<T<1$,  $1<q<\infty, \theta\geq 0, \gamma\in\big(\frac{1}{q'},1\big)$. Then for any  $\mathbf{A}(t)\in\mathcal{S}(\R;\mathcal{L}(\mathcal{X}) )$, the operator $\chi_T(t)\int_0^t\mathbf{A}(t')\d t'\in \mathcal{L}(\mathcal{F}L_q^{\gamma}(\mathcal{X}) )$. Moreover, for any $\mathbf{F}(t)\in \mathcal{F}L_q^{\gamma}(\mathcal{X})$,
		\[
		\Big\|\chi_T(t)\Big(\int_0^{t}\mathbf{A}(t')\d t'\Big)(\mathbf{F}(t)) \Big\|_{\mathcal{F}L_q^{\gamma}(\mathcal{X}) }\lesssim T^{\theta-\gamma+\frac{1}{q'}} \|\mathbf{A}\|_{\mathcal{L}(\mathcal{X},\mathcal{F}L_q^{\gamma-1+\theta}(\mathcal{X}) )} \|\mathbf{F}(t)\|_{\mathcal{F}L_{q}^{\gamma}(\mathcal{X}) }.
		\]  
	\end{lemme}
\begin{remarque}
The reason for which we make the assumption  $\mathbf{A}(t)\in \mathcal{S}(\R;\mathcal{L}(\mathcal{X}))$ is to avoid the issue of justifying interchange of orders of integral as well as the time Fourier transform. This regularity assumption is always satisfied in our application, as every object is smooth (in time and in space).
\end{remarque}
\begin{remarque}
Comparing to Lemma \ref{inhomogenousabstract}, we have an extra loss $T^{-\gamma+\frac{1}{q'}}$. This is basically due to the fact that $\|\chi_T(t)\mathbf{F}(t) \|_{\mathcal{F}L_q^{\gamma}(\mathcal{X})}\lesssim T^{\frac{1}{q'}-\gamma}\|\mathbf{F}(t)\|_{\mathcal{F}L_q^{\gamma}(\mathcal{X})}$ since $\gamma>\frac{1}{q'}$. 
\end{remarque}
	
	\begin{proof}
			The proof is similar to the proof of Lemma \ref{inhomogenousabstract}. Without loss of generality, we assume that $\mathbf{F}(t)\in\mathcal{S}(\R;\mathcal{X})$ so that all manipulation can be justified.  The function $\mathbf{A}(t')\mathbf{F}(t)$ is well-defined with its time Fourier transform (with respect to $t'$ variable) is $\widehat{\mathbf{A}}(\tau_1)\mathbf{F}(t)$.  Moreover, the Fourier transform in $t$-variable is $\widehat{\mathbf{A}}(\tau_1)\widehat{\mathbf{F}}(\tau)$.
			
			 We write $	\chi_T(t)\int_0^t\mathbf{A}(t')\mathbf{F}(t)\d t'$ as
		\begin{align*}
		&\frac{\chi_T(t)}{2\pi}\int_0^t\d t'\int_{\R}\widehat{\mathbf{A}}(\tau_1)\mathbf{F}(t)\e^{it_1\tau_1}\d \tau_1
			=\frac{\chi_T(t)}{2\pi}\int_{\R}
		\frac{\e^{it\tau_1}-1}{i\tau_1}
		\widehat{\mathbf{A}}(\tau_1)\mathbf{F}(t)
		\d \tau_1\\
			=&\underbrace{\frac{\chi_T(t)}{2\pi}\int_{T|\tau_1|\leq 1}\frac{\widehat{\mathbf{A}}(\tau_1)\mathbf{F}(t)}{i\tau_1}\sum_{k=1}^{\infty}\frac{(i\tau_1t)^k}{k!}\d\tau_1}_{\mathrm{I}}-\underbrace{\frac{\chi_T(t)}{2\pi}\int_{T|\tau_1|>1}\widehat{\mathbf{A}}(\tau_1)\mathbf{F}(t)\frac{1}{i\tau_1}\d\tau_1}_{\mathrm{II}} \\
			+&\underbrace{\frac{\chi_{T}(t)}{2\pi}\int_{T|\tau_1|>1}\widehat{\mathbf{A}}(\tau_1)\mathbf{F}(t)\frac{\e^{it\tau_1}}{i\tau_1}d\tau_1}_{\mathrm{III}}.
		\end{align*}

\noi
$\bullet${\bf Estimate for I}:  	Note that
\begin{align*} 
	&\int_{T|\tau_1|\leq 1}|(i\tau_1)^{k-1}|\|\lg\tau\rg^{\gamma}\widehat{\mathbf{A}}(\tau_1) \mathcal{F}_{t}(\chi_T(t)t^k\mathbf{F})(\tau)\|_{\mathcal{X}}\d\tau_1\\  \leq &T^{-k+1}\|\lg\tau\rg^{\gamma}\lg\tau_1\rg^{\gamma-1+\theta}\widehat{\mathbf{A}}(\tau_1)\mathcal{F}_{t}(\chi_T(t)t^k\mathbf{F})(\tau)\|_{L_{\tau_1}^q(\mathcal{X})}\Big(\int_{T|\tau_1|\leq 1}\lg\tau_1\rg^{(1-\gamma-\theta)q'}\d\tau_1\Big)^{\frac{1}{q'}}\\ \lesssim & T^{\gamma-\frac{1}{q'}+\theta}\|\lg\tau\rg^{\gamma}\lg\tau_1\rg^{\gamma-1+\theta}\widehat{\mathbf{A}}(\tau_1)\mathcal{F}_{t}(\chi_T(t))(t/T)^k\mathbf{F}))(\tau)\|_{L_{\tau_1}^q(\mathcal{X})}.
\end{align*}
Set $\mathbf{G}_k(t)= \chi_T(t)(t/T)^k\mathbf{F}(t)$. From Remark \ref{remtimelocalization} and Proposition \ref{timelocalization},
	\[ \|\mathbf{G}_k \|_{\mathcal{F}L_q^{\gamma}(\mathcal{X}) }\lesssim T^{\frac{1}{q'}-\gamma} \|\chi(t)t^k\|_{W_t^{10,1}} \lesssim  k^{10}T^{\frac{1}{q'}-\gamma}\|\mathbf{F}\|_{\mathcal{F}L_q^{\gamma}(\mathcal{X}) },
	\]
where the implicit constant is independent of $k$. 
Therefore,
\begin{align*}
\| \mathrm{I}\|_{\mathcal{F}L_q^{\gamma}(\mathcal{X})} \lesssim & T^{\gamma-\frac{1}{q'}+\theta}
\sum_{k=1}^{\infty}\frac{\big\| \| \lg\tau\rg^{\gamma}\lg\tau_1\rg^{\gamma-1+\theta}\widehat{\mathbf{A}}(\tau_1)\widehat{\mathbf{G}_k}(\tau)\|_{\mathcal{X}}  \big \|_{L_{\tau,\tau_1}^q }}{k!} \\
=& T^{\gamma-\frac{1}{q'}+\theta}
\sum_{k=1}^{\infty}
\frac{\big\| \| \lg\tau\rg^{\gamma}\lg\tau_1\rg^{\gamma-1+\theta}\widehat{\mathbf{A}}(\tau_1)\widehat{\mathbf{G}_k}(\tau)\|_{L_{\tau_1}^q\mathcal{X}}  \big \|_{L_{\tau}^q }}{k!}\\
\leq & T^{\gamma-\frac{1}{q'}+\theta}
\sum_{k=1}^{\infty}\frac{ \big\|
\|\mathbf{A}\|_{\mathcal{X}\rightarrow \mathcal{F}L_q^{\gamma-1+\theta} } \|\lg\tau\rg^{\gamma}\widehat{\mathbf{G}_k}(\tau) \|_{\mathcal{X}}
\big\|_{L_{\tau}^q}
}{k!}\\
\lesssim &T^{\theta} \|\mathbf{A}\|_{\mathcal{X}\rightarrow \mathcal{F}L_q^{\gamma-1+\theta} }
\sum_{k=1}^{\infty}\frac{k^{10}\|\mathbf{F}\|_{\mathcal{F}L_q^{\gamma}(\mathcal{X})} }{k!}\\
\lesssim &
T^{\theta}\|\mathbf{A}\|_{\mathcal{X}\rightarrow \mathcal{F}L_q^{\gamma-1+\theta} }
\|\mathbf{F}\|_{\mathcal{F}L_q^{\gamma}(\mathcal{X})},
\end{align*}
as desired. 
\medskip

\noi
$\bullet${\bf Estimate for II}: \begin{align*} \int_{T|\tau_1|>1}&\frac{\|
			\lg\tau \rg^{\gamma}
			\mathbf{A}(\tau_1)\widehat{\chi_T\mathbf{F}}(\tau)\|_{\mathcal{X}}}{|\tau_1| }\d\tau_1\\ \leq &\|\lg\tau\rg^{\gamma}\lg\tau_1\rg^{\gamma-1+\theta}\widehat{\mathbf{A}}(\tau_1)\widehat{\chi_T\mathbf{F}}(\tau) \|_{L_{\tau_1}^q(\mathcal{X})}\Big(\int_{T|\tau_1|>1}\lg\tau_1\rg^{-q'(\gamma+\theta)}\d\tau_1\Big)^{\frac{1}{q'}}\\
	\lesssim & T^{\gamma-\frac{1}{q'}+\theta} \|\lg\tau\rg^{\gamma} \lg\tau_1\rg^{\gamma-1+\theta}\widehat{\mathbf{A}}(\tau_1)\widehat{\chi_T\mathbf{F}}(\tau) \|_{L_{\tau_1}^q(\mathcal{X})},
\end{align*}
provided that $\gamma>\frac{1}{q'}$. The right hand side can be bounded by
	\[
	 T^{\gamma-\frac{1}{q'}+\theta}\|\mathbf{A}\|_{\mathcal{X}\rightarrow \mathcal{F}L_q^{\gamma-1+\theta} }\|\lg\tau\rg^{\gamma}
\widehat{\chi_T\mathbf{F}}(\tau) \|_{\mathcal{X}}.
	\]
Taking the $L_{\tau}^q$-norm, we obtain the bound
	\[
	T^{\gamma-\frac{1}{q'}+\theta}
\|\mathbf{A}\|_{\mathcal{X}\rightarrow \mathcal{F}L_q^{\gamma-1+\theta} }\|\chi_T\mathbf{F} \|_{\mathcal{F}L_q^{\gamma}(\mathcal{X})}\lesssim T^{\theta}
\|\mathbf{A}\|_{\mathcal{X}\rightarrow \mathcal{F}L_q^{\gamma-1+\theta} }
\|\mathbf{F}\|_{\mathcal{F}L_q^{\gamma}(\mathcal{X})},
	\]
where we have used (2) of Proposition \ref{timelocalization}.
\medskip

\noi
$\bullet${\bf Estimate of III :}
	Finally, for III,  we denote by 
	\[ \mathbf{J}(\tau)=\frac{1}{2\pi}\int_{T|\tau_1|>1}\frac{\widehat{\mathbf{A}}(\tau_1)\widehat{\chi_T\mathbf{F}}(\tau-\tau_1)   }{i\tau_1}\d\tau_1
	\]
and we need to control $\|\lg\tau\rg^{\gamma} \mathbf{J}(\tau) \|_{L_{\tau}^q\mathcal{X} }$. First,
\begin{align*}
\lg\tau\rg^{\gamma}\|\mathbf{J}(\tau)\|_{\mathcal{X}}\leq & \int_{T|\tau_1|>1}\lg\tau\rg^{\gamma}\lg\tau_1\rg^{-1}\|\widehat{\mathbf{A}}(\tau_1)\widehat{\chi_T\mathbf{F}}(\tau-\tau_1)\|_{\mathcal{X}}\d \tau_1.
\end{align*}
Next, we observe that if $T|\tau_1|>1$ and $|\tau|\lesssim |\tau_1|$,
$ \lg\tau\rg^{\gamma}\lg\tau_1\rg^{-1}\lesssim T^{\theta}\lg\tau_1\rg^{\gamma-1+\theta}.
$
If  $|\tau|\gg |\tau_1|$, we have $|\tau|\lesssim |\tau-\tau_1|$,  hence
$ \lg\tau\rg^{\gamma}\lg\tau_1\rg^{-1}\sim T^{\theta}\lg\tau-\tau_1\rg^{\gamma}\lg\tau_1\rg^{\gamma-1+\theta}\lg\tau_1\rg^{-\gamma}.
$
Thus
\begin{align}\label{control:mathbfJ}
\lg\tau\rg^{\gamma}\|\mathbf{J}(\tau)\|_{\mathcal{X}}\lesssim & T^{\theta}\int_{T|\tau_1|>1} \|
\lg\tau_1\rg^{\gamma-1+\theta}
\widehat{\mathbf{A}}(\tau_1)\widehat{\chi_T\mathbf{F}}(\tau-\tau_1) \|_{\mathcal{X}}\d \tau_1\notag  \\
+&T^{\theta}\int_{T|\tau_1|>1}\| \lg\tau_1\rg^{\gamma-1+\theta}\lg\tau-\tau_1\rg^{\gamma}
\widehat{\mathbf{A}}(\tau_1)\widehat{\chi_T\mathbf{F}}(\tau-\tau_1)
\|_{\mathcal{X}}\cdot \lg\tau_1\rg^{-\gamma} \d \tau_1.
\end{align}
By changing variables, the first term on the right hand side can be written as
	\[
	 T^{\theta}\int_{T|\tau-\tau_1|>1}\|\lg\tau-\tau_1\rg^{\gamma-1+\theta}\widehat{\mathbf{A}}(\tau-\tau_1)\widehat{\chi_T\mathbf{F}}(\tau_1) \|_{\mathcal{X}}\d \tau_1\,.
	\]
By Minkowski, its $L_{\tau}^q$-norm can be bounded by
\begin{align*}
 &T^{\theta}\int_{\R}\|\lg\tau-\tau_1\rg^{\gamma-1+\theta}\widehat{\mathbf{A}}(\tau-\tau_1) \widehat{\chi_T\mathbf{F}}(\tau_1) \|_{L_{\tau}^q\mathcal{X}}\d \tau_1\\
 = & T^{\theta}\int_{\R}\|\lg\tau\rg^{\gamma-1+\theta}\widehat{\mathbf{A}}(\tau) \lg\tau_1\rg^{\gamma}\widehat{\chi_T\mathbf{F}}(\tau_1) \|_{L_{\tau}^q\mathcal{X}}\lg\tau_1\rg^{-\gamma}\d \tau_1 \\
 \leq & T^{\theta}\|\mathbf{A}\|_{\mathcal{X}\rightarrow \mathcal{F}L_q^{\gamma-1+\theta}(\mathcal{X}) }\|\chi_T\mathbf{F}\|_{\mathcal{F}L_q^{\gamma}(\mathcal{X})}\|\lg\tau_1\rg^{-\gamma} \|_{L_{\tau_1}^{q'}}\\
 \lesssim & T^{\theta+\frac{1}{q'}-\gamma}\|\mathbf{A}\|_{\mathcal{X}\rightarrow \mathcal{F}L_q^{\gamma-1+\theta}(\mathcal{X}) } \|\mathbf{F}\|_{\mathcal{F}L_q^{\gamma}(\mathcal{X})},
\end{align*}
where to the last step, we have used H\"older and (2) in Proposition \ref{timelocalization}.

Similarly, to control the $L_{\tau}^q$-norm of the second term on the right hand side of \eqref{control:mathbfJ}, we apply H\"older and Minkowski to get 
\begin{align*}
&T^{\theta}\big\| \|\lg\tau-\tau_1\rg^{\gamma-1+\theta} \lg\tau_1\rg^{\gamma}\widehat{\mathbf{A}}(\tau-\tau_1)\widehat{\chi_T\mathbf{F}}(\tau_1) \|_{L_{\tau_1}^q\mathcal{X}}\|\lg\tau-\tau_1\rg^{-\gamma}\mathbf{1}_{T|\tau-\tau_1|>1 } \|_{L_{\tau_1}^{q'} }
\big\|_{L_{\tau}^q}\\
\lesssim & T^{\theta+\gamma-\frac{1}{q'}}\big\|
\|\lg\tau-\tau_1\rg^{\gamma-1+\theta}\widehat{\mathbf{A}}(\tau-\tau_1)
\lg\tau_1\rg^{\gamma}
\widehat{\chi_T\mathbf{F}}(\tau_1) \|_{L_{\tau}^q\mathcal{X}}
  \big\|_{L_{\tau_1}^q} \\
 \lesssim & T^{\theta+\gamma-\frac{1}{q'}}\|\mathbf{A} \|_{\mathcal{X}\rightarrow \mathcal{F}L_q^{\gamma-1+\theta}(\mathcal{X}) }\|\chi_T\mathbf{F} \|_{\mathcal{F}L_q^{\gamma}}\\
 \lesssim & T^{\theta}
 \|\mathbf{A} \|_{\mathcal{X}\rightarrow \mathcal{F}L_q^{\gamma-1+\theta}(\mathcal{X}) }\|\mathbf{F} \|_{\mathcal{F}L_q^{\gamma}},
\end{align*}
where we used (2) in Proposition \ref{timelocalization} in the last step. This completes the proof of Lemma \ref{inhomogeneousabstract2}.\end{proof}

	\section{Some elementary facts}

	\begin{lemme}\label{lem:unitary}
		Let $\mathcal{A}(t)$ be a operator-valued function that is self-adjoint on a  Hilbert space $\mathcal{X}$ for every $t\in\R$. Let $\mathcal{H}(t)$ be the solution operator of the well-posed Cauchy problem
		\[
		i\partial_t\mathcal{H}(t)=\mathcal{A}(t)\mathcal{H}(t),\quad \mathcal{H}(0)=\mathrm{Id}\,. 
		\]	
		Then $\mathcal{H}(t)$ is unitary. 
	\end{lemme}
	\begin{proof} 
		Note that $\mathcal{H}(t)$ is a linear operator for any $t\in\R$. Since $\mathcal{H}(0)=\mathrm{Id}$, it suffices to show that  $\mathcal{H}(t)$ preserves the norm of the Hilbert-space. the For any $f\in\mathcal{X}$, consider
		\begin{align*}
			\frac{d}{dt}\|\mathcal{H}(t)f\|_{\mathcal{X}}^2=2\Re(\partial_t\mathcal{H}(t)f,\mathcal{H}(t)f)_{\mathcal{X}}=2\Im(\mathcal{A}(t)\mathcal{H}(t)f,\mathcal{H}(t)f)_{\mathcal{X}}=0,
		\end{align*}
		since $(\mathcal{A}(t))^*=\mathcal{A}(t)$ for all $t\in\R$. Therefore, $\|\mathcal{H}(t)f\|_{\mathcal{X}}=\|\mathcal{H}(0)f\|_{\mathcal{X}}=\|f\|_{\mathcal{X}}$, for any $f\in\mathcal{X}$. This leads to $\mathcal{H}(t)\mathcal{H}(t)^*=\mathrm{Id}$.
	\end{proof}
	
	\begin{lemme}\label{lem:unitaryextend}
		Let $\mathcal{A}(t)$ be in Lemma \ref{lem:unitary}. Let $0<T<\frac{1}{2}$. Assume that $\chi\in C_c^{\infty}(\R)$ is a bump function that equals to $1$ on $|t|\leq \frac{1}{2}$ and vanishes for $|t|\geq 1$. Let  $\mathcal{H}^{\dag}(t)$ and $\mathcal{G}^{\dag}(t)$ be the unique solutions of the equations
		\begin{align*}
			&\mathcal{H}^{\dag}(t)=\chi(t)\mathrm{Id}-i\chi\big(\frac{t}{2T}\big)\int_0^t\mathcal{A}(t')\mathcal{H}^{\dag}(t')dt',\\
			&\mathcal{G}^{\dag}(t)=\chi(t)\mathrm{Id}+i\chi\big(\frac{t}{2T}\big)\int_0^t\mathcal{A}(t')\mathcal{H}^{\dag}(t')\mathcal{G}^{\dag}(t)dt'.
		\end{align*}
		Then for all $|t|\leq T$, $\mathcal{H}^{\dag}(t)$ is unitary and 
		$ \mathcal{G}^{\dag}(t)=(\mathcal{H}^{\dag}(t))^*.
		$
	\end{lemme}
	\begin{proof} 
		Note that $\mathcal{H}^{\dag}(0)=\mathrm{Id}$, and that $\mathcal{H}^{\dag}(t)$ solves the equation
		\[
		i\partial_t\mathcal{H}^{\dag}(t)=
		\chi'(t)\mathrm{Id}
		+
		\chi(\frac{t}{2T})\mathcal{A}(t)\mathcal{H}^{\dag}(t)-\frac{i}{2T}\chi'\big(\frac{t}{2T}\big)\int_0^t\mathcal{A}(t')\mathcal{H}^{\dag}(t')dt'\,.
		\]
		We observe that for $|t|\leq T$, $\mathcal{H}^{\dag}(t)$ solves the same equation as $\mathcal{H}(t)$ with the same initial data. By uniqueness of the linear equation, we have $\mathcal{H}^{\dag}(t)=\mathcal{H}(t)$ for $|t|\leq T$. The unitarity of $\mathcal{H}(t)$ was proved in Lemma \ref{lem:unitary}. For the $\mathcal{G}^{\dag}(t)$ part, as it solves the same integral equation
		\[
		\mathcal{G}^{\dag}(t)=\mathrm{Id}+i\int_0^t\mathcal{A}(t')\mathcal{H}(t')\mathcal{G}^{\dag}(t)dt'
		\]
		as $(\mathcal{H}(t))^*$
		for $|t|\leq T$, by uniqueness, we have $\mathcal{G}(t)^{\dag}=(\mathcal{H}^{\dag}(t))^{\ast}$.
	\end{proof}
	
	\begin{lemme}
		\label{lemme:law}
		Let $X\sim \mathscr{N}_{\R^n}(0 ; \operatorname{Id})$ be a random Gaussian vector on a probability space $(\Omega ; \mathcal{F}, \mathbb{P})$. Assume that $\mathcal{A}\subseteq \mathcal{F}$ is a sub $\sigma$-algebra, and that $X$ is independent of $\mathcal{A}$. Let $A\in \mathcal{O}(n)$ be a random orthogonal matrix that is $\mathcal{A}$-measurable. Then 
		\[
		\mathscr{L}(AX) = \mathscr{L}(X)\,.
		\]
	\end{lemme}
	\begin{proof} We compute the characteristic function of $AX$: given $\xi\in\R^n$, 
		\[
		\varphi_{AX}(\xi) = \mathbb{E}[e^{i\langle A(\omega)X , \xi \rangle} ] =  \mathbb{E}[e^{i\langle X , A^\ast(\omega)\xi \rangle} ] =  \mathbb{E}[\mathbb{E}[e^{i\langle X , A^\ast(\omega)\xi \rangle}\mid \mathcal{A} ]]\,.
		\]
		Since $A^\ast\xi$ is $\mathcal{A}$-measurable and since $X$ is independent of $\mathcal{A}$, we have that almost-surely in $\omega$, 
		\begin{multline*}
			\mathbb{E}[e^{i\langle X , A^\ast(\omega)\xi \rangle}\mid \mathcal{A} ] =  \varphi_X(A^\ast(\omega) \xi) = e^{-\frac 12\langle A^\ast(\omega)\xi , A^\ast(\omega)\xi \rangle} = e^{-\frac 12\langle AA^\ast(\omega) \xi , \xi \rangle} 
			= e^{-\frac 12 |\xi|^2 } = \varphi_X(\xi)\,.
		\end{multline*}
		Hence, 
		\[
		\varphi_{AX}(\xi) = \varphi_X(\xi)\,.
		\]
		Since this holds for every $\xi\in\R^n$, we conclude that $\mathscr{L}(AX) = \mathscr{L}(X) = \mathscr{N}_{\R^n}(0 , \operatorname{Id})$.
	\end{proof}

\end{document}